\documentclass{amsart}

\usepackage[usenames]{color}

\usepackage{xspace}

\usepackage{setspace}

\usepackage{hyperref}

\providecommand{\abs}[1]{\ensuremath{\left\lvert#1\right\rvert}\xspace}
\providecommand{\norm}[1]{\ensuremath{\left\lVert#1\right\rVert}\xspace}

\providecommand{\defin}[1]{{\em #1}}
\providecommand{\C}{\ensuremath{{\mathbb C}}\xspace}
\providecommand{\R}{\ensuremath{{\mathbb R}}\xspace}
\providecommand{\N}{\ensuremath{{\mathbb N}}\xspace}
\providecommand{\Z}{\ensuremath{{\mathbb Z}}\xspace}
\providecommand{\Q}{\ensuremath{{\mathbb Q}}\xspace}
\providecommand{\T}{\ensuremath{{\mathbb T}}\xspace}

\providecommand{\id}{\ensuremath{{\mathbf 1}}\xspace}
\providecommand{\idfunc}{\ensuremath{\mathbf{1}}}

\providecommand{\eqsp}{\quad}
\providecommand{\setsp}{\;}

\providecommand{\dm}{\ensuremath{\,d}}
\providecommand{\dmes}[1]{\ensuremath{\,{\mathrm d}#1}}

\DeclareMathOperator{\spec}{spec}
\DeclareMathOperator{\Span}{Span}
\DeclareMathOperator{\im}{Im}

\DeclareMathOperator{\supp}{supp}
\DeclareMathOperator{\dist}{dist}

\theoremstyle{plain}
\newtheorem{theorem}{Theorem}[section]
\newtheorem{prop}[theorem]{Proposition}
\newtheorem{lemma}[theorem]{Lemma}
\newtheorem{cor}[theorem]{Corollary}

\theoremstyle{definition}

\newtheorem{remark}[theorem]{Remark}
\newtheorem{example}[theorem]{Example}

\numberwithin{equation}{section}

\author{Yu. I. Lyubich}
\title[Cohomological equations]
{The cohomological equations in nonsmooth categories}

\date{}

\begin{document}


\bibliographystyle{amsplain}


\begin{abstract} 
The functional equation $\varphi(Fx) - \varphi(x) = \gamma(x)$ 
is considered in topological, measurable and related categories 
from the point of view of functional analysis and general 
theory of dynamical systems. The material is presented in the form of 
a self-contained survey. 
\end{abstract}

\maketitle

1. Introduction 
  
2. Existence of invariant measures

3. Discrete  systems 

4. Ces\'aro summation

5. Resolving functionals 

6. The Gottschalk - Hedlund Theorem (GHT)

7. The c.e. on topological groups

8. An almost periodic counterpart of GHT

9. Recurrence and ergodicity. Some applications 

10. The total attractor

11. The normal solvability in $C(X)$

12. Absence of measurable solutions

13. Summation of divergent  series

\section{Introduction }
\label{sec:intro}

This is an extended version of the author talk  
at the conference ``Operator Theory and Applications: Perspectives and Challenges'' 
(18.03-28.03, 2010, Jurata, Poland) supported by the project TODEQ.  
I am grateful to Professor Zemanek for the suggestion to write this paper.  
 
The cohomological equations (c.e.) and some of their generalizations naturally 
appear and play an important role in many areas of mathematics: 
group representations \cite{kirillov67}, \cite{moore80}, 
dynamical systems and ergodic theory \cite{anosov73}, \cite{gottschalk55}, 
\cite{katok01}, \cite{katok95},\cite{livsh72}, \cite{schmidt77}, stochastic processes  
\cite{furst60}, singularity theory \cite{arnold}, 
local classification of smooth mappings \cite{belitskii03}, summation of 
divergent series \cite{lyubich08-axit}, etc. We consider the c.e. 
as a subject of functional analysis. This approach originates from 
Browder's paper \cite{browder58}.

By definition, the {\em c.e.} is a linear functional equation
\begin{equation}
  \label{eq:1}
  \varphi(Fx) - \varphi(x) = \gamma(x), \eqsp x\in X, 
\end{equation}
where $F$ is a self-mapping of a nonempty set $X$, $\gamma$ is a given
real-  or complex-valued function on $X$, and $\varphi$ is an unknown
function. Note that the complex case in \eqref{eq:1} immediately 
reduces to the real one by separation the real and imaginary parts. For 
this reason we assume that all functions under consideration are real-valued, 
unless otherwise stated. The term {\em cohomological} is borrowed from a 
topological language where the {\em coboundary} means a function $\gamma$ 
of the form \eqref{eq:1}, c.f. e.g. \cite{katok01}. (In \cite{anosov73} and 
in some subsequent papers the term ``homological'' is used.) 

In physical language the set $X$ is the {\em space of states} of a {\em dynamical 
system} $(X,F)$ with the discrete {\em time variable} $n\in\N = \{0,1,2,\cdots\}$ and with 
the {\em evolutionary operator} $F$. By the standard definition of the latter, if 
$x\in X$ is the state of the system at a moment $n$ then $Fx$ is the state at the 
moment $n+1$. Thus, if $x$ is a state at the moment $n=0$ (the {\em inital state}) 
then the corresponding {\em trajectory} or {\em orbit} is 
\[
O_F(x)=\left(F^n x\right)_{n\geq 0},\eqsp x\in X. 
\]
In the c.e. \eqref{eq:1} the functions $\gamma$ and $\varphi$ relate 
respectively to the {\em input} and {\em ouput} of the system $(X,F)$. In these terms 
to solve the equation \eqref{eq:1} means to find out the response of the system 
to a given sequence of input signals. This problem may be especially nontrivial 
if a solution has to belong to a prescribed functional space.  

Since the equation \eqref{eq:1} is linear, it is naturally to consider it 
in a subspace $E$ of the space $\Phi(X)$ of all scalar 
functions on $X$. More generally, 
$\varphi$ should be find in $E$, while $\gamma$ is given in another subspace. 
In any case the subspaces under consideration have to be invariant for the 
{\em Koopman operator} $T_F\varphi=\varphi\circ F$.  Accordingly , the c.e. \eqref{eq:1} 
can be rewritten as 
\begin{equation}
  \label{eq:11} 
  T_F\varphi-\varphi=\gamma.  
\end{equation}

In any subspace $E$ the simplest c.e. is the {\em homogeneous} one, 
i.e such that $\gamma=0$. It is always solvable, 
its {\em trivial solution}  is $\varphi = 0$. If the function  $\id(x)\equiv 1$ belongs to 
$E$ then all constants are solutions. The general solution consists 
of the {\em invariant functions} $\psi\in E$, i.e. such  that $T_F\psi=\psi$ 
or, equivalently,$\psi\in\ker(T_F-I)$ where $I$ is the identity operator. 
Like any linear equation (algebraic, differential, etc.), if the nonhomogeneous 
c.e. \eqref{eq:1} has a solution $\varphi_0\in E$ then its general solution in $E$ 
is $\varphi_0+\psi$ where $\psi$ is the general solution to 
corresponding homogeneous equation.

For a dynamical system $(X,F)$ a subset $M\subset X$ is called {\em invariant} if 
$FM\subset M$. In this case the {\em subsystem} $(M,F|M)$ is well defined. 
If the complement $X\setminus M$ is also invariant then $M$ is called 
{\em completely invariant}. For this property it is necessary and sufficient 
that the indicator function of $M$ is invariant.

The underlying  subspace $E$ is usually determined by a structure on $X$ such as 
topology, measure, etc. We consider the following ``nonsmooth'' situaions: 

a) $X$ is an arbitrary set, $E=\Phi(X)$ (Sections \ref{sec:disy},  
\ref{sec:cesu}) or $E=B(X)$, the space of bounded functions (Section \ref{sec:sofu}); 
 
b) $X$ is a topological space, $E=C(X)$, the space of continuous functions,  
or $E=CB(X)= C(X)\cap B(X)$ (Sections \ref{sec:coso}, \ref{sec:also}, 
\ref{sec:noso}, \ref{sec:meso}); 
 
c) $X=G$, where $G$ is a commutative topological group, $E=C(G)$ 
or $E=AP(G)$, the space of almost periodic functions (Sections \ref{sec:togr},
\ref{sec:also}, \ref{sec:noso}, \ref{sec:meso});

d) $X$ is a set with a measure $\mu$, $E= M(X,\mu)$, the space of all measurable functions  
or $E=L_p(X,\mu)$ , $1\leq p\leq\infty$ (Sections \ref{sec:sofu}, \ref{sec:reer},
\ref{sec:meso}, \ref{sec:sods}). 

In each of these cases the mapping $F$ should be a morphism of the 
corresponding category. Thus, in the case a) $F$ can be any self-mapping of $X$ , while in 
the case b) it should be continuous,  
and then the dynamical system $(X,F)$ is called {\em topological}. 
This system is called {\em compact} or {\em discrete}
if such is the space $X$. Obviosly, a) can be considered as b) with 
discrete topology on $X$. 

A topological dynamical system $(X,F)$ (as well as $F$ itself) 
is called {\em minimal} if all orbits are dense. 
A weaker property is the {\em topological transitivity} meaning that there exists a 
dense orbit. For the minimality it is necessary and sufficient that 
any closed invariant subset is either $\emptyset$ or $X$. If $(X,F)$ is 
topological transitive then all continuous invariant functions are constants. 

In the case c) we only consider $F = \tau_g$ that is the shift by an element $g\in G$: 
$\tau_gx = gx$, $x\in G$. 

In the case d) $F$ has to be {\em measurable}, i.e. such that for every measurable 
set $M\subset X$ the preimage $F^{-1}M$ is measurable. The system $(X,F)$ 
is also called {\em measurable} in this case. We assume that  
$M=X$ is measurable and $0<\mu(X)\leq\infty$. If $\mu(X)<\infty$, i.e.  the measure $\mu$  
is {\em finite}, then one can normalize it in order to get $\mu(X)=1$. 
As a rule, we will deal with this case.
  
Later on the abbreviation ``a.e.'' means ``almost everywhere'' with respect to 
a given measure $\mu$. We assume that if a set $M$ is measurable and $\mu(M)=0$ 
then all $N\subset M$ are measurable and then $\mu(N)=0$ automatically.  

In the measurable situation it makes sense to consider the {\em solutions a.e.}  
to the equation \eqref{eq:1}. Every such a solution $\varphi(x)$ has to 
satisfy \eqref{eq:1} on an invariant measurable subset 
$Y_{\varphi}\subset X$ such that $\mu(X\setminus Y_{\varphi})=0$. 
As a rule, we mean this case when speaking about solutions  
in $M(X,\mu)$ or $L_p(X,\mu)$.  

The measure $\mu$ is called 
{\em invariant} if $\mu(F^{-1}M)=\mu(M)$ for every measurable set $M\subset X$. 
If $\mu$ is finite then it is invariant if and only if  
\begin{equation}
\label{eq:inme}
\int_X\phi(Fx)\dmes\mu = \int_X\phi(x)\dmes\mu,\eqsp\phi\in L_1(X,\mu). 
\end{equation}
The Krylov-Bogolyubov theorem 
\cite{krybog37} states that there exists an invariant regular Borel measure for any compact 
dynamical system $(X,F)$. In this case $\mu$ is finite and its invariance is equivalent 
to \eqref{eq:inme} with $\phi\in C(X)$. There are several proofs of this theorem 
\cite{katok95}, \cite{sinaietalergodictheory}, \cite{krybog37}, \cite{lyubich92-funcan}. 
In Section \ref{sec:exin} we give another proof based on a ``cohomological'' lemma. 

On any locally compact topological group there exists a 
regular Borel measure invariant for all left (for definiteness) shifts, the 
{\em Haar measure} \cite{haar33}, \cite{halmos50}. This measure is unique up to 
a multiplicative constant. This is finite if and only if the group is compact. 
In this case we assume that the Haar measure is normalized, thus unique. 
We denote it $\nu$ throughout the paper. 

A dynamical system $(X,F)$ with an invariant measure $\mu$ 
is called {\em ergodic} if for every completely invariant measurable
set $M$ either $\mu(M) = 0$ or $\mu(X\setminus M) = 0$. 
Also, $F$ and $\mu$ are called {\em ergodic} in this case. 
An invariant measure $\mu$ is ergodic if and only if all measurable 
invariant functions are constants a.e. If, in addition, such a measure is unique 
up to a constant factor then  $(X,F)$ is called {\em uniquely ergodic}.  

An obvious but very important consequence of \eqref{eq:inme} is the equality 
\begin{equation}
\label{eq:14}
  \int_X \gamma\dmes\mu=0 
\end{equation}
for every $\gamma\in L_1(X,\mu)$ such that the c.e. \eqref{eq:1}
is solvable in $L_1(X,\mu)$. 
We call \eqref{eq:14} a {\em trivial necessary condition} ({the \em TNC}, for brevity).  
In particular, if $X$ is a compact topological space then for the solvability 
of \eqref{eq:1} in $C(X)$ the TNC has to be valid for all invariant regular Borel 
measures $\mu$. 

Remarkably, {\em on any set $X$ with a finite invariant measure 
$\mu$ if $\gamma\in L_1(X,\mu)$ is such that \eqref{eq:1} has a measurable 
(maybe, nonintegrable!) solution then the TNC \eqref{eq:14} is fulfilled}.  
This is Theorem 1 from Anosov's paper \cite{anosov73}. We reproduce it 
in Section \ref{sec:reer}. The proof is based on 
the Birkhoff-Khinchin Individual Ergodic Theorem (IET) and the Poincar\'e 
Recurrence Theorem (RT). For the latter we give a ``cohomological'' proof.  

The simplest nonhomogeneous c.e. is the {\em Abel equation} (A.e.)
\begin{equation}
\label{eq:ae}
   \varphi(Fx) - \varphi(x) = 1,\eqsp x\in X, 
\end{equation}
coming back to \cite{abel}, p.p. 36-39. 
Since \eqref{eq:ae} implies $\varphi(F^nx) = \varphi(x)+n$ for all $n\geq 1$,  
{\em the A.e. has no bounded solutions}.  
A fortiori, there are no continuous solutions if $X$ is a compact 
topological space. In contrast, on a noncompact $X$ the A.e. may be 
solvable in $C(X)$. For example, we have  
$\varphi(x)=x$ if $Fx=x+1$ in $X=\N$ or $X=\R$. The A.e. on $\R$ is well studied   
\cite{belitskii03}, \cite{kuczma90}, \cite{szekeres10}. 
The relations of A.e. to other c.e. in 
$C(X)$ on noncompact $X$  were investigated in \cite{belitskii98a}, \cite{belitskii99}. 
Note that if $\mu$ is a finite invariant measure then \eqref{eq:14} shows that  
the A.e.has no solutions in $L_1(X,\mu)$.

A very interesting and deep example is the c.e. generated on the unit circle $\T$ 
by a rotation through the angle
$2\pi\alpha$, $0\leq \alpha<1$, with $\alpha$ irrational.
This equation is equivalent to
\begin{equation}
  \label{eq:7}
  f(x+\alpha) - f(x) = h(x),\eqsp x\in\R,
\end{equation}
where the functions $f$ and $h$ are 1-periodic. The latter was mentioned by
Hilbert in context of his
5\textsuperscript{th} problem (on analiticity of Euclidean topological
groups) as a functional equation which may be unsolvable
in analytic functions even for an analytic $h$ satisfying the TNC 
\begin{equation}
  \label{eq:8}
  \int_0^1 h(x)\dm x = 0.
\end{equation}

The equation \eqref{eq:7} can be immediately solved in formal
trigonometric series. Namely, if the Fourier decompositiion of $h$ is 
\begin{equation*}
h(x)\sim\sum_{n=-\infty}^{\infty} h_n e^{2\pi inx}
\end{equation*}
then \eqref{eq:8} means that $h_0=0$, and then for $\alpha$ irrational
the formal solution is 
\begin{equation}
\label{eq:smd}
\sum_{n\neq 0}\frac{h_{n}}{e^{2\pi in\alpha}-1}e^{2\pi inx}
\end{equation}
up to an additive constant. Wintner \cite{wintner} noted that for an analytic $h$ 
the series \eqref{eq:smd} determines an analytic solution  
if the number $\alpha$ is approximated by rationals slowly enough, while  
in the opposite case a phenomenon of ``small denominators'' appears:
there are no analytic solutions if $h$ is not a trigonometric polynomial and 
the approximation of $\alpha$ by rationals is fast enough.
For a further development of this approach see \cite{katok01}, \cite{rozh08} 
and the references therein.  

Anosov \cite{anosov73} mentioned that the flows on the torus $\T^2$ 
constructed by von Neumann \cite{neumann}, Kolmogorov \cite{kolmogorov53},
and Krygin \cite{krygin74} yield some continuous  1-periodic functions $h$ satisfying 
\eqref{eq:8} such that there are no measurable solutions to \eqref{eq:7}. Gordon
\cite{gordon75} constructed such a ``bad'' $h$ directly. In \cite{lyubich80} the existence 
of a ``bad'' $h$ was proven via the Closed Graph Theorem. 
This shows that the set of ``bad'' $h$'s is of second Baire's category 
in the space of continuous  1-periodic functions satisfying \eqref{eq:8}. 
In \cite{belitskii98} the method of \cite{lyubich80}  
was extended to the topologically transitive shifts of compact commutative
groups and then to all uniformly stable
dynamical systems on compact metric spaces, see Section \ref{sec:meso}. These 
dynamical systems closely relate to the almost periodic operator 
semigroups \cite{lyubich88}.

In \cite{anosov73} Anosov elaborated a subtle analytic technique
to construct for any irrational $\alpha$ a continuos 1-periodic function 
$h$ such that there exists a measurable but nonintegrable solution to \eqref{eq:7}. 
Using another method Kornfeld \cite{kornfeld76} extended this result to 
a class of compact dynamical systems $(X,F)$. It would be interesting to 
put it into the frameworks of functional analysis. 

Returning to the operator form \eqref{eq:11} of the c.e. \eqref{eq:1} note that 
its solvability in a $T_F$-invariant space $E$ means $\gamma\in\im(T_F-I)$.   
Moreover, if the operator $T_F-I$ is invertible then $\varphi = (T_F-I)^{-1}\gamma$ 
is a (unique) solution. However, this is not a case if $\id\in E$ since $\id\in\ker(T_F-I)$. 
Nevertheless, it is very useful to consider the formal solution 
\begin{equation}
  \label{eq:12} 
 -( \gamma + T_F\gamma + T_F^2\gamma+\cdots) 
\end{equation}
arising from the formal expansion 
\begin{equation*}
  I +T_F + T_F^2+\cdots,
\end{equation*}
of the operator $(I-T_F)^{-1}$. The series inside the parentheses in \eqref{eq:12} is actually  
\begin{equation}
  \label{eq:3}
  \gamma(x) + \gamma(Fx) +\cdots+ \gamma(F^n x)+\cdots ,\eqsp x\in X.
\end{equation}
We call it the {\em resolving series} for the c.e. \eqref{eq:1}.  
The behavior of its partial sums 
\begin{equation}
\label{eq:2'}
s_n(x)=\sum_{k=0}^n\gamma(F^kx), \eqsp n\geq 0, 
\end{equation}
plays a crucial role in the problem of solvability of the c.e. \eqref{eq:1}. 

Let us emphasize that the resolving series may diverge even for 
functions $\gamma$ such that the c.e. has a very good solution. For example, 
let in \eqref{eq:7} with irrational $\alpha$ the function $h$ be a 
nonzero trigonometrical polynomial 
\begin{equation*}
h(x)=\sum_{\abs{l}\leq m} h_l e^{2\pi ilx}
\end{equation*}
with $h_0=0$. Then the trigonometrical polynomial 
\begin{equation*}
f(x)=\sum_{\abs{l}\leq m, l\neq 0}\frac{h_{l}}{e^{2\pi il\alpha}-1}e^{2\pi ilx}
\end{equation*}
is a solution. On the other hand, we have  
\begin{equation*}
s_n(x)=\sum_{k=0}^n h(x+k\alpha) = 
\sum_{\abs{l}\leq m, l\neq 0}
h_l\frac{e^{2\pi i(n+1)l\alpha}-1}{e^{2\pi il\alpha}-1}e^{2\pi ilx}.
\end{equation*}
It is easy to see that the sequence $(s_n(x))$ diverges at every point $x$. 

If the resolving series \eqref{eq:3} converges or, at least, summable 
(say, by Ces\`aro method) for all $x\in X$ 
then the corresponding sum is a solution to the c.e. \eqref{eq:1}.
Also, the convergence or summabilty in a norm yields a solution if the operator 
$T_F$ is bounded. (For example, $\norm{T_F}\leq 1$ in $B(X)$ and in 
each $L_p(X,\mu)$ with an invariant measure $\mu$).
In this context a powerful tool is the Mean Ergodic Theorem (MET), 
see \cite{linsine83} and the references therein. Note that historically first version 
of the MET is due to von Neumann \cite{neumann32}. 

In \cite{lyubich08-axit} the IET was 
applied to solve a.e. the c.e. \eqref{eq:1} with $\gamma\in L_1(X,\mu)$ 
under the TNC \eqref{eq:14} where $\mu$ is a finite ergodic invariant measure.
This result was obtained in frameworks of an axiomatic theory of summation of 
divergent series \cite{lyubich92-funcan}, \cite{lyubich08-axit}. The latter  
is based on a modern form of the summation axioms 
introduced by Hardy (\cite{hardy49}, Section 1.3) and Kolmogorov \cite{kolmogorov25}. 
We partially reproduce our theory in Section \ref{sec:sods}. 
As an application we consider the c.e.
\begin{equation}
\label{eq:lyu} 
f(qx) - f(x) = h(x), \eqsp x\in\R,  
\end{equation}
in $2\pi$- periodic functions. This situation is quite different from that 
of \eqref{eq:7}. In \cite{lyubich08-axit} it was proven that 
{\em if $q\in\N$, $q\geq 2$, and $h$ is a trigonometric 
polynomial such that the ratios of its frequencies are not powers of $q$
then all solutions to \eqref{eq:lyu} are nonmeasurable}. 
Our techniques is a development of that which Zygmund 
applied to the case $q=2$, $h(x)=\cos x$, see \cite{zygmund59}, Chapter 5, Problem 26 . 

For $q=3$ and $h(x)=\sin x$ Kolmogorov \cite{kolmogorov25} formulated (without proof) 
the following conditional statement : {\em if the trigonometric series}  
\begin{equation}
\label{eq:col}
  \sin x + \sin 3x +\cdots + \sin 3^nx + \cdots, \eqsp x\in\R ,
\end{equation}
{\em is summable then one can efficiently construct a nonmeasurable function.} 
He evidently meant that this function can be produced by a summation under 
his axioms. Our Theorem \ref{thm:7.9} implies the existence of such a summation. 
Note that \eqref{eq:col} is the resolving series for the c.e. \eqref{eq:lyu} 
with Kolmogorov's data.  

According to the Hardy-Kolmogorov axioms the summations must be linear. Surprisingly,  
some nonlinear procedures can also be applied to solve the c.e. \eqref{eq:1} 
inspite of its linearity, see Sections \ref{sec:sofu} - \ref{sec:also}.
The most imporant example is the formula  
\begin{equation}
\label{eq:vasu}
\varphi(x)=\sup_{n}(-s_n(x)),\eqsp x\in X,
\end{equation}
for a bounded continuos solution in the case of minimal $(X,F)$ 
and $\gamma\in CB(X)$ such that the sums $s_n(x)$ are uniformly bounded.
Under these conditions for compact metric $X$ and invertible $F$
the existence of a continuous solution 
was established  by Gottshalk and Hedlund (\cite{gottschalk55}, Theorem 14.11).  
Browder \cite{browder58} extended this fundamental result to any Hausdorff $X$ 
and any continuous $F:X\rightarrow X$. 
Formula \eqref{eq:vasu} appeared in the proof given by Lin and Sine \cite{linsine83}. 
In our version of the latter (Theorem \ref{thm:6.0}) $X$ is an arbitrary 
topological space. (See \cite{mccutcheon99} for an alternative proof.) 
It is interesting that the nonlinearity of \eqref{eq:vasu} can be removed by subtraction 
a linear functional over \eqref{eq:vasu}, see Theorem \ref{cor:solin}. 
In this sense the nonlinearity is 1-dimensional.  

For the shift $\tau_g$ in a compact commutative group $G$ the minimality condition 
mentioned above can be reformulated in terms of characters.
This yields a new criterion of the solvability of the 
corresponding c.e. in $C(G)$, see Section \ref{sec:togr}. 
By the Bohr compactification this result extends to 
the almost periodic solutions on any commutative topological group $G$. 

In an almost periodic context the minimality is redundant (Section \ref{sec:also}). 
A general result of such a kind is Theorem \ref{thm:ceapf} proved by Schauder's 
Fixed Point Principle. Its main application is that the precompactnes of the 
set $\{s_n\}$ is a criterion of solvability of the c.e. \eqref{eq:1} in 
$C(X)$ for the uniformly stable system $(X,F)$ on a compact metric space $X$ 
(Theorem \ref{thm:6.0n}). This system can be not minimal, and then the boundedness 
of $\{s_n\}$ is not sufficient for the solvability in $C(X)$ (Example \ref{ex:ubwec}). 
Note that there is no minimality if $X$ is a convex compact in a Banach space. 

In accordance with a general terminology, 
the c.e. \eqref{eq:1} in a linear topological space $E$ is called 
{\em normally solvable} if $\im(T_F-I)$ is closed. If the space $E$ is complete 
metrizable then the absence of the normal solvability 
implies that $\im(T_F-I)$ is a subset of the first Baire's category in its closure. 
In this sense the unsolvability of the c.e. \eqref{eq:1} is typical if this equation 
is not normally solvable. 

For $E=C(X)$  with compact $X$ the normal solvability 
means the solvability in this space as long as the TNC \eqref{eq:14} is valid 
for all invariant regular Borel measures. It turns out this case is very rare: 
{\em if $(X,F)$ 
is a compact dynamical system then the c.e. \eqref{eq:1} is normally solvable
in $C(X)$ if and only if $F$ is preperiodic}, i.e. $F^{l+p} = F^l$
with a {\em preperiod} $l\geq 0$ and a {\em period} $p\geq 1$ \cite{belitskii98}. 
This result is represented in Section \ref{sec:noso}. The proof is based on    
the Dunford-Lin Uniform Ergodic Theorem (UET)\cite{dunford43}, \cite{lin74}. 

On the locally compact spaces the normal solvability was investigated in 
\cite{belitskii99}.

\section {Existence of invariant measures}
\label{sec:exin}

In general, a dynamical system has no ``good'' invariant measures. 
\begin{example}
\label{ex:noim}
Let $X=\N$ and let $Fn=n+1$, $n\in\N$. If a measure $\mu$ is 
invariant and all singletons are measurable then $\mu(\{n\})=\mu(\{n-1\})$ 
for $n\geq 1$, but $\mu(\{0\}) = \mu(\emptyset) = 0$. 
Hence, $\mu(\{n\})=0$  for all $n$, thus $\mu(X)=0$.
\qed
\end{example} 

The following fundamental theorem is due to Krylov and Bogolyubov \cite{krybog37}. 
\begin{theorem}\label{thm:3.1}
  For every compact dynamical system $(X,F)$ there exists an invariant  
regular Borel measure $\mu$ such that $\mu(X)=1$.
\end{theorem}
The proof below based on the following ``cohomological'' lemma. 
\begin{lemma}  
\label{lem:3.2}
Let $X$ be a set, and let $F$ be its self-mapping. 
In the space $B(X)$ the distance from the function $\idfunc$ 
to the subspace generated by the coboundaries is equal to 1.
\end{lemma}  
This is a quantitative version of the unsolvability of the Abel equation in $B(X)$.
\begin{proof}
The distance in question is the infimum over $\varphi\in B(X)$ of the functional   
  \[
    \delta[\varphi] =
    \sup_x\abs{1-\left(\varphi(Fx)-\varphi(x)\right)}. 
  \]
Since $\delta[0]=1$, it suffices to prove that $\delta[\varphi]\geq 1$. 

By definition,    
  \[
 \delta[\varphi]\geq\abs{1-\left(\varphi(Fx)-\varphi(x)\right)}
  \]
for all $\varphi\in B(X)$ and all $x\in X$. Changing $x$ to $F^kx$, $k\geq 0$, we obtain
  \[
\delta[\varphi]\geq\abs{1-\left(\varphi(F^{k+1}x)-\varphi(F^k x)\right)},  
  \]
whence
  \[
n\delta[\varphi]\geq\sum_{k=0}^{n-1}\abs{1-\left(\varphi(F^{k+1}x)-\varphi(F^k x)\right)}
\geq\abs{n-\left(\varphi(F^n x)-\varphi(x)\right)}, \eqsp n\geq 1,
  \]
by the triangle inequality. Therefore, 
  \[
    \delta[\varphi]\geq\abs{1 - \frac{\varphi(F^n x)-\varphi(x)}{n}}. 
  \]
Passing to the limit as $n\rightarrow \infty$ we get $\delta[\varphi]\geq 1$ 
since $\varphi\in B(X)$.
\end{proof}
\begin{proof}
[Proof of Theorem \ref{thm:3.1}]
  In the space $C(X)$ we consider the linear continuous 
functional $f$ such that $f[\idfunc]=1$, $\norm{f}=1$ and 
\begin{equation}
\label{eq:3.0}
    f[\varphi\circ F - \varphi]=0, \eqsp \varphi\in C(X).
\end{equation}
This functional exists by Lemma~\ref{lem:3.2} and the Hahn-Banach theorem. 
It easy to see that $f[\varphi]\geq 0$ 
for $\varphi\geq 0$. Indeed, let $0\leq \varphi(x)\leq 1$ for 
all $x\in X$. Then $\norm{\idfunc -\varphi}\leq 1$, hence, 
$1 - f[\varphi] = f[\idfunc-\varphi]\leq 1$.

Now a required measure $\mu$ comes from the Riesz representation 
  \[
    f[\varphi] = \int_X f\dmes\mu. 
  \]
The measure $\mu$ is invariant since \eqref{eq:3.0} is equivalent 
to \eqref{eq:inme}. In addition, $\mu(X) = f(\id) = 1$.  
\end{proof}
\begin{remark}
In $B(X)$ Lemma \ref{lem:3.2} yields a 
nonnegative invariant additive function  $\mu(M)$ 
on the set of all subsets $M\subset X$, $\mu(X)=1$. 
\qed
\end{remark}

For a compact group $G$ and a shift $\tau_gx = gx$ $(x,g\in G)$
Theorem \ref{thm:3.1} yields an invariant regular Borel measure $\mu_g$, $\mu_g(G)=1$.  
\begin{prop}
\label{prop:meha}
Let $G$ be a compact commutative group, and let $g\in G$ be such that 
the subsemigroup $[g] = \{g^n:n\geq 0\}$ is dense. Then the measure $\mu_g$ 
coincides with the Haar measure $\nu$ on $G$.  
\end{prop} 
\begin{proof}
We need to prove that measure $\mu_g$ is invariant for all shifts. 
For a function $\varphi\in C(G)$ let us consider the convolution     
  \begin{equation*}
      \psi(h)= \int_G\varphi(xh)\dmes\mu_g(x) ,\eqsp h\in G. 
  \end{equation*}
By commutativity of $G$ and $\tau_g$-invariance of $\mu_g$ we have  
  \begin{equation}
\label{eq:psic}
     \psi(gh)=  \int_G\varphi(xgh)\dmes\mu_g(x) = \int_G\varphi(gxh)\dmes\mu_g(x)=
\int_G\varphi(xh)\dmes\mu_g(x) =\psi(h), 
  \end{equation}
whence $\psi(yh) = \psi(h)$ for all $y\in [g]$. The function $\psi$ is continuous 
since $\varphi$ is uniformly continuous and the measure $\mu_g$ is finite. Since $[g]$ 
is dense, we obtain $\psi(yh) = \psi(h)$ for all $y\in G$. In particular, 
$\psi(y)=\psi(e)$ where $e$ is the unit of the group $G$. Thus,   
   \begin{equation*}
       \int_G\varphi(xy)\dmes\mu_g(x) = \int_G\varphi(x)\dmes\mu_g(x),\eqsp y\in G.    
  \end{equation*}
\end{proof}

We use this fact at the end of Section \ref{sec:prat}.

\section{Discrete  systems}
\label{sec:disy}

Here we consider the c.e. \eqref{eq:1} in the space $\Phi(X)$ of all 
functions on an arbitrary nonempty set $X$.
Although this situation is elementary, it is a prototype for more complicated 
cases and also a source of some useful general information.

It follows from (\ref{eq:1}) that 
\begin{equation*}
  \varphi(F^{k+1}x) - \varphi(F^kx) = \gamma(F^k x), \eqsp k\geq 0.
\end{equation*}
By summation over $k\in [0,n]$ we get 
\begin{equation}
  \label{eq:2}
  \varphi(F^{n+1} x)=\varphi(x)+ \sum_{k=0}^{n}\gamma(F^k x) = 
\varphi(x) + s_n(x),\eqsp n\geq 0.
\end{equation}
Formula \eqref{eq:2} shows that the restriction of the solution $\varphi$ 
to the orbit 
$O_F(x)$ is determined by the initial value $\varphi(x)$.
However, the asymptotic behavior of $\varphi(F^n x)$ as ${n\rightarrow\infty}$ 
reduces to that of $s_n(x)$. 

Assume that the resolving series \eqref{eq:3} converges for all $x\in X$ to a function $s(x)$.  
Then the function $\varphi_0(x)=-s(x)$ is a solution to (\ref{eq:1}) such that 
\begin{equation*}
\lim_{n\rightarrow\infty} \varphi_0(F^n x)=0 
\end{equation*}
since
\begin{equation*}
\varphi_0(F^n x)=-s(F^nx)=-\sum_{k=n}^{\infty}\gamma(F^kx). 
\end{equation*}
Furthemore, for any solution $\varphi$ there exists 
\begin{equation}
\label{eq:4}
  \hat\varphi(x) = \lim_{n\rightarrow\infty} \varphi(F^n x)= \varphi(x) + s(x)  
\end{equation}
by \eqref{eq:2}. The function $\hat\varphi(x)$ is invariant, and  
\begin{equation*}
\varphi(x)= -s(x)+\hat\varphi(x)=\varphi_0(x)+\hat \varphi(x).
\end{equation*}
This is an explicit representation of any solution $\varphi$  
as the sum of the particular solution $\varphi_0=-s$ and the solution 
$\hat\varphi$ to the corresponding homogeneous equation.  

Those $\gamma$'s for which the resolving series 
converges constitute a $T_F$-invariant subspace of the space $\Phi(X)$. 

The following theorem is a general solvability criterion for the c.e. \eqref{eq:1}.
\begin{theorem}  
  \label{thm:2.1}
The c.e. \eqref{eq:1} is solvable if and only if
for every integer $p\geq 1$ and every periodic point $x$ of the period $p$ the 
sum $s_{p-1}(x)$ vanishes, i.e.
\begin{equation}
    \label{eq:2.1}
\sum_{k=0}^{p-1}\gamma(F^k x) =0. 
  \end{equation}
\end{theorem}
\begin{proof}
  The necessity of~\eqref{eq:2.1} follows from~\eqref{eq:2}
  immediately. To prove the sufficiency we introduce an equivalence  
  relation on $X$ letting $x_1\sim x_2$ if there are some
  $m,n$ such that $F^mx_1 = F^nx_2$. Since the equivalence classes are 
  invariant, it suffices to solve~\eqref{eq:1} on 
  each class separately. 
  Thus, one can assume that all $x\in X$ are equivalent to an  
 $x_0$. We show that {\em for an arbitrary ``initial value'' $\varphi(x_0)$ there is    
 a unique solution $\varphi(x)$ to the c.e. \eqref{eq:1}}. 
 To this end we denote by $l_x$ the minimal $n$ such that $x=F^n x_0$. Let us 
consider two cases.  

1) $x\in O_F(x_0)$. Assume that the orbit $O_F(x_0)$ is infinite. Then the points 
$F^kx_0$ $(k=0,1,2,\cdots)$ are pairwise distinct. In particular, $Fx= F^{l_x+1}x_0$ 
is different from all $F^kx_0$ with $k\leq l_x$, hence $l_{Fx}=l_x+1$. Letting  
\begin{equation} 
\label{eq:2.2}
    \varphi(x) = \varphi(x_0) + \sum_{k=0}^{l_x-1}\gamma(F^k x_0),\eqsp x\in O_F(x_0), 
\end{equation}
we get \eqref{eq:1} since 
\begin{equation} 
\label{eq:2.21}
    \varphi(Fx) = \varphi(x_0) + \sum_{k=0}^{l_x}\gamma(F^k x_0)=\varphi(x)+\gamma(x).
\end{equation}

Now assume that $O_F(x_0)$ is finite. Then the point $x_0$ is preperiodic of a 
preperiod $l$ and a period $p$, i.e $F^{l+p}x_0=F^lx_0$. Let these values be minimal. 
Then the orbit reduces to 
\begin{equation}
 \label{eq:2.211}
\tilde{O}_F(x_0) = \{x_0, \cdots, F^lx_0, F^{l+1}x_0,\cdots, F^{l+p-1}x_0\}  
\end{equation} 
where all points are pairwise distinct. Obviosly, 
we have $l_{Fx}=l_x+1$ for all $x\in\tilde{O}_F(x_0)$, except for the last member 
in \eqref{eq:2.211}. With this exception we define $\varphi(x)$ by 
\eqref{eq:2.2} and get \eqref{eq:1} as before. In the 
exceptional case we have $l_{Fx}=l$, while $l_x=l+p-1$. Accordingly, if we let 
\begin{equation} 
\label{eq:2.20}
    \varphi(x) = \varphi(x_0) + \sum_{k=0}^{l+p-2}\gamma(F^k x_0), \eqsp x = F^{l+p-1}x_0,    
\end{equation}
and use 
\begin{equation} 
\label{eq:2.210}
    \varphi(Fx) = \varphi(x_0) + \sum_{k=0}^{l-1}\gamma(F^k x_0)   
\end{equation}
then 
\begin{equation} 
\label{eq:2.01}
    \varphi(x)- \varphi(Fx)= \sum_{k=l}^{l+p-2}\gamma(F^k x_0)=
\sum_{k=0}^{p-2}\gamma(F^k x_1)
\end{equation}
with $x_1=F^lx_0$. This is a periodic point of period $p$. By \eqref{eq:2.1} the last sum 
in \eqref{eq:2.01} is equal to $-\gamma(F^{p-1}x_1)= -\gamma(F^{l+p-1}x_0)=-\gamma(x)$, 
so \eqref{eq:2.01} reduces to \eqref{eq:1}.  

2) $x\notin O_F(x_0)$. Let $m_x$ be the minimal $n\geq 1$ such
  that $F^nx$ belongs to $O_F(x_0)$. Such $n$'s exist 
  since $x\sim x_0$. Obviously, $m_{Fx}=m_{x}-1$ if $m_x>1$. Since $\varphi(F^{m_x}x)$ 
  is already determined in the case 1), the definition  
  \begin{equation}
    \label{eq:2.3}
    \varphi(x)=\varphi(F^{m_x}x) - \sum_{k=0}^{m_x-1} \gamma(F^kx),\eqsp x\notin O_F(x_0),
  \end{equation}
is correct. Accordingly, 
  \begin{equation}
    \label{eq:2.4}
    \varphi(Fx) = \varphi(F^{m_x}x) - \sum_{k=1}^{m_x-1} \gamma(F^kx)=\varphi(x)+\gamma(x)
  \end{equation}  
if $m_x>1$. However, if $m_x=1$ then yet \eqref{eq:2.3} is equivalent to \eqref{eq:1}. 

Since all constructions above follow from \eqref{eq:2}, the solution is unique.
\end{proof}
\begin{remark}
\label{rem:2.02}
In the case of invertible $F$ the equivalence class of $x_0\in X$ is the 
two-sided orbit 
$\left( F^k x_0\right)_{k=-\infty}^\infty$. Hence, for $x\not\in O_F(x_0)$ 
we have $F^{m}x=x_0$ with a uniquely determined $m\geq 1$. Thus, $m_x=m$ 
in this case. 
\qed
\end{remark}
\begin{remark}
\label{rem:2.2}
It suffices to have \eqref{eq:2.1} for every periodic point $x$  and for {\em one} of 
its periods, say, $p$. Indeed, $p$ is a multiple of the smallest period $p_0$, say $p=dp_0$. 
Accordingly, $s_{p-1}(x) = ds_{p_0-1}(x)$, whence $s_{p_0-1}(x) = 0$, and finally,  
$s_{q-1}(x) = 0$ for every period $q$.
\qed
\end{remark}

The orbit of a periodic point $x$ is a {\em cycle} $Z_x$. 
The condition \eqref{eq:2.1} can be rewritten as 
\begin{equation}
  \label{eq:2.6} 
  \int_X\gamma\dmes\mu^{(x)} =0
\end{equation} 
where $\mu^{(x)}$ is the invariant measure concentrated and uniformly distributed on $Z_x$. 
Thus, \eqref{eq:2.1} is just the TNC \eqref{eq:14} corresponding to $\mu=\mu^{(x)}$. Theorem 
\eqref{thm:2.1} states that for any discrete dynamical system the conjunction of all these 
TNC's is sufficient for the solvability of the c.e \eqref{eq:1}. In topological terms, 
{\em a function  $\gamma$ is a coboundary 
if and only if its integrals of over all cycles vanish.} 
A remarkable case of sufficiency of these TNC's for the solvability 
of \eqref{eq:1} in a smooth category is Liv\v{s}ic's theorem 
\cite{livsh72} concerning H\"older's solutions.

In the completely formalized proof of Theorem \ref{thm:2.1} one has to 
consider a subset of $X$ which intersects every equivalence class
by a singleton. This subset is called a {\em transversal}. Its 
existence follows from the axiom of choice. {\em Given a transversal $X_0$, the 
restriction $\varphi\mapsto\varphi |X_0$ bijectively maps the set   
of solutions onto $\Phi(X_0)$.} 
The inverse mapping is described by the formulas ~\eqref{eq:2.2}, \eqref{eq:2.20}
and ~\eqref{eq:2.3}, where $x_0$ runs over $X_0$. For example, {\em every solution 
to the homogeneous c.e. is constant on each 
equivalence class, and, conversely, every such a function is a solution}. 

The following particular case of Theorem \ref{thm:2.1} merits to be mentioned.  
\begin{cor}
\label{cor:1}
If the mapping $F$ has no periodic points (thus, $X$ is infinite) 
then with any $\gamma$ the c.e.~\eqref{eq:1} is solvable.    
\end{cor}

For example, any irrational  rotation of the unit circle $\T$ has no periodic points.  
Therefore, {\em the equation \eqref{eq:7} with irrational $\alpha$ is solvable 
in 1-periodic functions for any 1-periodic $h$}. 
A more general example is the c.e. 
\[
\varphi(gx)- \varphi(x) = \gamma(x),\eqsp x\in G, 
\]
where $G$ is a group and $g\in G$ is of infinite order.  
\begin{cor}
\label{cor:10}
With $\gamma >0$ (in particular, in the case of Abel equation) 
the c.e.~\eqref{eq:1} is solvable 
if and only if the mapping $F$ has no periodic points.    
\end{cor}
\begin{proof}
The condition \eqref{eq:2.1} is not fulfilled if all $\gamma (F^kx) >0$.  
\end{proof}
The set $\Pi$ of periodic points of any mapping $F$ is invariant but, in general, it is not 
completely invariant. Indeed, the union of the preimages $F^{-l}\Pi$, $l\geq 0$,  
is the set of preperiodic points. The latter is already completely invariant. 
By Corollary \ref{cor:1} we have     
\begin{cor}
\label{cor:100}
On the complement of the set of preperiodic points the c.e. \eqref{eq:1} is solvable 
with any $\gamma$.  
\end{cor}
\begin{example}
\label{ex:pow}
Let $q$ be an integer, $q\geq 2$. The preperiodic points of the mapping 
$x\mapsto qx$ $({\rm mod}$ $2\pi)$, $0\leq x<2\pi$, are  
\[
x_{n,l,p} =\frac{2\pi n}{q^l(q^p-1)}\eqsp (\rm{mod} 2\pi), \eqsp 
\]
where $n,l,p\in\N$, $p\geq 1$. On the complement of this countable set the c.e. 
\eqref{eq:lyu} is solvable in $2\pi$-periodic functions for any $2\pi$-periodic $h$. 
\qed
\end{example}
 
\section{Ce\'saro summation}
\label{sec:cesu}

The reference to the axiom of choice in the proof of Theorem \ref{thm:2.1} 
makes it nonconstructive. However, under some 
natural conditions a solution to the c.e. \eqref{eq:1} can be found by 
``analytical'' tools, say by a summation of the resolving series. 
In this section we consider the Ce\'saro summation. By the standard definition, the   
{\em Ce\'saro sum} of the series \eqref{eq:3} is the limit $\sigma(x)$ 
as $N\rightarrow\infty$ of the arithmetic means 
\begin{equation}
\label{eq:gcs}  
\sigma_N(x)=\frac{1}{N+1}\sum_{n=0}^{N}s_n(x)=
\sum_{k=0}^{N}\left(1-\frac{k}{N+1}\right)\gamma(F^kx). 
\end{equation}
Denote by $C_{\gamma}$ the set of those $x\in X$ for which $\sigma(x)$ exists. 
\begin{prop}
\label{thm:cesre}
The set $C_{\gamma}$ is completely invariant, and the function 
$-\sigma(x)$ satisfies the c.e. \eqref{eq:1} for $x\in C_{\gamma}$. 
\end{prop}
\begin{proof}
Using the recurrence relation  
\begin{equation}
\label{eq:resn}
s_n(Fx)=s_{n+1}(x)-\gamma (x)
\end{equation}
we get 
\begin{equation}
\label{eq:sisi}
\sigma_{N-1}(Fx)=\frac{N+1}{N}(\sigma_{N}(x)-\gamma (x)),\eqsp N\geq 1.
\end{equation}
It remains to pass to the limit as $N\rightarrow\infty$.
\end{proof}

Now we introduce the set $E_{\varphi}$ of those $x\in X$ for which the 
limit $\tau(x)$ of the arithmetic means 
\begin{equation}
\label{eq:mera}
    \tau_N(x) = \frac{1}{N+1}\sum_{n=0}^{N}\varphi(F^n x)     
\end{equation}
exists as ${N\rightarrow\infty}$. Here $\varphi$ can be any function on $X$.  

\begin{prop}
\label{thm:cesreso}
The set $E_{\varphi}$ is completely invariant, and the function $\tau(x)$ is invariant.  
If $\varphi$ is a solution to the c.e. \eqref{eq:1} then 
$E_{\varphi} =  C_{\gamma}$, and on this set 
\begin{equation}
\label{eq:esig}
\sigma(x)=\tau(x)-\varphi(x).   
\end{equation}
\end{prop}
Thus, the set $E_{\varphi}$ is independent of the solution $\varphi$ to 
\eqref{eq:1} with a fixed $\gamma$. Formula \eqref{eq:esig} represents the 
solution $\varphi(x)$ to the c.e. \eqref{eq:1} restricted to $E_{\varphi}$ 
as the sum of the solution $-\sigma(x)$ and the solution $\tau(x)$ to the 
corresponding homogeneous equation. 
\begin{proof}
The first statement follows from the recurrence relation  
\begin{equation}
\label{eq:tata}
\tau_{N-1}(Fx)=\frac{N+1}{N}\tau_{N}(x)-\frac{1}{N}\varphi (x),\eqsp N\geq 1.
\end{equation}
The second one follows from the formula 
\begin{equation}
\label{eq:sita}
\sigma_N(x)=\frac{N+2}{N+1}\left(\tau_{N+1}(x)-\varphi(x)\right) 
\end{equation}
which, in turn, immediately follows from \eqref{eq:gcs}, \eqref{eq:mera} and \eqref{eq:2}.
\end{proof}

As an application we consider a preperiodic system $(X,F)$. 
The following theorem was formulated without proof in \cite{belitskii98}. 
\begin{theorem}
\label{thm:preper}
Let $F^{l+p}=F^l$, $l\geq 0$, $p\geq 1$. Then the c.e.~\eqref{eq:1} 
is solvable if and only if 
\begin{equation}
\label{eq:2.0}
   \sum_{k=0}^{p-1}\gamma(F^{k+l}x) =0, \eqsp x\in X. 
\end{equation}
Under this condition a solution is 
\begin{equation}
\label{eq:2.prep}
 \varphi(x)= -\sum_{k=0}^{l-1}\gamma (F^kx) + 
 \frac{1}{p}\sum_{k=1}^{p-1} k\gamma (F^{k+l}x). 
\end{equation}
\end{theorem}
\begin{proof}
The first statement follows from Theorem \ref{thm:2.1} since the condition \eqref{eq:2.0} 
is equivalent to \eqref{eq:2.1}. Indeed, for every $x\in X$ the point $F^lx$ is periodic 
of period $p$, thus \eqref{eq:2.1} implies \eqref{eq:2.0}. 
Conversely, let \eqref{eq:2.0} be fulfilled, and let $z\in X$ be a periodic 
point of a period $q$. Then $F^{nq}z=z$ for all $n\in \N$. With $n\geq l/q$ 
we get $z=F^lx$ where $x=F^{nq-l}z$. Obviously, $p$ is a period of $z$ and 
\begin{equation*}
   \sum_{k=0}^{p-1}\gamma(F^{k}z) =0.
\end{equation*}
Thus, \eqref{eq:2.1} is fulfilled. 

Now assuming \eqref{eq:2.0} we derive \eqref{eq:2.prep} 
via the Ces\'aro summation. Let $\varphi$ be a solution to 
\eqref{eq:1}. With $N>l$ and $q$ such that $l+pq\leq N<l+p(q+1)$ we have 
\begin{equation*}
\sum_{n=0}^{N}\varphi(F^{n}x) = \sum_{n=0}^{l-1}\varphi(F^{n}x) + 
 q\sum_{n=l}^{l+p-1}\varphi(F^{n}x) + 
\sum_{n=l}^{N-pq}\varphi(F^{n}x). 
\end{equation*} 
Hence, for all $x\in X$   
\begin{equation*} 
\tau(x)= \lim_{N\rightarrow\infty}\frac{1}{N+1}\sum_{n=0}^{N}\varphi(F^n x)=
\frac{1}{p}\sum_{n=l}^{l+p-1}\varphi(F^{n}x). 
\end{equation*}
Thus, $E_{\varphi}=X$, so $C_{\gamma}=X$ by Proposition \ref{thm:cesreso}, and 
$-\sigma(x)$ satisfies \eqref{eq:1} by Proposition \ref{thm:cesre}. 

Now we take $N=l-1+mp$  in \eqref{eq:gcs}, $m\rightarrow\infty$. 
This yields 
\begin{equation}
\label{eq:gcsi}  
-\sigma(x)=-\sum_{k=0}^{l-1}\gamma (F^kx) + \lim_{m\rightarrow\infty}
\sum_{k=l}^{l-1+mp}\left(\frac{k}{l+mp}-1\right)\gamma(F^kx).
\end{equation}
It remains to show that the limit in \eqref{eq:gcsi} is equal to the second summand in 
\eqref{eq:2.prep}. To this end we parametrize the index $k$ in the second sum in 
\eqref{eq:gcsi} as $k=l+i+pj$, where $0\leq i\leq p-1$, $0\leq j\leq m-1$. 
Since $F^{l+i+pj}=F^{l+i}$, this sum reduces to 
\begin{equation*}
\sum_{j=0}^{m-1}\sum_{i=0}^{p-1}\left(\frac{l+i+pj}{l+mp}-1\right)\gamma(F^{l + i}x)=
\Sigma_1 +\Sigma_2
\end{equation*}
where
\[
\Sigma_1 =\sum_{j=0}^{m-1}\left(\frac{l+pj}{l+mp}-1\right)\sum_{i=0}^{p-1}\gamma(F^{l + i}x) 
\]
and 
\[
\Sigma_2=\frac{m}{l+mp}\sum_{i=1}^{p-1}i\gamma(F^{l + i}x).
\]
However, $\Sigma_1 =0 $ by \eqref{eq:2.0}, while $\Sigma_2$ 
tends to the second summand in \eqref{eq:2.prep}.
\end{proof}
\begin{cor}
\label{cor:sopreper}
Let $F$ be preperiodic, and let \eqref{eq:2.0} be fulfilled for a function $\gamma$ 
from a $T_F$-invariant subspace $E\subset\Phi(X)$. Then the c.e. ~\eqref{eq:1} 
has a solution in $E$. 
\end{cor}
\begin{proof}
The required solution is that of \eqref{eq:2.prep}.
\end{proof}
By Theorem \ref{thm:2.1} this can be reformulated as follows. 
\begin{cor}
\label{cor:soprepe}
Let $F$ be preperiodic, and let $\gamma$ belongs to a $T_F$-invariant subspace 
$E\subset\Phi(X)$. 
Then if the c.e. \eqref{eq:1} is solvable in $\Phi(X)$ then it is solvable in $E$. 
\end{cor}

For example, if $(X,F)$ is a preperiodic topological dynamical system and $\gamma$ 
is continuous function on $X$ then there exists a continuous solution as long as 
a solution exists at all.

For $l=0$ the first summand in \eqref{eq:2.prep} vanises and we get the following  
\begin{cor}
\label{cor:per} 
Let $F$ be periodic of a period $p$, i.e. $F^p = I$. If the  
condition \eqref{eq:2.1} is fulfilled for all $x\in X$ then the function
\begin{equation}
\label{eq:per} 
 \varphi(x)= \frac{1}{p}\sum_{k=1}^{p-1} k\gamma (F^{k}x) = 
-\frac{1}{p}\sum_{n=0}^{p-1}s_n(x) 
\end{equation}
is a solution to the c.e.~\eqref{eq:1}.
\end{cor}
\begin{example}
\label{ex:rateq} 
Consider the equation ~\eqref{eq:7} with rational $\alpha = r/p$, $g.c.d.(r,p)=1$,  
and 1-periodic $h$. The shift $Fx=x+\alpha$ modulo 1 is periodic of period $p$. 
Under conditions  
\begin{equation}
\label{eq:9}
\sum_{k=0}^{p-1} h(x+k\alpha) = 0, \eqsp 0\leq x <1,
\end{equation}
the Corollary \ref{cor:per} yields the 1-periodic solution  
\begin{equation}
\label{eq:10}
f(x) = \frac{1}{p}\sum_{k=1}^{p-1}kh(x+k\alpha). 
\end{equation}

Note that \eqref{eq:9} is equivalent to its restriction to $0\leq x<1/p$. 
Indeed, for every $x\in$ [0,1) there exists a unique 
pair of integers $s,t$ such that the integral part of $px$ is equal to 
$rs+pt$ with $0\leq s\leq p-1$. Then in \eqref{eq:9} one can change $x$ to 
$z=x-s\alpha -t\in [0,1/p)$.
In fact, we see that the interval $[0,1/p)$ is a transveral.
Hence, the general solution to the corresponding 
homogeneous equation is $g(z)$ where $z=z(x)$ is as above and $g$ is an 
arbitrary function on $[0,1/p)$. 
\qed
\end{example}

\begin{remark}
\label{rem:indys}
If $(X,F)$ is an invertible dynamical system and $X_0$ is a transversal then 
\begin{equation}
\label{eq:trun}
X=\cup\{F^nX_0: n\in \Z\}.
\end{equation}
It turns out that {\em if $F$ has no periodic points then} $X_0$ {\em is nonmeasurable  
with respect to any finite invariant measure on $X$}. Indeed, the constituents in 
\eqref{eq:trun} are pairwise disjoint. If $X_0$ is measurable then 
$\mu(F^nX_0) = \mu(X_0)$ for all $n$. This implies $\mu(X_0)=0$ since $\mu(X)<\infty$, 
thus, $\mu(X)=0$. 

For example, {\em for the irrational rotation of the unit circle 
every transversal is nonmeasurable}, in contrast to the rational case.
This is a counterpart of the classical example of a nonmeasurable set, 
see e.g., \cite{gelbaum64}, Chapter 8, Example 11. 
\qed
\end{remark}

\section{Resolving functionals }
\label{sec:sofu}

In some important situations the c.e. \eqref{eq:1} 
can be solved by a ``nonlinear summation'' of the resolving series \eqref{eq:3}. 
\begin{theorem}
\label{thm:2.20}
Given a dynamical system $(X,F)$ and a function $\gamma$ on $X$, 
let $\Lambda$ be a set of scalar sequences $(\eta_n)_{n\geq 0}$ containing 
all $\left(-s_n(x)\right)$, $x\in X$, and let $\omega$ be a functional 
$\Lambda\rightarrow\R$. Assume that  
  
{\em a)} $\Lambda$ is  invariant with respect to the shift $n\mapsto n+1$ 
and to all translations $(\eta_n)\mapsto (\eta_n+\eta)$,  

and 

{\em b)} $\omega$ is shift invariant and translation covariant, i.e. 
\begin{equation*}
\omega[(\eta_{n+1})]=\omega[(\eta_{n})],\eqsp\omega[(\eta_n +\eta)]=\omega[(\eta_n)] +\eta. 
\end{equation*}

Then $\omega$ is a {\em resolving functional} for the c.e. \eqref{eq:1} in the sense that 
\begin{equation}
\label{eq:sol} 
\varphi(x)=\omega[(-s_n(x))]
\end{equation}
is a solution to the c.e. \eqref{eq:1}. 
\end{theorem}
\begin{proof}
We have 
\begin{equation*} 
\varphi(Fx)=\omega[(-s_n(Fx))]=\omega[(-s_{n+1}(x)+\gamma(x))]=
\omega[(-s_{n+1}(x))]+\gamma(x). 
\end{equation*}
by \eqref{eq:resn} and the translation covariance of $\omega$. Hence,  
\begin{equation}
\label{eq:mis} 
\varphi(Fx)-\varphi(x)=\gamma(x)-\{\omega[(-s_{n}(x))]- \omega[(-s_{n+1}(x))]\}. 
\end{equation}
This reduces to \eqref{eq:1} by the shift invariance of $\omega$.
\end{proof}
Obviously, for every finite family $\{\omega_k\}$ of resolving functionals all linear 
combinations $\sum_k\alpha_k\omega_k$ with $\sum_k\alpha_k  = 1$ are resolving functionals. 

The conditions a) and b) of Theorem \ref{thm:2.20} are fulfilled for  $\Lambda=B(\N)$, 
(i.e. for the space of all bounded sequences $(\eta_n)$) and    
\begin{equation}
\label{eq:om} 
\omega[(\eta_n)]=\overline{\lim}_{n}(\eta_n). 
\end{equation}
Other resolving functionals on this space are:
the lower limit , the upper  or lower limit of the arithmetic means of $(\eta_n)_{n=0}^N$, 
etc. All of them are nonlinear. A linear example is the Banach limit.
\begin{cor}
\label{cor:bos}
If the sequence $(s_n(x))$ is bounded at every $x\in X$ then the function 
\begin{equation}
\label{eq:ubc} 
\varphi(x) = \overline{\lim}_n[-s_n(x)] 
\end{equation}
is a solution to the c.e. \eqref{eq:1}. 
\end{cor}
\begin{cor}
\label{thm:bou}
The c.e. \eqref{eq:1} has a solution $\varphi\in B(X)$ if and only if the 
sums $s_n(x)$ ($n\geq 0$) are uniformly bounded. (In particular, 
$\gamma(x)\equiv s_0(x)$ is bounded.) Under this condition a bounded 
solution is \eqref{eq:ubc}. Moreover, the inequality 
\begin{equation}
\label{eq:cub} 
\frac{1}{2}\sup_n\norm{s_n}\leq\norm{\varphi}\leq\sup_n\norm{s_n}
\end{equation} 
holds.  
\end{cor}
\begin{proof}
The ``if'' part with the right-hand inequality in \eqref{eq:cub} follows from 
Theorem \ref{thm:2.20} with the upper limit in role of $\omega$.
The ``only if'' part with the left-hand inequality follows from \eqref{eq:2} 
for any bounded solution $\varphi$.
\end{proof}
\begin{remark}
\label{rem:3.3}
The `if'' part of Corollary \ref{thm:bou} can also be extracted from the proof of 
Theorem \ref{thm:2.1}. Indeed, let the sums $s_n(x)$ be uniformly bounded. 
For every $p\geq 1$ and every periodic point $x$ of period $p$ we have 
$s_{p-1}(x)= n^{-1}s_{np-1}(x)$ that yields \eqref{eq:2.1} as $n\rightarrow\infty$. 
Now if $x_0$ runs over a transversal $X_0$ and $\varphi(x_0)$ is bounded then we obtain 
a bounded solution $\varphi(x)$ by \eqref{eq:2.2}, \eqref{eq:2.20} and \eqref{eq:2.3}. 
\qed 
\end{remark}
\begin{theorem}
\label{thm:2.2}
For any measurable dynamical system $(X,F)$ the c.e. \eqref{eq:1}
has a bounded measurable solution if and only if the function $\gamma(x)$ 
is measurable and the sums $s_n(x)$ are uniformly bounded. In this case 
the function \eqref{eq:ubc} is such a solution. 
\end{theorem}
This is a version of Proposition 8 from \cite{linsine83} where 
a bounded solution is determined as the upper limit of the arithmetic means 
of $(-s_n(x))_{n=0}^N$ as $N\rightarrow\infty$. 
\begin{theorem}
\label{thm:limsup}
Let $(X,F)$ be a measurable dynamical system with a measure $\mu$, and let 
$\gamma\in L_{\infty}(X,\mu)$. If the sequence $(s_n)$ is $L_{\infty}$-bounded  
then the c.e. \eqref{eq:1} is solvable in this space. 
If the measure $\mu$ is invariant then the boundedness condition is necessary.
\end{theorem}
\begin{proof}
The necessity follows from \eqref{eq:2}. 
Now assume that in $L_{\infty}(X,\mu)$ we have $\sup_n\norm{s_n}<\infty$. 
However, $\norm{s_n}=\sup_{x\in X_n}\abs{s_n(x)}$ 
for a subset $X_n\subset X$ such that $\mu(X\setminus X_n)=0$. For 
the intersection $Y$ of all $X_n$ we have $\mu(X\setminus Y)=0$ and    
\begin{equation}
\label{eq:linf}
\sup_{x\in Y}\sup_{n}\abs{s_n(x)}=\sup_{n}\sup_{x\in Y}\abs{s_n(x)}
\leq \sup_{n}\norm{s_n}<\infty. 
\end{equation}

Now let us consider the set $Z=\{x\in X: \sup_{n}\abs{s_n(x)}<\infty\}$. 
Since $Y\subset Z\subset X$, we have $\mu(X\setminus Z)=0$ and $\mu(Z\setminus Y)=0$. 
The set $Z$ is $T_F$-invariant because of \eqref{eq:resn}. 
For $x\in Z$ a solution $\varphi(x)$ can be determined by \eqref{eq:ubc}, and 
$\sup_{x\in Y}\abs{\varphi(x)}<\infty$ by \eqref{eq:linf},thus $\varphi\in L_{\infty}(X,\mu)$.
\end{proof}

Theorem \ref{thm:limsup} generalizes Theorem 2 from \cite{browder58} where 
the $L_{p}$-solutions are constructed for $1<p<\infty$ and then an 
$L_{\infty}$-solution is obtained as $p\rightarrow\infty$. 
For other results concerning $L_{\infty}$-solutions see \cite{linsine83} and \cite{sator03}. 

The question arises: {\em is the supremum (or the infimum) a resolving functional? }
In general, {\em the answer is negative}. 
The reason is that these functionals are not shift invariant, so the 
subtrahend in formula \eqref{eq:mis} may not vanish. 
\begin{example}
\label{ex:nosup}
Let $X=\N$, and let $Fn=n+1$, $n\geq 0$. 
Then the c.e. \eqref{eq:1} turns into the simplest difference equation 
\begin{equation}
\label{eq:difeq}
\varphi(n+1)-\varphi(n) = \gamma(n). 
\end{equation}
Its general solution is 
\begin{equation}
\label{eq:pars}
\varphi(n)=\varphi(0)+\sum_{k=0}^{n-1}\gamma(k)= \varphi(0)+ s_{n-1}(0), \eqsp n\geq 1,  
\end{equation}
with an arbitrary $\varphi(0)$. The difference of any two solutions is a constant. 

From \eqref{eq:pars} it follows that if the sums $s_n(0)$ are bounded then all 
solutions are bounded, therefore, the sums $s_n(m)$ are 
uniformly bounded. However, the function $\varphi(m)=\sup_n[-s_n(m)]$ is not a 
solution, in general. Indeed, since  
\begin{equation*}
s_n(m) = s_{n+m}(0)-s_{m-1}(0), \eqsp m\geq 1,    
\end{equation*} 
we have 
\begin{equation}
\label{eq:sup0}
\varphi(m)=s_{m-1}(0) - \inf_{n\geq m}[s_n(0)].  
\end{equation}
This is a solution if and only if the infimum in 
\eqref{eq:sup0} is a constant since $s_{m-1}(0)$ is a solution. 
However, if, for instance, $\gamma(n)>0$ for all $n$ then the infimum 
is equal to $s_m(0)$ that increases along with $m$.
\qed
\end{example}
Nevertheless, {\em the supremum and infimum are resolving functionals 
in some special situations.} For definiteness we consider the supremum. Then
\begin{equation}
\label{eq:mon}
\sup_n(\eta_n)\geq\sup_n(\eta_{n+1}), \eqsp (\eta_n)\in B(\N),
\end{equation} 
instead of the shift invariance. On this base we establish    
an useful lemma concerning the {\em residual function}
\begin{equation}
\label{eq:refu}
\delta(x)=\gamma(x) + \varphi(x) - \varphi(Fx), \eqsp x\in X.   
\end{equation}
This one is defined for any pair of functions $\varphi ,\gamma$,  
and $\delta=0$ if and only if $\varphi$ is a solution to the c.e. \eqref{eq:1}.) 
\begin{lemma}
\label{lem:depo}
Let $s_n(x)$ be uniformly bounded, and let $\varphi(x)=\sup_n[-s_{n}(x)]$. Then 

1). The residual function $\delta(x)$ is nonnegative. 

2). The series 
\begin{equation}
\label{eq:dese} 
\delta(x)+\delta(Fx)+\cdots + \delta(F^n x)+\cdots
\end{equation}
converges for all $x$ to a bounded nonnegative function $\Delta(x)$.    
\end{lemma}
\begin{proof}
1). From \eqref{eq:refu} and \eqref{eq:mis} with $\omega\equiv\sup$ it follows that 
\begin{equation}
\label{eq:mise} 
\delta(x)=\sup_n[-s_{n}(x)]- \sup_n[-s_{n+1}(x)]\geq 0. 
\end{equation}

2). From \eqref{eq:refu} it follows that 
\begin{equation}
\label{eq:del}
\sum_{k=0}^{n}\delta(F^k x)= s_n(x)+\varphi(x)-\varphi(F^{n+1} x).
\end{equation}
The right-hand side of \eqref{eq:del} is uniformly bounded. Hence, $1)\Rightarrow 2)$ 
with the inequality $0\leq\Delta(x)\leq 3\sup_n\norm{s_n}$.  
\end{proof}
\begin{cor}
\label{cor:refun}
Under conditions of Lemma \ref{lem:depo} $\delta(F^nx)\rightarrow 0$ 
as $n\rightarrow\infty$. 
\end{cor}
Lemma \ref{lem:depo} allows us to construct a {\em regularization} of the resolving series 
\eqref{eq:3} in the case of its divergence. Namely, let us consider the difference 
\begin{equation}
\label{eq:diff}
d_N(x)=\sup_n[-s_n(F^{N+1}x)]- s_N(x).
\end{equation} 
\begin{theorem}
\label{thm:resol}
If $s_n(x)$ are uniformly bounded then the sequence $(d_N(x))$ is non-increasing and  
it tends to a solution to the c.e. \eqref{eq:1} as $N\rightarrow\infty$.
\end{theorem}
\begin{proof}
Let $\varphi(x)=\sup_n[-s_{n}(x)]$ as before. The function $\Delta(x)$ from Lemma 
\ref{lem:depo} satisfies the c.e. $\Delta(Fx)-\Delta(x)=-\delta(x)$. 
By \eqref{eq:refu} the difference $\hat{\varphi}=\varphi - \Delta$ 
is a solution to the c.e. \eqref{eq:1}. From \eqref{eq:del} it follows
that $d_N(x)$ monotonically tends to $\hat{\varphi}(x)$ from above. 
\end{proof}
If the resolving series converges to a function $s(x)$ for all $x\in X$ 
then the supremum in \eqref{eq:diff} tends to zero as $N\rightarrow\infty$, 
thus $d_N(x)$ monotonically tends to the solution $-s(x)$ from above.

Now we consider a situation where the supremum is a resolving functional.
\begin{theorem}
\label{thm:sup}
Let $(X,F)$ be a measurable dynamical system with a finite invariant measure $\mu$, 
and let $\gamma$ be a measurable function on $X$ such that the sums $s_n(x)$  
are uniformly bounded. Then the function  
\begin{equation}
\label{eq:sup} 
\varphi(x)=\sup_{n}[-s_n(x)]
\end{equation} 
is a bounded measurable solution to the c.e. (1.1) a.e.  
\end{theorem}
\begin{proof}
We only need to prove that $\varphi$ is a solution a.e.. By Lemma \ref{lem:depo} 
\begin{equation}
\label{eq:sude} 
\sum_{k=0}^{n} \delta(F^k x)\leq\Delta(x), \eqsp x\in X , 
\end{equation}
for all $n$. Therefore, 
\begin{equation*} 
(n+1)\int _X\delta\dmes\mu \leq \int _X\Delta\dmes\mu < \infty 
\end{equation*}
since the measure is invariant and finite and the function $\Delta$
is measurable and bounded. For $n\rightarrow\infty$ we get
\begin{equation*} 
\int _X\delta\dmes\mu\leq 0,  
\end{equation*}
so $\delta(x)= 0$ a.e. since $\delta(x)\geq 0$ for all $x\in X$. 
\end{proof}
\begin{remark}
\label{rem:tnc}
Under conditions of Theorem \ref{thm:sup} the TNC \eqref{eq:14} is fulfilled. 
To show this directly note that  
\begin{equation*} 
\int _X s_n(x)\dmes\mu=(n+1)\int _X\gamma\dmes\mu 
\end{equation*}
and pass to the limit as $n\rightarrow\infty$.
\qed
\end{remark}

The conditions of solvability in the space $M(X,\mu)$ of all measurable 
functions are more subtle. 
Note that if $\mu$ is a finite invariant measure and $\varphi$ is a measurable solution then 
for every $\varepsilon >0$ there exists $C>0$ such that 
$\mu\{x:\abs{\varphi(x)}>C\}<\varepsilon$,  
and then $\mu\{x:\abs{\varphi(F^{n+1}x)}>C\}<\varepsilon$. By \eqref{eq:2} 
we get $\mu\{x:\abs{s_n(x)}>2C\}<2\varepsilon$ for all $n$. In this sense 
the sums $s_n(x)$ are ``uniformly bounded in measure'' in the case 
of solvability in $M(X,\mu)$. For $\mu$ ergodic
the converse is due to Schmidt , see \cite{schmidt77}, p. 181. 

The following general criterion was established by Helson \cite{helson85} using the 
harmonic analysis of the corresponding {\em multiplicative cohomological 
equation} $\chi(Fx)/\chi(x)=\xi(x)$ in $\T$-valued functions. This kind of c.e.   
is important and interesting (see e.g., \cite{furst61}, \cite{katok01}), 
but we do not touch it in the present paper.   
\begin{theorem}
\label{thm:hel}
Let $(X,F)$ be an invertible measurable dynamical system with a finite 
invariant measure $\mu$, and let $\gamma\in M(X,\mu)$. Then the c.e. \eqref{eq:1}
has a solution $\varphi\in M(X,\mu)$ if and only if $s_{n}(x)$ is bounded
a.e. on a sequence of $n$ having positive upper density. 
\end{theorem}

Other criteria of solvability in $M(X,\mu)$ are due to Krzy\.zewski \cite{krz00} 
and Sato \cite{sato03}.

\section {The Gottschalk-Hedlund Theorem (GHT)}
\label {sec:coso}

In this section $(X,F)$ is a topological dynamical system. Accordingly, we 
look for the continuous solutions to the c.e. \eqref{eq:1} with continuous 
$\gamma(x)$. 
\begin{lemma}
\label{lem:un}
If $(X,F)$ is topologically transitive  
then every continuous solution $\varphi$ to the homogeneous equation 
\begin{equation}
\label{eq:hoco}
\varphi(Fx)-\varphi(x)=0, \eqsp x\in X,
\end{equation}
is a constant. 
\end{lemma}
\begin{proof}
The resriction of $\varphi$ to the closure of any orbit is a constant.   
\end{proof}

Further we focus on the minimal systems $(X,F)$. Note that if $X$ is an infinite 
Hausdorff (or, at least, $T_1$-) space and $F$ is minimal then $F$ has no periodic points,  
therefore, the c.e. \eqref{eq:1} is solvable with any $\gamma$. 
\begin{theorem}
\label{thm:6.0}
Let $(X,F)$ be minimal. 
The c.e.~\eqref{eq:1} has a solution $\varphi\in CB(X)$ if and only if 
$\gamma\in CB(X)$ and the sums $s_n(x)$
are uniformly bounded. Under these conditions the 
continuous bounded solution is 
\begin{equation}
\label{eq:supco} 
\varphi(x)=\sup_{n}[-s_n(x)]
\end{equation}
up to an arbitrary additive constant. 
\end{theorem}

This is a constructive version (c.f. \cite{linsine83}) of 
Browder's generalization \cite{browder58} of the GHT. The latter is a  
principal part of Theorem 14.11 from the book \cite{gottschalk55}.
In its original form the space 
$X$ is compact metric and $F$ is invertible, while in \cite{browder58} and 
\cite{linsine83} $F$ is any continuous self-mapping of a Hausdorff space $X$.  
In Theorem \ref{thm:6.0} $X$ is any topological space 
admitting a minimal $F$, c.f. \cite{mccutcheon99}. (In particular, $X$ must be separable). 
We preserve the name GHT for all these generalizations. Our proof of the GHT is 
a modification of that of \cite{linsine83}. Its the only nontrivial 
``if'' part consists of two statements: 
 
a){\em the function \eqref{eq:supco} is continuous};  

b){\em this is a solution to the c.e.~\eqref{eq:1}}. 

The statement a) is equivalent to that the oscillation 
$\Omega_{\varphi}(x)$ equals 0 for all $x$. 
Recall that for every bounded function $\phi$ its {\em oscillation} at a point $x$ is 
\begin{equation}
\label{eq:os}
\Omega_{\phi}(x)= \lim_U(\sup_{u\in U}\phi(u)-\inf_{u\in U}\phi(u))
= \inf_U(\sup_{u\in U}\phi(u)-\inf_{u\in U}\phi(u))\geq 0 , 
\end{equation}
where $U$ runs over the directed set of neighborhoods of $x$. The function 
$\Omega_{\phi}$ is bounded and upper semicontinuous. 

In general, a function $\psi$ is called {\em upper (lower) semicontinuous} 
if for every $\rho\in\R$ 
the set $\{x:\psi(x)<\rho\}$ (the set $\{x:\psi(x)>\rho\}$, respectively) is open. 
The continuous functions are just those which are upper and lower semicontinuous 
simultaneously. The supremum (the infimum) of any pointwise bounded family of continuous 
functions is lower (upper, respectively) semicontinuous.  

\begin{lemma}
\label{lem:6.0}
If an upper (lower) bounded function $\phi$ is lower (upper) semicontinuous 
then $\inf\Omega_{\phi}=0$.
\end{lemma}
\begin{proof}
For definitenes, let $\phi$ be upper bounded. With $c=\sup\phi$ 
and $\varepsilon >0$ let $x$ be such that $\phi(x)>c-\varepsilon$.
Since $\phi$ is lower semicontinuous, 
there exists a neighborhood $U$ of $x$ such that $\phi(u)>c-\varepsilon$ for 
$u\in U$. Since $\phi(u)\leq c$ for all $u$, we get $\Omega_{\phi}(x)\leq \varepsilon$.   
\end{proof}
\begin{cor}
\label{cor:om}
Under conditions of Lemma \ref{lem:6.0}, if $\Omega_{\phi}$ is a constant then 
the function $\phi$ is continuous. 
\end{cor}

After this preparation we can proceed to the proof of Theorem \ref{thm:6.0}.
\begin{proof}
To prove a) we note that  
\begin{equation*}
 \varphi(x)=\max\{-\gamma(x),\sup_{n}[-s_{n+1}(x)]\}, 
\end{equation*}
whence 
\begin{equation*}
\varphi(x)+\gamma(x) = \max\{0, \sup_{n}[\gamma(x)-s_{n+1}(x)]\}. 
\end{equation*}
By \eqref{eq:resn} we obtain 
\begin{equation}
\label{eq:cep}
\varphi(x)+\gamma(x)=\max\{0, \sup_{n}[-s_n(Fx)]\}=\max\{0, \varphi(Fx)\}= 
(\varphi\circ F)_{+}(x).
\end{equation}
From \eqref{eq:cep} it follows that $\Omega_{\varphi}(x)=\Omega_{(\varphi\circ F)_{+}}(x)$
since $\gamma$ is continuous. However, $\Omega_{\phi_{+}}(x)\leq\Omega_{\phi}(x)$ 
for any bounded function $\phi$. Thus, 
$\Omega_{\varphi}(x)\leq\Omega_{\varphi\circ F}(x)$. In turn,  
$\Omega_{\varphi\circ F}(x)\leq\Omega_{\varphi}(Fx)$ 
by continuity of $F$. As a result, $\Omega_{\varphi}(x)\leq \Omega_{\varphi}(Fx)$ and then  
\begin{equation}
\label{eq:inin}
\Omega_{\varphi}(x)\leq \Omega_{\varphi}(F^nx),\eqsp n\geq 1.
\end{equation}
Since the orbit $O_F(x)$ is dense and the function $\Omega_{\varphi}$
is upper semicontinuous, the inequality \eqref{eq:inin} yields 
$\Omega_{\varphi}(x)\leq \Omega_{\varphi}(y)$ for all $y\in X$.
By alternation $x\leftrightarrow y$ we obtain the equality  
$\Omega_{\varphi}(x)=\Omega_{\varphi}(y)$ for all $x,y$, i.e. 
$\Omega_{\varphi}$ is a constant. Therefore, $\Omega_{\varphi}=0$ 
by Corollary \ref{cor:om}, i.e. $\varphi$ is continuous.

Now it remains to prove that a)$\Rightarrow$ b). We have to show that 
$\delta = 0$ where $\delta$ is the residual function defined by \eqref{eq:refu} 
with $\varphi$ from Lemma \ref{lem:depo}. By this lemma the function $\delta$ 
is nonnegative, and it is continuous by a).  
Suppose to the contrary 
that $\delta(x)>0$ at a point $x$. Then for any $\varepsilon >0$ there exists a 
neighborhood $W$ of $x$ such that $\delta(w)\geq\varepsilon$ when $w\in W$. 
On the other hand, by Corollary \ref{cor:refun} there is $m\in\N$ such that 
$\delta (F^nx)<\varepsilon$ for all $n\geq m$. Since the orbit of $F^mx$ is dense, 
there is $n\geq m$ such that $F^nx\in W$, so $\delta (F^nx)\geq\varepsilon$, 
a contradiction.   
\end{proof}
\begin{remark}
\label{rem:antik}
In \cite{mccutcheon99} the proof of the GHT is given 
in the form b)$\&$ [b)$\Rightarrow$a)] (in our notation). The main argument for b) 
is that from the minimality of $(X,F)$ it follows that 
\begin{equation}
\label{eq:ins} 
\sup_{n}[-s_{n+1}(x)] = \sup_{n}[-s_{n}(x)], \eqsp x\in X, 
\end{equation}
i.e. the shift invariance of the supremum on the set of sequences $(s_n(x))$.   
\qed
\end{remark}
\begin{remark}
\label{rem:sign}
{\em The solution \eqref{eq:supco} in Theorem \ref{thm:6.0} is nonnegative.}  
Indeed, from \eqref{eq:cep} and \eqref{eq:1} we see that 
$\varphi(Fx)=\varphi_{+}(Fx)$, i.e  $\varphi(Fx)\geq 0$ for all $x$. 
This inequality implies $\varphi\geq 0$, since the image of $F$ is dense 
and $\varphi$ is continuous. 
\qed
\end{remark}
\begin{remark}[\cite{linsine83}]
\label{rem:op}
{\em For minimal $(X,F)$ and continuous $\gamma$ if the sums $s_n(x)$
are bounded at a point $x_0$ then they are uniformly bounded.} Indeed, 
let $r_{n,m}(x) = s_n(x)-s_m(x)$, $n>m\geq -1$, $s_{-1}(x)=0$, and let 
$\sup_n\abs{s_n(x_0)} = C<\infty$. Then the subset 
$M=\{x\in X:\sup_{n,m}\abs{r_{n,m}(x)}\leq 2C \}$ is nonempty (since $x_0\in M$), 
closed (since $\gamma$ and $F$ are continuous) and invariant (since 
$r_{n,m}(Fx) = r_{n+1,m+1}(x)$). Therefore, $M = X$ by minimality. Since 
$r_{n,-1}(x) = s_n(x)$, we get $\sup_n\abs{s_n(x)}\leq 2C$ for all $x\in X$.
\qed
\end{remark}

The minimality is essential for all conclusions in Theorem \ref{thm:6.0}.
\begin{example}
\label{ex:bla}
The system $(\N,F)$ from Example \ref{ex:nosup} is topologically transitive 
but not minimal. Indeed, since the topological space $\N$ is discrete, 
we have $\overline{O_F(0)}=\N$ and $\overline{O_F(1)}=\N\setminus\{1\}$. 
The c.e. \eqref{eq:difeq} is solvable in $B(\N)=CB(\N)$ if and only 
if the sums $s_n(0)$ are bounded, however, the sup-formula fails 
if $\gamma_n > 0$. 
\qed 
\end{example}
\begin{example}
\label{ex:last}
Consider the one-point compactification 
$\overline{\N} = \N\cup\{\infty\}$, $\overline{F}(\infty)=\infty$, 
$\overline{\gamma}(\infty) = 0$. The dynamical system $(\overline{X},\overline{F})$ is    
topologically transitive but not minimal for the reason like that in Example \ref{ex:bla}. 
If the resolving series 
\begin{equation}
\label{eq:rs}
\gamma(0)+\gamma(1)+\cdots +\gamma(n)+\cdots
\end{equation}
diverges but the sums $s_n(0)$ are bounded then there are no continuous solutions. 
Actually, in this case all solutions are bounded but discontinuous at infinity.
\qed
\end{example}
Nevertheless, the minimality can be replaced by some other assumptions. 
In \cite{schwartz96} the construction \eqref{eq:ubc} was used to prove a counterpart  
of the GHT such that $X$ is a compact metric space but, 
instead of the minimality, $F$ is uniquely ergodic and the invariant 
measure $\mu$ is such that $\mu(U)>0$ for all open $U\subset X$. 
In Section \ref{sec:also} we consider another important class of dynamical systems 
where the minimality is not needed due to an assumption on $s_n(x)$ 
stronger than in Theorem \ref{thm:6.0}.
   
Let us emphasize that the class of underlying spaces for the minimal systems 
is very special. For example, if such a space is locally compact then it is compact 
\cite{gottschalk46}. In particular,     
all minimal discrete dynamical system $(X,F)$ are finite cycles: 
$X=Z_x=\{x, Fx, \cdots , F^{p-1}x\}$ for an $x$ such that $F^{p}x = x$. 
Indeed, let $O_F(x)$ be a dense 
orbit. Then $X=O_F(x)$ since $X$ is discrete. Therefore, $X$ has to be finite, 
otherwise, the orbit $O_F(Fx)$ is not dense. Hence, the point $x$ 
is preperiodc, thus it is periodic by minimality. 

For any topological dynamical system $(X,F)$ one can consider in the space 
$CB(X)$ the linear (nonclosed, in general) subspace 
\[
E_F=\{\gamma:\norm{\gamma}_F \equiv\sup_{n}\norm{s_n}<\infty\}.
\]
Obviosly, $\norm{\gamma}_F\geq\norm{\gamma}$. With this stronger norm 
$E_F$ is a Banach space. According to \eqref{eq:11} and \eqref{eq:2} 
we have $\im(T_F-I)\subset E_F$. Theorem \ref{thm:6.0} states that  
if $(X,F)$ is minimal then $\im(T_F-I)=E_F$. Hence, {\em in this case 
the image $\im(T_F-I)$  is a Banach space with respect to the norm $\norm{.}_F$}.  
As a consequence, we have 
\begin{prop}
\label{prop:cloq}
Let $(X,F)$ be minimal. Then  $\im(T_F-I)$ is closed in $CB(X)$ 
if and only if the norms $\norm{\gamma}_F$ and $\norm{\gamma}$ are equivalent on 
this subspace.
\end{prop}
Theorem \ref{thm:6.0} determines a mapping $R:E_F\rightarrow CB(X)$ such that  
$\varphi = R\gamma$ is a solution to the c.e. \eqref{eq:1}. This equation 
is linear but the mapping $R$ is not linear! However, this collision can be  
removed by a small perturbation of $R$.
\begin{theorem}
\label{cor:solin}
Let $f$ be a linear functional on $CB(X)$ such that $f[\id]=1$. 
Under conditions of Theorem \ref{thm:6.0} the formula 
\begin{equation}
\label{eq:soli}    
\varphi_0(x)=\sup_{n}[-s_n(x)] - f[\sup_{n}[-s_n(.)]]
\end{equation}
determines a continuous bounded solution to the c.e. \eqref{eq:1} such that 
the mapping $\varphi_0= R_0\gamma$ is linear. In addition, if $\norm{f}=1$ 
then 
\begin{equation}
\label{eq:oli}    
\frac{1}{2}\norm{\gamma}_{F}\leq\norm{R_0\gamma}\leq\norm{\gamma}_{F}.
\end{equation}
Thus, $R_0$  is a linear topological isomorphism between $\norm{.}_F$-normed  
space $\im(T_F-I)$  and the closed hyperplane $\ker f\subset CB(X)$.
\end{theorem}
\begin{proof}
In \eqref{eq:soli} the first summand is a solution to \eqref{eq:1} in $CB(X)$, 
while the second summand is a constant. Hence, $\varphi_0$ is a solution in $CB(X)$  
as well. Moreover, by \eqref{eq:soli} the mapping $R_0:E_F\rightarrow CB(X)$ is 
such that $\im R_0\subset\ker f$. This allows us to consider $R_0$ as a mapping  
$E_F\rightarrow\ker f$. Let us also consider the linear mapping 
$D = T_F - I:CB(X)\rightarrow E_F$. This is a left inverse to $R$ by definition of 
the latter. 

Now let $D_0=D|\ker f$. The mapping $D_0$ is injective. Indeed, $\ker D =  \Span\{\id\}$ 
by Lemma \ref{lem:un}, so $\ker D_0= \Span\{\id\}\cap\ker f = 0$
since $f[\id]=1$. On the other hand, $D_0$ is surjective since 
$D_0R_0=DR = I|E_F$. Eventually, $D_0$ is bijective, and $R_0$ is its inverse mapping. 
Hence, $R_0$ is linear.   

Now assume $\norm{f}=1$. Then the functional $f$ is nonnegative. In view of   
Remark \ref{rem:sign} formula \eqref{eq:soli} yields the inequality 
\[
- f[\sup_{n}[-s_n(.)]]\leq\varphi_0(x)\leq\sup_{n}[-s_n(x)],    
\]
and the right-hand part of \eqref{eq:oli} follows. The left-hand part
follows from \eqref{eq:2}.
\end{proof}
\begin{remark}
\label{rem:ud}
The solution \eqref{eq:soli} is uniquely determined by the additional 
condition $f[\varphi_0] = 0$. 
\qed
\end{remark}
\begin{example}
\label{ex:ol}
Let $X$ be a compact space with a regular Borel measure $\mu$ such that $\mu(X)=1$. 
Then formula \eqref{eq:soli} can be realized as 
\begin{equation}
\label{eq:psoli}    
\varphi_0(x)=\sup_{n}[-s_n(x)] - \int_X\sup_{n}[-s_n(.)]\dmes\mu.
\end{equation}  
This solution is uniquely determined by the condition 
\[
\int_X\varphi_0\dmes\mu = 0.
\]
In particular, for any $x_0\in X$ we can take the corresponding Dirac measure  
and get 
\begin{equation}
\label{eq:psol}    
\varphi_0(x)=\sup_{n}[-s_n(x)] - \sup_{n}[-s_n(x_0)].
\end{equation}  
This formula is applicable to any topological space $X$ and yields the unique solution 
such that $\varphi_0(x_0)=0$.
\qed
\end{example}

Now let us compare Theorem \ref{thm:6.0} to Corollary \ref{thm:bou}. 
\begin{cor}
\label{cor:cbc}
Let $(X,F)$ be minimal, and let $\gamma\in C(X)$.  
Then if the c.e.~\eqref{eq:1} is solvable in $B(X)$ then it is solvable in $CB(X)$. 
\end{cor}
Actually, the bounded solution \eqref{eq:ubc} 
coincides with $\varphi$ from Theorem \ref{thm:6.0} under conditions of the latter.  
This follows from \eqref{eq:ins} as $m\rightarrow\infty$.

\section {The c.e. on topological groups}
\label {sec:togr}

Let us start with some general remarks. Let $G$ be a topological group. 
For any $g\in G$ the subsemigroup 
\[
[g] = \{g^n:n\in\N\} 
\]
is the $\tau_g$-orbit of the unit $e$. If this is dense (thus,  
the dynamical system $(G,\tau_g)$ is topologically transitive)  
then this system is minimal. Indeed, for any $x\in G$ 
its orbit is $[g]x$, hence the closure of the latter is $Gx=G$.    

The two-sided orbit of $e$ is the subgroup 
\[
<g> = \{g^n:n\in\Z\}.
\]   
A group $G$ such that $<g>$ is dense for a $g\in G$ is called 
{\em monothetic}, and the element $g$ is called its {\em generator}. 
For example, {\em the unit circle $\T$ is monothetic}: $\T =<\zeta>$ 
where $\zeta = e^{2\pi i\alpha}$ and $\alpha$  is any irrational number. 
Moreover, $\T =[\zeta]$, thus the dynamical system $(\T,\tau_{ \zeta})$ is 
minimal.    

In general, if $[g]$ is dense in a group $G$ then such is $<g>$, 
thus $G$ is monothetic. The converse is not true. For example,   
the discrete additive group $\Z$ is monothetic since 
 $\Z=<1>$, while $[z]=\{nz:n\in\N\}\neq\Z$. 
However, if $G$ is compact monothetic then 
$[g]$ is dense for any generator $g$, see  Lemma \ref{lem:gs} below. 

Obviously, every monothetic group is commutative and separable. On the other hand, 
every connected compact separable commutative group is monothetic \cite{halmos42}.    

In this section we systematically use the classical character theory. 
For the reader convenience let us recall some basic definitions and results. 
For more detail see e.g. \cite{hewittross63}, \cite{lyubichbook88}, \cite{weil40}. 

Let $G$ be a commutative topological group, and let $e$ be its unit.
A continuous function $\chi:G\rightarrow\T$
is called a {\em character of} $G$ if $\chi(gh)=\chi(g)\chi(h)$ for all $g,h\in G$.  
In particular, $\chi(e)=1$ and $\chi(g^{-1})=\chi(g)^{-1}=\overline{\chi(g)}$, 
where the bar means the complex conjugation. The characters constitute a commutative 
group $G'$ with respect the pointwise multiplication. Its unit is $\id$, and it 
may happen that $G'=\{\id\}$, inspite of $G\neq\{e\}$. For example, if $G$ is the additive 
group of a linear topological space without nonzero continuous linear functionals 
(say, $G=L_p(0,1)$, $0<p<1$) then $G'=\{\id\}$. In contrast, if $G$ is  
locally compact and $G'=\{\id\}$ then $G=\{e\}$.
A natural topology on $G'$ is discrete being determined by the inclusion  
$G\subset CB(G)$. Indeed, let $\chi_1, \chi_2 \in G'$, 
and let $\chi_1\neq\chi_2$. Then for every $g\in G$
we have $\abs{\chi_1(g)-\chi_2(g)}= \abs{\chi(g)-1}$ where 
$\chi(g)=\chi_1(g)\overline{\chi_2(g)}\neq\id$. 
If $g$ is such that $\chi(g)\neq 1$ then 
$\abs{\chi(g^n)-1}=\abs{\chi(g)^n-1}\geq \sqrt 2$ 
for some $n\in\N$. A fortiori, $\norm{\chi_1-\chi_2}\geq\sqrt 2$. 
 
If $G$ is locally compact then the classical Pontrjagin-van Kampen duality theory
prescribes to endow $G'$ with the compact-open topology and get the 
{\em dual group} $G^*$. This one is  
locally compact but not discrete, except for the case of compact $G$. 
Accordingly, although $G^*$ coincides with $G'$ algebraically but 
they are different topological groups if (and only if) $G$ is not compact. 
The Duality Theorem states that topological groups $G^{**}$ and $G$ 
can be identified by the topological isomorphism $G\rightarrow G^{**}$
which is as $g\mapsto g^{**}$ such that $g^{**}(\chi) = \chi(g)$.

In the rest of this section the group $G$ is compact, unless otherwise stated. 
\begin{lemma}
\label{lem:gs}
The closures $\overline{[g]}$ and $\overline{<g>}$ coincide for every $g\in G$. 
\end{lemma}
\begin{proof}
It suffices to show that $g^{-1}\in\overline{[g]}$. 
Suppose to the contrary. Then the set  
$U= G\setminus \overline{[g]} $ is a neighborhood of $g^{-1}$ separating 
this element from $\overline{[g]}$. Hence, $V=gU$ is a neighborhood of $e$ 
such that $g^r\notin V$ for all $r\geq 1$. Consider the intersection 
\begin{equation}
\label{eq:attr}
\Omega = \bigcap_{n\in\N}\overline{\{g^k:k\geq n\}}.
\end{equation}
This is not empty since $G$ is compact. Let $x\in\Omega$, and let $W$ 
be a neighborhood of $x$ such that $WW^{-1}\subset V$. By definition of $\Omega$ 
there exists $n\in\N$ such that $g^n\in W$, and there is $m>n$  
such that $g^m\in W$. Then $g^{r}\in V$ for $r=m-n$, a contradiction. 
\end{proof}
\begin{cor}
\label{cor:bodeg}
The closure $\overline{[g]}$ of the semigroup $[g]$ is a group. 
\end{cor}
\begin{cor}
\label{cor:bode}
The sets $[g]$ and $<g>$ are dense or not dense simultaneously.  
\end{cor}

\begin{lemma}
\label{lem:gred}
The semigroup $[g]$ is dense (thus, the group $G$ is monothetic, generated by $g$) 
if and only if $\id$ is the only character $\chi$ such that $\chi(g)=1$.
\end{lemma}
\begin{proof}
If $\chi(g)=1$ then  $\chi(h)=1$ for $h\in \overline{[g]}$. Therefore, 
$\chi = \id$ if $[g]$ is dense.
Now let $[g]$ be not dense. By Corollary \ref{cor:bode} the quotient group 
$\Gamma =  G/\overline{<g>}$ is nontrivial. The natural epimorphism 
$p:G\rightarrow\Gamma$ is continuous with respect to the standard topology on $\Gamma$. 
Hence, the group $\Gamma$ is compact. If $\xi$ is a nonunity character of $\Gamma$ 
then $\chi = \xi\circ p$ is a character of $G$ such that $\chi(g)=1$ but $\chi\neq\id$.
\end{proof}
\begin{example}
\label{ex:tor}
The characters of the $d$-dimensional torus 
\[
\T^d = \{g = (e^{2\pi i\alpha_k}): 0\leq\alpha_k<2\pi, 1\leq k\leq d \} 
\] 
are  
\[
\chi_{n_1,\cdots ,n_d}(g) = e^{2\pi i\sum_{k=1}^d n_k \alpha_k}, 
\eqsp (n_1,\cdots ,n_d)\in\Z^d .
\] 
Hence, 
\[
\chi(g)\neq 1\Leftrightarrow\sum_{k=1}^d n_k\alpha_k \notin\Z.
\]
Therefore, $\chi(g)\neq 1$ for all $\chi\neq\id$ if and only if 
the real numbers $1, \alpha_1,\cdots, \alpha_d$ are linearly 
independent over the field  $\Q$ of rational numbers. By Lemma \ref{lem:gred} 
this condition is necessary and sufficient for the density of the subsemigroup 
$[g]$ in $\T^d$. This is the famous Kronecker theorem. As a consequence, 
{\em the group $\T^d$ is monothetic}. 
Note that for $d=1$ the Kronecker condition just means that $\alpha_1$ is irrational. 
\qed
\end{example}

Now we immediately obtain 
\begin{theorem}
\label{thm:6.00}
Let $G$ be a compact commutative group, and let $g\in G$ be such that $\chi(g)\neq 1$ 
for all nonunity characters of $G$. Then the c.e. 
\begin{equation}
\label{eq:gama}
\varphi(gx)-\varphi(x)=\gamma(x), \eqsp x\in G,
\end{equation}
has a solution $\varphi\in C(G)$ if and only if $\gamma\in C(G)$ and the sums 
\begin{equation*}
s_n(x)=\sum_{k=0}^n\gamma(g^kx),\eqsp n\geq 0, 
\end{equation*}
are uniformly bounded. Under these conditions with real $\gamma$ the continuous 
solution is 
\begin{equation}
\label{eq:supcogr} 
\varphi(x)=\sup_{n}[-s_n(x)]  
\end{equation}
up to an arbitrary additive constant.
\end{theorem} 
\begin{proof}
By Lemma \ref{lem:gred} the system $(G,\tau_g)$ is minimal. Hence, 
Theorem \ref{thm:6.0} is applicable.
\end{proof}
\begin{remark}
\label{rem:st}
 According to Remark \ref{rem:op} it suffices to require that the set of sums
\begin{equation*}
s_n(e)=\sum_{k=0}^n\gamma(g^k),\eqsp n\geq 0, 
\end{equation*}
is bounded. 
\end{remark}
\begin{cor}
\label{cor:eqir}
The equation \eqref{eq:7} with continuous 1-periodic $h$ and irrational 
$\alpha$ has a continuous 1-periodic solution $f$ if and only if the sums 
\[
s_n(x) = \sum_{k=0}^n h(x+k\alpha), \eqsp 0\leq x<1, 
\]  
are uniformly bounded. 
\end{cor}
By Remark \ref{rem:st} {\em it sufficient that these sums are bounded at $x=0$}. 

The nonlinear construction \eqref{eq:supcogr} can be transformed into a linear one 
according to Theorem \ref{cor:solin}. For instance, such is  
\begin{equation}
\label{eq:dif}   
\varphi_0(x) =\sup_{n}[-s_n(x)] - \int_G\sup_{n}[-s_n(.)]\dmes\nu 
\end{equation}
where $\nu$ is the Haar measure, $\nu(G)=1$, see Example \ref{ex:ol}. This solution is 
determined by the condition 
\begin{equation}
\label{eq:ac}
\int_G\varphi_0\dmes\nu =0. 
\end{equation}

Now let us analyze the c.e. \eqref{eq:gama} by means of harmonic analysis in   
the complex space $L_2(G,\nu)$. In this space 
the set $G^*$ of the characters is orthonormal and complete. 
Thus, for every $\psi\in L_2(G,\nu)$ we have the {\em Fourier decomposition}    
 \begin{equation}
 \label{eq:fode}
\psi\sim\sum_{\chi\in G^*} c_{\chi}[\psi]\chi 
\end{equation}
where 
 \begin{equation}
 \label{eq:fodeco}
 c_{\chi}[\psi]=\int_{G}\psi\overline{\chi}\dmes\nu 
\end{equation}
are the {\em Fourier cofficients}. 
The set of those $\chi\in G^*$ for which $c_{\chi}[\psi]\neq 0$  
is at most countable. This set is called the {\em spectrum} of $\psi$ and 
denoted by $\spec\psi$. After ommiting of the vanishing summands in 
\eqref{eq:fode} and an arbitrary ordering of the rest we get 
the {\em reduced Fourier series} 
 \begin{equation}
 \label{eq:foder}
\psi=\sum_{\chi\in\spec\psi} c_{\chi}[\psi]\chi 
\end{equation}
convergent to $\psi$ in $L_2$-norm. As a consequence, $\spec\psi =\emptyset$  
if and only if $\psi=0$. If $\psi\in C(G)$ then its Fourier series can 
be divergent in $C(G)$ but $\psi$ is the uniform limit of a sequence of 
linear combinations of the characters from $\spec\psi$. 

%
The member corresponding to the unity character in \eqref{eq:fode} is the constant  
 \begin{equation}
 \label{eq:tric}
c_{\id}[\psi]=\int_{G}\psi\dmes\nu   
\end{equation} 
according to the equality 
 \begin{equation}
 \label{eq:trin}
\int_{G}\chi\dmes\nu =0, \eqsp \chi\in G^*\setminus\{\id\}.   
\end{equation} 
Therefore, 
\[
\int_{G}\psi\dmes\nu =0 
\]
if and only if $\id\notin\spec\psi$.
 
The Fourier coefficients of the shifted function $\psi_g(x)=\psi(gx)$ are  
 \begin{equation}
 \label{eq:coshi}
c_{\chi}[\psi_g]=\chi(g)c_{\chi}[\psi]. 
\end{equation}
Indeed, 
\[
\int_{G}\psi(gx)\overline{\chi(x)}\dmes\nu = 
\chi(g)\int_{G}\psi(x)\overline{\chi(x)}\dmes\nu
\]
since the measure $\nu$ is invariant and $\overline{\chi(g^{-1}x)} = \chi(g)\overline{\chi(x)}$.
By \eqref{eq:coshi} the Fourier image of the c.e. \eqref{eq:gama} is
 \begin{equation}
 \label{eq:cef}
(\chi(g)-1)c_{\chi}[\varphi] = c_{\chi}[\gamma],\eqsp \chi\in G^*.  
\end{equation} 
For $\chi=\id$ this yields the TNC  $c_{\id}[\gamma]= 0$, i.e. $\id\notin\spec\gamma$. 
Also, from \eqref{eq:cef} it follows that $\chi(g)\neq 1$ if $\chi\in\spec\gamma$. 
Thus, the Fourier decomposition of the solution \eqref{eq:dif} is 
\begin{equation}
\label{eq:fos}
\varphi_0=\sum_{\chi\in\spec\gamma}\frac {c_{\chi}[\gamma]}{\chi(g)-1}\chi. 
\end{equation}
Obviously, $\spec\varphi_0=\spec\gamma$, while  $\spec\varphi=\spec\gamma\cup\{\id\}$   
for all solutions $\varphi\neq\varphi_0$. 

An interesting consequence of the formula  \eqref{eq:fos} is the inequality  
\begin{equation}
\label{eq:par}
\sum_{\chi\in\spec\gamma}\abs{\frac {c_{\chi}[\gamma]}{\chi(g)-1}}^2 \leq 
\sup_{n}\norm{s_n}^2
\end{equation}
which is valid for all continuous coboundaries. This follows by changing the $\sup$-norm 
on the left of \eqref{eq:oli} to the $L_2$-norm which, in turn, can be found 
from \eqref{eq:fos} by the Parseval equality.

After these preliminaries we can generalize Theorem \ref{thm:6.00} as follows.  
For any $g\in G$  we consider its {\em annihilator} $g^{\perp} = \{\chi\in G^*:\chi(g) = 1)\}$.  
It is a subgroup of the group $G^*$. In Theorem \ref{thm:6.00} we actually assume   
$g^{\perp} = \{\id\}$. In order to relax this condition we introduce the subgroup 
$H_{\gamma}\subset G^*$ generated by the subset $\spec\gamma$. (For $\gamma = 0$ 
we let $H_{\gamma}=\{\id\}$.) 
\begin{theorem}
\label{thm:6.000}
Theorem \ref{thm:6.00} remains true under the condition
\begin{equation}
\label{eq:nga}
g^{\perp}\cap H_{\gamma}= \{\id\} 
\end{equation}
instead of $g^{\perp} = \{\id\}$.
\end{theorem}
\begin{proof}
The embedding $H_{\gamma}\subset G^{*}$ induces (by duality) 
a continuous epimorphism $j:G\rightarrow H_{\gamma}^*$. 
The group $H_{\gamma}^*$ is compact since such is $G$.  
Every $\theta\in C(H_{\gamma}^*)$ can be lifted to $G$ as $\theta\circ j$. 
This is a linear isometric mapping  $C(H_{\gamma}^*)\rightarrow C(G)$.  
Its image is the uniform closure $\overline{\Span H_{\gamma}}$ 
since $C(H_{\gamma}^*)=\overline{\Span{H_{\gamma}^{**}}}$.  

Now in $C(H_{\gamma}^*)$ we consider the c.e. 
\begin{equation}
\label{eq:gamah}
\theta(hz)-\theta(z)=\beta(z),\eqsp z\in H_{\gamma}^*, 
\end{equation}
where $h = jg\in H_{\gamma}^*$ and  $\beta\in C(H_{\gamma}^*)$ is such that 
$\beta(jx)=\gamma(x)$, $x\in G$. The only character $\xi$ of the group 
$H_{\gamma}^*$ such that $\xi(h)=1$ is $\xi=\id$. Indeed, $\chi = \xi\circ j$ is 
a character of $G$ such that $\chi(g)=1$, i.e. $\chi\in g^{\perp}$. On the other 
hand, $\chi\in H_{\gamma}$. By \eqref{eq:nga} we get $\chi = \id$. Therefore, $\xi=\id$ 
since $j$ is surjective.   

Thus, we have $h^{\perp} = \id$. Furthermore, 
\begin{equation*}
\hat{s}_n(z)\equiv\sum_{k=0}^n\beta(h^kz)= \sum_{k=0}^n\beta(j(g^kx)) = 
\sum_{k=0}^n\gamma(g^kx) = s_n(x)
\end{equation*}
for $z = jx$, $x\in G$. Since these sums are uniformly bounded,  
the c.e. \eqref{eq:gamah} has a continuos solution $\theta$ 
by Theorem \ref{thm:6.00}. Then $\varphi(x) = \theta (jx)$ is a continuous 
solution to the c.e. \eqref{eq:gama}. With real $\gamma$ one can take   
\begin{equation}
\label{eq:supc} 
\theta(z)=\sup_{n}[-\hat{s}_n(z)] = \sup_{n}[-s_n(x)].
\end{equation}
\end{proof}

\begin{remark}
\label{rem:con}
With $\gamma =0$ the general continuous solution to \eqref{eq:gama} is
$\overline{\Span g^{\perp}}$. This subspace of $C(G)$ consists of 
constants if and only if $g^{\perp} =  \{\id\}$.
\qed
\end{remark}

Now let $G$ be an arbitrary commutative topological group. An {\em almost 
periodic function} ({\em a.p.f.}) on $G$ is a complex-valued function 
$\psi\in CB(G)$ such that the set $\{\tau_g\psi:g\in G\}$ of its shifts  
is precompact in $CB(G)$. The a.p.f. constitute a subspace $AP(G)\subset CB(G)$. 
All characters of $G$ are almost periodic, so $G'\subset AP(G)$. Moreover, 
the linear span of $G'$ is dense in $AP(G)$. This Approximation Theorem is 
central in the theory of a.p.f., see e.g. \cite{weil40} and the references therein. 
On the additive group $\R$ the theory of a.p.f. was created 
by Bohr~\cite{bohr25}, \cite{bohr} 
and Bochner \cite{bochner27}. The characters on $\R$ are $\chi_{\lambda}(x)=e^{i\lambda x}$ 
$(x,\lambda\in\R)$, so in this case the Approximation Theorem means that the a.p.f. on $\R$ 
are just the uniform limits of the linear combinations of these exponents. 
In particular, all periodic functions on $\R$ are almost periodic. In this case 
the Approximation Theorem turns into the Weierstrass theorem from the classical 
Fourier analysis. 

On any compact group $G$ all continuous functions are almost periodic, i.e. 
$AP(G) = C(G)$ in this case. Remarkably, the general case reduces to this one as 
follows. (See e.g. \cite{lyubichbook88} for more detail.) 

Let $K$ be a compact group such that there exists   
a continuous homomorphism $j:G\rightarrow K$ then for every  
$\psi\in C(K)$ the function $\psi\circ j$ is an a.p.f. on $G$. 
In particular, one can take the {\em Bohr compact} 
$K = bG\equiv G'^*$ and $(jg)(\chi) = \chi(g)$, $g\in G$ , $\chi\in G'$. 
Although this $j$ is not surjective, its image is dense. 
It turns out that this construction yields all a.p.f. on $G$. 
The natural extension of $j$ to $AP(G)$ is a bijective linear isometry 
$AP(G)\rightarrow C(bG)$. This allows one to translate the harmonic analysis 
from $C(bG)$ to $AP(G)$. In particular, the {\em spectrum} of $\varphi\in AP(G)$  
is defined as the spectrum of the corresponding $\psi\in C(bG)$, 
the Fourier coefficients of $\varphi$ are defined as the corresponding ones for $\psi$, etc.  

By the way, in the case $G=\R$ the Fourier coefficients can be introduced 
with no reference to the Bohr compactification, namely, 
 \begin{equation}
\label{eq:focor}
 c^{(\lambda)}[\varphi]=\lim_{a\rightarrow\infty}\frac{1}{2a}
\int_{-a}^a \varphi(x)e^{-i\lambda x}\dmes x.  
 \end{equation}
We call the set of those $\lambda\in\R$ 
for which $c^{(\lambda)}[\varphi]\neq 0$  the {\em frequence spectrum} of 
$\varphi\in AP(\R)$. If $\varphi$ is periodic, say, 1-periodic,  
then its frecuence spectrum is a subset of $2\pi\Z$ and 
 \begin{equation*}
 c^{(2\pi n)}[\varphi] =c_{n}[\varphi]\equiv 
\int_0^1 \varphi(x)e^{-2\pi inx}\dmes x, \eqsp n\in\Z. 
 \end{equation*}

\begin{theorem}
\label{thm:apso}
Let $G$ be a commutative topological group, and let $\gamma\in AP(G)$. Assume that  
the condition \eqref{eq:nga} is fulfilled for the subgroups $g^{\perp}$ and $H_{\gamma}$ 
of $G'$ defined as before.
Then the c.e.  
\begin{equation}
\label{eq:gamma}
\varphi(gx)-\varphi(x)=\gamma(x), \eqsp x\in G,
\end{equation}
has a solution $\varphi\in AP(G)$ if and only if the sums 
\begin{equation*}
s_n(x)=\sum_{k=0}^n\gamma(g^kx),\eqsp n\geq 0, 
\end{equation*}
are uniformly bounded. Under this condition and with real $\gamma$ 
an a.p. solution is 
\begin{equation}
\label{eq:apfsol}
\varphi(x)=\sup_n[- s_n(x)].    
\end{equation}
\end{theorem} 
\begin{proof}
One can identify $G'$ with $(bG)^*$. Thus, Theorem \ref{thm:6.000} is applicable 
to that c.e. in $C(bG)$ which corresponds to \eqref{eq:gamma}. 
\end{proof}
\begin{remark}
\label{rem:conap}
According to Remark \ref{rem:con} the a.p. solution to \eqref{eq:gamma} is
determined up to an arbitrary summand $\omega\in\overline{\Span g^{\perp}}$. 
The latter means that $\spec\omega\subset g^{\perp}$.
\qed
\end{remark}

Let us apply the Theorem \ref{thm:apso} to the c.e. 
\begin{equation}
\label{eq:gmar}
\varphi(x+\alpha)-\varphi(x)=\gamma(x), \eqsp x\in\R,\eqsp \alpha\in\R\setminus\{0\},
\end{equation}
in $AP(\R)$. Note that in this case the TNC \eqref{eq:8} should be changed to 
 \begin{equation}
 \label{eq:tnchm1}
\lim_{a\rightarrow\infty}\frac{1}{2a}\int_{-a}^a \gamma(x)\dmes x = 0    
 \end{equation}
according to \eqref{eq:focor}.
\begin{cor}
\label{thm:aprs}
Let $\gamma\in AP(\R)$, and let the frequence spectrum of $\gamma$ be $\Lambda$. 
If $2\pi/\alpha$ does not belong to the rational linear span of $\Lambda$ then 
the c.e. \eqref{eq:gmar} has an a.p. solution if and only if the sums  
\begin{equation*}
s_n(x)=\sum_{k=0}^n \gamma(x+k\alpha)
\end{equation*}
are uniformly bounded. Under this condition and with real $\gamma$ 
an a.p. solution is 
\begin{equation}
\label{eq:apfsolr}
\varphi(x)=\sup_n[- s_n(x)]. 
\end{equation}
\end{cor}
\begin{proof}
We have 
\[ 
\alpha^{\perp} = \{\chi_{\lambda}: e^{i\lambda\alpha} = 1\} = 
\{e^{i\lambda x}: \lambda = 2\pi n/\alpha, n\in\Z\},
\] 
while $H_{\gamma} = \{\chi_{\lambda}: \lambda\in\Delta_{\Lambda}\}$  
where $\Delta_{\Lambda}$ is the subgroup of $\R$ generated by the subset $\Lambda$.
The assumption in Corollary \ref{thm:aprs} just means that 
$\alpha^{\perp}\cap H_{\gamma} = \{\id\}$. 
\end{proof}
Under conditions of Corollary \ref{thm:aprs} the a.p. solution is determined up to an 
arbitrary continuous $\alpha$-periodic summand. Indeed, an a.p.f. $\omega$ 
with $\spec\omega\subset\alpha^{\perp}$ is $\alpha$-periodic, and vice versa. 

\section { An almost periodic counterpart of GHT}
\label {sec:also}

In the GHT the minimality condition is essential, see Example \ref{ex:last}. 
However, in an almost periodic context this condition does not appear. 
Let us start this topic with some known general definitions, c.f. \cite{lyubichbook88}.

Let $T$ be a bounded linear operator in a Banach space $B$. A vector $v\in B$ 
is called {\em almost periodic} ({\em a.p.}) if its orbit $(T^nv)_{n\geq 0}$ 
is precompact. The set of all a.p. vectors is denoted by $AP(B,T)$. This is a linear 
subspace  of $B$ containing all eigenvectors of $T$ with eigenvalues 
$\lambda$ such that $\abs{\lambda}\leq 1$. 
The operator $T$ is called {\em a.p.} if $AP(B,T)=B$. By the Banach-Steinhaus 
theorem  every a.p. operator is  power bounded, i.e. $\sup_{n\geq 0}\norm{T^n} <\infty$. 
\begin{theorem}
\label{thm:ceapf}
With $h\in B$ the equation 
\begin{equation}
\label{eq:opeq}
Tf-f =h 
\end{equation}
has an a.p. solution $f$ if and only if the set of sums 
\[
s_n = \sum_{k=0}^nT^kh , \eqsp n\geq 0,
\]
is precompact. Under this condition all solutions are a.p.
\end{theorem}
\begin{proof}
The last sentence is true since the solutions of the corresponding homogeneous equation
are fixed vectors. Now let $f$ be a solution to \eqref{eq:opeq}. Then 
\begin{equation}
\label{eq:su}
s_n = T^{n+1}f - f. 
\end{equation}
Thus, the precompactness of $\{s_n\}$ is equivalent to the almost periodicity of the 
vector $f$. Therefore, it suffices to prove that if $\{s_n\}$ is precompact 
then the equation \eqref{eq:opeq} is solvable. To this end we consider 
the closure $K$ of the set $\{-s_n\}$. By assumption, $K$ is compact. 
Furthermore, $K$ is invariant for the affine mapping $V:B\rightarrow B$ defined as 
$Vf = Tf-h$. Indeed, let $f\in K$, 
and let for a given $\varepsilon >0$ a number $n$ is such that 
$\norm{f+ s_n}<\varepsilon/\norm{T}$. Then $\norm{Tf+ Ts_n}<\varepsilon$. 
However, 
\[
Ts_n = s_{n+1} - h. 
\] 
Therefore, $\norm{Tf-h + s_{n+1}}<\varepsilon$, i.e. 
$\norm{Vf+ s_{n+1}}<\varepsilon$. As a result, $Vf\in K$.   

Now we consider the closure $Q$ of the convex hull of $K$. This is a convex 
compact $V$-invariant set. By the Shauder fixed point theorem there exists 
$f\in Q$ such that $Vf = f$. This $f$ is a solution to \eqref{eq:opeq} 
by definition of $V$. 
\end{proof}
\begin{remark}
\label{rem:aeq}
{\em The precompactness of the set $\{s_n\}$ implies that $h$ is a.p.}. Indeed,
\[
T^nh= s_n-s_{n-1},\eqsp n\geq 1. 
\]
\qed
\end{remark}
\begin{cor}
\label{thm:6.0n}
Let $(X,F)$ be a metric compact dynamical system. Then with $\gamma\in C(X)$  
the c.e.~\eqref{eq:1} has a solution $\varphi\in AP(C(X),T_F)$ if and only  
if the subset  $\{s_n\}\subset C(X)$ is uniformly bounded and equicontinuous. 
In this case all continuous solutions are a.p.. 
\end{cor}
\begin{proof}
By the Arzela-Askoli theorem the subset $\{s_n\}$ is precompact in $C(X)$ if and 
only if it is bounded and equicontinuous. 
\end{proof}

Now let $X$ be any metric space with a distance $d$, and let 
$F$ be a mapping $X\rightarrow X$.
The dynamical system $(X,F)$ is called {\em uniformly stable} 
if for every $\varepsilon>0$ there exists $\delta>0$ such that
\begin{equation}
\label{eq:7.3a}
  d(x,y)<\delta \Longrightarrow d(F^nx,F^ny)<\varepsilon 
\end{equation}
for all $x,y\in X$, $n\in\N$. For $n=1$ this means that $F$ is uniformly  continuous.  

In fact, the uniform stability is topologically equivalent to a more elementary property. 
Namely, the dynamical system $(X,F)$ is called {\em dissipative} if 
\begin{equation}
\label{eq:7.3d}
  d(Fx,Fy)\leq d(x,y) \eqsp (x,y\in X),
\end{equation}
i.e. if $F$ is a contraction.
Obviously, every dissipative system is uniformly stable. In the converse direction we have 
\begin{prop}
\label{lem:7.3}
If a dynamical system $(X,F)$ is uniformly stable then it 
is dissipative with respect to a distance topologically equivalent to the original one.
\end{prop}
\begin{proof}
Let us introduce the distance   
  \begin{equation*}
    \tilde d(x,y) = \sup_{n\geq 0}\frac{d(F^n x,F^n y)}{1+d(F^nx,F^ny)}
  \end{equation*}
instead of the original distance $d$. Obviously, $\tilde d(Fx,Fy)\leq \tilde d(x,y)$. Since
  \begin{equation*}
    \tilde d(x,y)\geq \frac{d(x,y)}{1+d(x,y)},
  \end{equation*}
  the $\tilde d$-topology on $X$ is stronger that the $d$-topology. On
  the other hand, by~\eqref{eq:7.3a}
  \begin{equation*}
    d(x,y) < \delta \Rightarrow \tilde d(x,y)<\varepsilon, 
  \end{equation*}
  i.e. the $d$-topology is stronger than the $\tilde d$-topology.
\end{proof}
An important example of a dissipative system $(X,F)$ is a contraction  
$F:X\rightarrow X$ of a convex set $X$ in a normed space: 
\begin{equation}
\label{eq:lip}
\norm{Fx-Fy}\leq\norm{x-y} \eqsp (x,y\in X).   
\end{equation}
For \eqref{eq:lip} it is sufficient for  
$F$ to be Gato differentiable with $\norm{F'(x)}\leq 1$, $x\in X$. 
\begin{lemma}
\label{lem:ec}
The operator $T_F$ in $C(X)$ associated with an uniformly stable compact 
dynamical system $(X,F)$ is a.p.. 
\end{lemma}
\begin{proof}
It suffices to show that for every $f\in C(X)$
the set $\{T_F^nf\}$ is bounded and equicontinuous. The boundedness is 
obvious since $T_F$ is a contraction. On the other hand, the function $f$ is 
uniformly continuous since $X$ is compact. Let $\eta>0$, and let 
$\varepsilon>0$ be such that 
\[
d(x,y)<\varepsilon\Rightarrow\abs{f(x)-f(y)}<\eta
\] 
for all $x,y\in X$. A fortiori, 
\begin{equation}
\label{eq:lipl}
d(F^nx,F^ny)<\varepsilon\Rightarrow\abs{f(F^nx)-f(F^ny)}<\eta
\end{equation} 
for all $x,y\in X$ and $n\in\N$. Combining \eqref{eq:lipl} and \eqref{eq:7.3a} 
we get what we need. 
\end{proof}
In the following counterpart of GHT the minimality is not assumed.
\begin{theorem}
\label{thm:6.n}
Let $(X,F)$ be an uniformly stable compact dynamical system.
The c.e.~\eqref{eq:1} is solvable in $C(X)$ if and only if 
$\gamma\in C(X)$ and the subset $\{s_n\}\subset C(X)$ is bounded and equicontinuous. 
\end{theorem}
\begin{proof}
``If'' directly follows from Corollary \ref{thm:6.0n}. (The uniform stability
is not needed in this part.) 

``Only if''. Note that for any $\varphi\in C(X)$ the subset 
$\{T_F^n\varphi\}\subset C(X)$ is precompact
by Lemma \ref{lem:ec}. If, in addition, $\varphi$ satisfies \eqref{eq:1} 
then  formula \eqref{eq:2} shows that $\{s_n\}$ is precompact in $C(X)$. 
Hence, this subset is bounded and equicontinuous. 
\end{proof}

If $X$ is a convex compact in a Banach space 
then every continuous self-mapping of $X$ is not minimal 
because it has a fixed point. The GHT is not 
applicable to this case, while Theorem  \ref{thm:6.n} works immediately. 
\begin{theorem}
\label{cor:lipeq}
Let $X$ be a convex compact in a Banach space, and let $F:X\rightarrow X$ 
be a contraction. The c.e.~\eqref{eq:1} has a solution 
$\varphi\in C(X)$ if and only if $\gamma\in C(X)$ and  
the subset $\{s_n\}\subset C(X)$ is bounded and equicontinuous.
\end{theorem}
In the following example all conditions of Theorem \ref{cor:lipeq} 
except for the equicontinuity are fulfilled. Accordingly, there are no continuous 
solutions in this example. 
\begin{example}
\label{ex:ubwec}
Let $X=[0,\sqrt{2/3}]$, and let $Fx =x-x^3$, $x\in X$. This mapping $X\rightarrow X$ 
is a contraction since $\abs{F'(x)}\leq 1$, $x\in X$. Note that $Fx\neq 0$ for $x\neq 0$ 
but $F(0)=0$. We consider the c.e. \eqref{eq:1} in $C(X)$ with  
\begin{equation}
\label{eq:sin}
\gamma(x)=\sin\frac{1}{Fx}-\sin\frac{1}{x}\eqsp (x\neq 0),\eqsp \gamma(0)=0. 
\end{equation}
This function is indeed continuous since we have   
\[
\abs{\gamma(x)}\leq\abs{\sin\frac{1}{Fx}-\sin\frac{1}{x}}\leq
\abs{\frac{1}{Fx}-\frac{1}{x}} = \frac{x}{1-x^2}\rightarrow 0
\]
as $x\rightarrow 0$. Also, from \eqref{eq:sin} it follows that   
\[ 
s_n(x) = \sin\frac{1}{F^{n+1}x}-\sin\frac{1}{x}\eqsp (x\neq 0),\eqsp s_n(0)=0,
\]
hence, $\abs{s_n(x)}\leq 2$ for all $x\in X$. However, the sums $s_n(x)$ are not 
equicontinuous. Indeed, suppose to the contrary. Then there is $\delta >0$ such that 
\begin{equation*}
\abs{\sin\frac{1}{F^{n+1}x}-\sin\frac{1}{x}}=\abs{s_n(x)-s_n(0)}<\frac{1}{2},
\eqsp 0<x<\delta,\eqsp n\geq 0, 
\end{equation*} 
whence
\begin{equation}
\label{eq:sin0}
\abs{\sin\frac{1}{F^{n}x_0}}<\frac{1}{2}, \eqsp n\geq 1,  
\end{equation}
where $x_0 = 1/m\pi$ with an integer $m>1/\pi\delta$. 
The inequality \eqref{eq:sin0} yields 
\[
\dist(z_n,\pi\Z)<\pi/6 
\] 
where $z_n = 1/x_n$, $x_n= F^nx_0$. This is a contradiction. 
Indeed, $z_n$ monotonically tends to infinity, while 
\begin{equation*}
z_{n+1} - z_n = \frac{1}{x_{n+1}}- \frac{1}{x_n} = 
\frac{1}{Fx_n}- \frac{1}{x_n} = \frac{x_n}{1-x_n^2}\rightarrow 0. 
\end{equation*}
Hence, the sequence of fractional parts of $z_n/\pi$ is dense in [0/1].
\qed
\end{example}
\begin{remark}
\label{rem:gap}
In Example \ref{ex:ubwec} the function $\varphi(x)$ which equals $\sin(1/x)+1$ for $x\neq 0$ 
and $\varphi(0)=0$ is a bounded discontinuous solution such that $\varphi(x)=\sup_n[-s_n(x)]$ 
for all $x\in X$.

\qed
\end{remark}
If on a compact metric space $X$ a dynamical system $(X,F)$ is uniformly stable 
and minimal then for every $\gamma\in C(X)$ the uniform boundedness 
of the sums $s_n(x)$ implies their equicontinuity. This immediately follows from  
the GHT combined with Corollary \ref{thm:6.0n}. In particular, 
this is true for the sums 
\[
s_n(x)=\sum_{k=0}^nh(x+k\alpha), \eqsp 0\leq x<1,
\]
where $h$ is a continuous 1-periodic function and $\alpha$ is irrational.
For rational $\alpha$ this is trivial since in this case if the set
$\{s_n\}$ is bounded then it is finite, see Remark \ref{rem:3.3}. Let us emphasize
that the rational rotation of $\T$ is not minimal because of its periodicity. 
 
A dynamical system $(X,F)$ on a metric space $X$ with a distance $d$ 
is called \defin{conservative} if 
\begin{equation}
\label{eq:con}
d(Fx,Fy) = d(x,y) \eqsp (x,y\in X), 
\end{equation}
i.e. if $F$ is isometric.
For example, such is any rotation of $\T$.
\begin{theorem}
\label{thm:7.4}
  Any compact topologically transitive conservative dynamical system 
  $(X,F)$ is homeomorphic to $(G,\tau_g)$, where $G$ is a monothetic compact 
  group and $g$ is its generator.
\end{theorem}
This is a particular case of Theorem 6 from \cite{lyubich88}. Its proof in 
\cite{lyubich88} is based on the theory of almost periodic 
representations of topological semigroups \cite{mylyubich}, \cite{lyubichbook88}.
We use this result in Section \ref{sec:meso}. However, in this case the system is minimal. 
Below we briefly consider a group situation without the minimality. 

By the Birkhoff - Kakutani theorem (see e.g. \cite{hewittross63}, 
theorem 8.3) on any metric group $G$ there exists a 
topologically equivalent distance $d$ which is invariant in the sense that 
\begin{equation}
\label{eq:mei}
d(gx,gy) = d(x,y),\eqsp (x,y\in X, g\in G).
\end{equation}
With respect to this distance all shifts $\tau_g$ are conservative.  
Theorem \ref{thm:6.n} yields 
\begin{cor}
\label{thm:6.0nn}
Let $G$ be a compact metric group, and let $g\in G$. The c.e. 
\begin{equation}
\label{eq:ceg}
\varphi(gx) - \varphi(x)=\gamma (x), \eqsp x\in G,
\end{equation}
has a solution $\varphi\in C(G)$  if and only if $\gamma\in C(G)$ and 
the subset $\{s_n\}\subset C(G)$ is bounded and equicontinuous.
\end{cor}
Here the group $G$ can be noncommutative.

\section{Recurrence and ergodicity. Some applications}
\label{sec:reer} 

In this section $(X,F)$ is a dynamical system with an invariant measure $\mu$. 
We assume that $\mu$ is finite, though in some cases (in particular, in Theorem   
\ref{thm:4.4} without the equality \eqref{eq:4.5p}) the $\sigma$-finiteness is 
sufficient.  

The following is a modern form of the 
Poincar\'{e} Recurrence Theorem (RT). 
\begin{theorem}
  \label{thm:4.1-poincare}
Let $M$ be a measurable subset of $X$, $\mu(M)>0$. Then for a.e. 
$x\in M$ there exists a subsequence $(n_i)_{i=1}^\infty\subset\N$ 
such that all $F^{n_i}x\in M$.
\end{theorem}
\begin{cor}
  \label{cor:4.1p}
  There is a subsequence $(m_i)_{i=1}^\infty\subset\N$ such that
  $\mu(F^{-m_i}M\cap M)>0$ for all $i$.
\end{cor}

Actually, a more general theorem can be easily
derived from Theorem~\ref{thm:4.1-poincare}, see \cite{halmos56}. 
Below we give a ``cohomological'' proof of this generalization.

\begin{theorem}
  \label{thm:4.2}
  For every nonnegative measurable function $\gamma$ on $X$ the resolving series
  \eqref{eq:3} diverges a.e. on the set 
  $M_{\gamma}^{+} = \left\{x\setsp:\setsp \gamma(x)>0\right\}$.
\end{theorem}

Theorem~\ref{thm:4.1-poincare} 
is just that case of Theorem~\ref{thm:4.2} where $\gamma$ is the 
indicator function of the subset $M$.

\begin{proof}
  Denote by $A$ the set of convergence of the series~\eqref{eq:3}, and let
  $s(x)$ be its sum on $A$. The set $A$ and the function $s$ are measurable. 
  Furthermore, $A$ is completely invariant and $s$ satisfies the c.e.
  \begin{equation}
    \label{eq:4.2}
    s(x) - s(Fx) = \gamma(x), \eqsp x\in A.
  \end{equation}
  We have to prove that 
  \begin{equation}
\label{eq:4.22}
\mu(A\cap M_{\gamma}^{+}) =0 .
   \end{equation}

  Assume $s\in L_1(A,\mu)$. Then we have TNC  
  \begin{equation}
    \label{eq:4.3}
    \int_A \gamma\dmes\mu =0, 
  \end{equation}
  and \eqref{eq:4.22} follows since $\gamma\geq 0$ and $\gamma(x) > 0$ on $M_{\gamma}^{+}$. 

  With a small modification the same argument works in general. Note
  that one can assume $\gamma(x)\leq 1$ without loss of
  generality. Indeed, if the theorem is true for the function
  $\min(1,\gamma(x))$ then it is true for $\gamma(x)$.

  Now we consider the sequence of measurable sets 
  $A_n = \left\{ x\in A\setsp : \setsp s(x)\leq n\right\}$.
  They are invariant since $s(Fx)\leq s(x)$. On
  the other hand, $A_n\subset F^{-1}A_n\subset A_{n+1}$. (The first
  inclusion follows from $F A_n\subset A_n$, the second from
  $s(x)\leq s(Fx)+1$.) Furthermore, $\mu(F^{-1} A_n\setminus A_n) =0$
  since $\mu(F^{-1}A_n) = \mu(A_n)$.

Since $s|A_n$ is bounded and $\mu$ is finite, one can integrate \eqref{eq:4.2} over $A_n$:
  \begin{equation*}
    \int_{A_n}\gamma\dmes\mu  =  
    \int_{A_n}s(x)\dmes\mu - \int_{A_n}s(Fx)\dmes\mu  = 
    \int_{A_n}s(x)\dmes\mu - \int_{F^{-1}A_n}s(x)\dmes\mu = 0.
  \end{equation*}
Hence, $\mu(A_n\cap M_{\gamma}^{+})=0$.
Passing to the limit as $n\rightarrow\infty$ we get \eqref{eq:4.22}.
\end{proof}
By formula \eqref{eq:2} we obtain 
\begin{cor}
  \label{cor:4.3}
  Under conditions of Theorem \ref{thm:4.2}, if $\varphi$ is a solution 
to the c.e.~\eqref{eq:1} then 
  $\varphi(F^n x)\rightarrow+\infty$ a.e. on the set $M_{\gamma}^{+}$.
\end{cor}
How fast can be this growth? If $\gamma$ is bounded then 
$\varphi(F^n x) = O(n)$ by \eqref{eq:2}. This simple observation can be 
refined by the following Individual Ergodic Theorem (IET) established 
by Birkhoff \cite{birk31} and Khinchin \cite{khinchin32}. (See e.g.,
\cite{halmos56}, \cite{katok95}, \cite{sinaietalergodictheory} for the
proofs and numerous applications of the IET). 
\begin{theorem}
  \label{thm:4.4}
  For every $\gamma\in L_1(X,\mu)$ there exists
  \begin{equation}
    \label{eq:4.5}
    \tau_{\gamma}(x) = \lim_{n\rightarrow\infty}\frac{1}{n}\sum_{k=0}^{n-1}
    \gamma(F^k x)\eqsp \text{(a.e.)}.    
  \end{equation}
The function $\tau_{\gamma}$ belongs to $L_1(X,\mu)$, and 
  \begin{equation}
    \label{eq:4.5p}
    \int_X \tau_{\gamma}\dmes\mu = \int_X \gamma\dmes\mu.    
  \end{equation}
\end{theorem}
Note that by Proposition \ref{thm:cesreso} {\em the set $E_\gamma$ of convergence 
in \eqref{eq:4.5} is completely invariant, and the function $\tau_{\gamma}$ 
is invariant}. The latter is originally defined on $E_\gamma$ only, however, all its 
continuation to $X$ are invariant and coincide as the 
elements of $L_1(X,\mu)$ since $\mu(X\setminus E_\gamma) = 0$. 
\begin{cor}
\label{cor:fem}
If in the IET the measure $\mu$ is ergodic then 
\begin{equation}
  \label{eq:4.7p}
  \lim_{n\rightarrow\infty}\frac{1}{n} \sum_{k=0}^{n-1} \gamma(F^k x)
  = \frac{1}{\mu(X)} \int_X\gamma\dmes\mu \eqsp \text{(a.e.)}.
\end{equation}
\end{cor}
\begin{proof}
This follows from \eqref{eq:4.5p} since $\tau_{\gamma}(x)$ is a constant a.e. 
by ergodicity. 
\end{proof}
In physical terms formula \eqref{eq:4.7p} 
is a mathematically correct form of the famous Boltzmann
Hypothesis about the asymptotic behavior of a physical system
consisting of a large number of particles. The
equality~\eqref{eq:4.7p} states that for any admissible function $\gamma$
(an ``observable quantity'') its average over the space of states  
coincides with its average over time along almost every orbit. 
\begin{cor}
  \label{thm:4.3p}
If  $\gamma\in L_1(X,\mu)$ and $\varphi$ is a solution 
to the c.e.~\eqref{eq:1} then there exists
  \begin{equation}
    \label{eq:4.4}
    \varepsilon_{\varphi}(x) = 
    \lim_{n\rightarrow \infty}\frac{\varphi(F^n x)}{n}\eqsp \text{(a.e.)}
  \end{equation}
\end{cor}
Under conditions of Corollary \ref{thm:4.3p} the set of convergence 
in \eqref{eq:4.4} coincides with $E_\gamma$, 
and $\varepsilon_{\varphi}(x) = \tau_{\gamma}(x)$ for all $x\in E_\gamma$. 
This equality can be extended to $X$ by any common continuation of 
$\varepsilon_{\varphi}(x)$ and $\tau_{\gamma}(x)$. By \eqref{eq:4.5p} 
we get 
  \begin{equation}
    \label{eq:4.5q}
    \int_X \varepsilon_{\varphi}\dmes\mu = \int_X \gamma\dmes\mu.    
  \end{equation}
The function $\varepsilon_{\varphi}(x)$ is invariant. Therefore, 
if the measure $\mu$ is ergodic then  
  \begin{equation}
    \label{eq:4.5qq}
\lim_{n\rightarrow \infty}\frac{\varphi(F^n x)}{n}=    
 \frac{1}{\mu(X)} \int_X \gamma\dmes\mu \eqsp (a.e.)   
  \end{equation}
by \eqref{eq:4.5q} and \eqref{eq:4.4}. 
\begin{theorem}
  \label{thm:4.4p}
Let $\gamma\in L_1(X,\mu)$, and let $\varphi$ be a measurable solution to 
the c.e.~\eqref{eq:1} then $\varepsilon_{\varphi}(x)=0$ a.e., 
so $\varphi(F^n x) = o(n)$ a.e..
\end{theorem}
\begin{proof}
It suffices to prove that $\varepsilon_{\varphi}(x) = 0$ a.e. on the set 
$M_l = \left\{ x\in X\setsp :\setsp\abs{\varphi(x)}\leq l\right\}$  
for any $l>0$. The RT yields for a.e. $x\in M_l$ a subsequence 
$(n_i)_1^\infty\subset\N$ such that $F^{n_i}x\in M_l$, i.e.
$\abs{\varphi(F^{n_i}x)}\leq l$. If, in addition, $\varepsilon_{\varphi}(x)$ 
exists then
  \begin{equation*}
  \varepsilon_{\varphi}(x)  =
    \lim_{i\rightarrow\infty}\frac{\varphi(F^{n_i}x)}{n_i} = 0.
  \end{equation*}
\end{proof}
In the corollaries of Theorem \ref{thm:4.4p} we assume $\gamma\in L_1(X,\mu)$.
\begin{cor}
  \label{cor:4.4pp}
If the c.e.~\eqref{eq:1} has a measurable solution then $\gamma$ 
satisfies the TNC 
\begin{equation}
\label{eq:4.30}
    \int_ X\gamma\dmes\mu =0. 
\end{equation}
\end{cor}
This is the Anosov Theorem 1 from \cite{anosov73}. In our way to this 
important result we have used the same arguments (IET, RT) as in \cite{anosov73}
but in a wider context.
\begin{cor}
\label{cor:4.4pq}
Let the TNC \eqref{eq:4.30} be not 
fulfilled, but let the conditions \eqref{eq:2.1} be valid. Then the c.e.~\eqref{eq:1} 
is solvable but all its solutions are nonmeasurable.
\end{cor}
In particular, we have 
\begin{cor}
\label{cor:noper}
Let the TNC \eqref{eq:14} be not fulfilled, but  
let $F$ have no periodic points.  
Then the c.e.~\eqref{eq:1} is solvable but all its solutions are nonmeasurable. 
\end{cor}

For example, {\em with irrational $\alpha$ the equation \eqref{eq:7} is 
 solvable in 1-periodic functions but all its solutions are nonmeasurable if 
the TNC \eqref{eq:8} is not fulfilled.} 

The existence of measurable solutions was briefly discussed at the end 
of Section \ref{sec:sofu}. Also, in this section the $L_\infty$-solvability was 
considered (Theorem \ref{thm:limsup}). Now let us consider the $L_p$-solvability starting 
with $p=1$.
\begin{theorem} 
\label{thm:cesum}
With $\gamma\in L_1(X,\mu)$ the c.e. \eqref{eq:1} is $L_1$-solvable  
if and only if the resolving series \eqref{eq:3} is Ces\`aro summable a.e. 
to a function $\sigma(x)\in L_1(X,\mu)$. In this case $-\sigma(x)$ 
is an $L_1$-solution.
\end{theorem}
\begin{proof}
``If'' follows from Proposition \ref{thm:cesre} where now $\mu(X\setminus C_{\gamma})=0$. 
``Only if'' follows from Proposition \ref{thm:cesreso} where now 
$\mu(X\setminus E_{\varphi})=0$ by IET.
\end{proof}
\begin{cor} 
\label{thm:ceserg}
Let the measure $\mu$ be ergodic. With $\gamma\in L_p(X,\mu)$, $1\leq p\leq\infty$,  
the c.e. \eqref{eq:1} is $L_p$-solvable 
if and only if the resolving series \eqref{eq:3} is Ces\`aro summable a.e. 
to a function $\sigma(x)\in L_p(X,\mu)$. In this case $-\sigma(x)$ 
is an $L_p$-solution.
\end{cor}
\begin{proof}
``If'' is the same as in the proof of Theorem \ref{thm:cesum} even without 
assumptions about $\mu$.  
For the ``only if'' note that $L_p(X,\mu)\subset L_1(X,\mu)$ since 
the measure $\mu$ is finite. Let $\varphi$ be an $L_p$-solution to \eqref{eq:1}, 
and let $\sigma$ be that of Theorem \ref{thm:cesum}. Then $\sigma\in L_1(X,\mu)$  
and $\sigma(x)$ is the Ces\`aro sum of the series \eqref{eq:3} a.e. and, finally, 
$-\sigma(x)$ is a solution to \eqref{eq:1}.  
The sum $\sigma(x) + \varphi(x)$ satisfies the corresponding homogeneous equation.  
By ergodicity this is a constant a.e., thus $\sigma\in L_p(X,\mu)$. 
\end{proof}
Other results concerning the $L_p$-solvability were obtained in  
\cite{krz89} and \cite{alonso00}.

The IET states the convergence a.e. 
for the sequence $(S_{n;F}\gamma)(x)$, where $S_{n;F}$ are the operators
in $L_1(X,\mu)$ defined as 
\begin{equation*}
  S_{n;F} = \frac{1}{n} \sum_{k=0}^{n-1} T_F^k, \eqsp n\geq 1.  
\end{equation*}
The Koopman operator $T_F\varphi=\varphi\circ F$  
acts isometrically (not bijectively, in general) in 
every $L_p(X,\mu)$, $1\leq p<\infty$, since the measure $\mu$ is invariant.
In $L_{\infty}(X,\mu)$ and in $B(X)$ the operator $T_F$ is a contraction. 
A fortiori, $T_F$ is power bounded in all of these cases. Therefore, the asymptotic 
behavior of the operator sequence $(S_{n;F})$ can be studied by 
the following well known Mean Ergodic Theorem (MET). For completeness
we give a proof of that. 
\begin{theorem}
\label{thm:4.5}
Let $T$ be a bounded linear  operator in a Banach space $B$, and let  
\begin{equation}
\label{eq:4.6}
    S_n=\frac{1}{n} \sum_{k=0}^{n-1} T^k, \eqsp n\geq 1.   
  \end{equation}
Assume that $n^{-1}T^n$ strongly tends to zero as $n\rightarrow\infty$ 
and $\sup_n\norm{S_n}<\infty$. Then for a vector $h\in B$ the sequence 
$(S_nh)$ converges if and only if $h$ belongs to   
\begin{equation*}
  L = \overline{\im(T-I)} + \ker(T-I).  
\end{equation*}
The limit operator is a bounded 
projection from $L$ onto $\ker(T-I)$ annihilating $\overline{\im(T-I)}$. 
\end{theorem}
\begin{proof}
``If''. Consider two cases. 

1. $Th = h$. Then $S_n h=h$. Thus, on $\ker(T-I)$ the sequence $(S_n)$ 
strongly converges to $I$.

2. $h=Tf-f$, $f\in B$. Then
\begin{equation}
  \label{eq:4.8}
  S_n h = \frac{T^n f -f}{n}\rightarrow 0\eqsp (n\rightarrow\infty ).
\end{equation}
Thus, on $\im(T-I)$ the sequence $(S_n)$ strongly converges to 0. This result 
extends to the closure $\overline{\im(T-I)}$ since the sequence $(\norm{S_n})$ 
is bounded. On the whole of $L$ the limit operator is as required. 

``Only if''. Let $\lim_{n\rightarrow\infty}S_nh=g$. Then 
\begin{equation*}
(T-I)g= \lim_ {n\rightarrow\infty}\frac{T^n h -h}{n} =0, 
\end{equation*}
i.e. $g\in\ker(T-I)$. Now let $g'=h-g$. Then 
\begin{equation*}
g'= \lim_ {n\rightarrow\infty}\frac{1}{n}\sum_{k=0}^{n-1}(h-T^k h) 
= \lim_ {n\rightarrow\infty}\frac{I-T}{n}\sum_{k=0}^{n-1}\sum_{i=0}^{k-1}T^ih, 
\end{equation*}
i.e. $g'\in\overline{\im(T-I)}$. Hence, $h=g'+g\in L$.
\end{proof}
\begin{cor}
\label{cor:met1}
Under conditions of {\em MET} $L$ is a closed subspace and
the topological direct decomposition
\begin{equation}
  \label{eq:4.77}
  L = \overline{\im(T-I)}\oplus\ker(T-I)   
\end{equation}
holds.
\end{cor}
In general, $L\neq B$. The equality $L=B$ just  means that the sequence 
$(S_n)$ strongly converges on the whole space $B$. 
\begin{cor}
\label{cor:met2}
Let the space $B$ be reflexive, and let the operator $T$ be such that 
$n^{-1}T^n$ uniformly tends to zero as $n\rightarrow\infty$ 
and $\sup_n\norm{S_n}<\infty$. Then $L=B$. 
\end{cor}
\begin{proof}
The operator  $T^*$ in the space  $B^*$ satisfies 
the same conditions. Applying MET to $T^*$ we conclude that 
$\ker(T^*-I)\cap\overline{\im(T^*-I)}=0$.       
Since this subspace of $B^*$ is the annihilator of $L$, we get $L=B$.  
\end{proof}
\begin{cor}
\label{cor:met3}
If the space $B$ is reflexive and the operator $T$ is power bounded then $L=B$. 
\end{cor}
This classical case of MET is due to Lorch \cite{lorch}. For the unitary operator 
in Hilbert space this fact was discovered by 
von Neumann \cite{neumann32} who proved it via the spectral theorem. 
The following corollary of MET is due to Lin and Sine, c.f. \cite{linsine83}, Theorem 1. 
\begin{cor} 
\label{thm:cesoper}
Under conditions of Corollary \ref{cor:met2} the equation 
\begin{equation}
\label{eq:opec}
Tf-f=h
\end{equation}
in the space $B$ is solvable if and only if the resolving series 
\begin{equation*}
h+Th+\cdots+T^nh+\cdots 
\end{equation*}
is Ces\`aro summable in the norm-topology. If the Ces\`aro sum is $s$ then $f=-s$ 
is a solution to \eqref{eq:opec}.   
\end{cor}
The proof is similar to that of Theorem \ref{thm:cesum}.
In particular, we have the following counterpart of Corollary \ref{thm:ceserg}.
\begin{cor} 
\label{thm:cesergp}
With $\gamma\in L_p(X,\mu)$, $1<p<\infty$, the c.e. \eqref{eq:1} is solvable in $L_p(X,\mu)$ 
if and only if the resolving series \eqref{eq:3} is Ces\`aro summable in the $L_p$-norm  
to a function $\sigma(x)$. In this case $-\sigma(x)$ is an $L_p$-solution.
\end{cor}
\begin{remark}
\label{rem:pin}
For the ``if'' part of Corollary \ref{thm:cesoper} the reflexivity 
of $B$ is not required. Thus, the cases $p=1,\infty$ can be included into this part 
of Corollary \ref{thm:cesergp}. 
\qed
\end{remark}
Comparing the Corollaries \ref{thm:cesergp} and \ref{thm:cesergp} we obtain 
\begin{cor} 
\label{cor:pae}
Let the measure $\mu$ be ergodic . If the resolving series \eqref{eq:3} is 
Ces\`aro summable a.e. to a function $\sigma\in L_p(X,\mu)$, $1<p<\infty$, then 
it is Ces\`aro summable to  $\sigma$ in the $L_p$-norm. 
\end{cor}

Now recall that the equation \eqref{eq:opec} in a Banach space $B$ 
is called {\em normally solvable} if $\im(T-I)$ is closed or, equivalently, 
if with $h$ running over $B$ it is solvable if and only if 
$h$ is annihilated by all linear functionals from $\ker(T^*-I)$. 
For this property the following ergodic criterion 
is due to Dunford \cite{dunford43} in the ``if'' part and 
due to Lin \cite{lin74} in the ``only if'' part. The paper 
\cite{lin74} also contains a proof of the ``if'' part different 
from that of \cite{dunford43}. 

\begin{theorem}
  \label{thm:4.6}
  Let $T$ be a bounded linear operator in a Banach space $B$ such that 
  $n^{-1}T^n$ uniformly tends to zero as $n\rightarrow\infty$. Then
  $\im(T-I)$ is closed if and only if the sequence $(S_n)$ converges
  uniformly. Moreover, in this case $L=B$ and the limit operator 
  is the projection onto $\ker(T-I)$ annihilating $\im(T-I)$.
\end{theorem}

Being considered as a criterion of the uniform converegence of $(S_n)$ 
this result is called the Uniform Ergodic Theorem (UET).

\section{The total attractor}
\label{sec:prat}

This short section contains an useful information from the general theory of 
dynamical systems. Let $(X,F)$ be a compact dynamical system, and let 
\begin{equation}
\label{eq:5.1}
  \Omega = \bigcap_{n\geq 1} F^n X , 
\end{equation}
c.f. \eqref{eq:attr}. Obviously, $\Omega$ is an invariant nonempty compact 
subset of $X$, and $\Omega\neq X$, except for the case of surjective $F$. 
\begin{lemma}
  \label{lem:5.1}
  The mapping $F|\Omega$ is surjective.
\end{lemma}
\begin{proof}
  Let $F\Omega\neq\Omega$, and let $x\in\Omega\setminus
  F\Omega$. Then $F^{-1}\{x\}\cap\Omega=\emptyset$ because
  $\Omega$ is invariant. Therefore, 
  \begin{equation*}
    F^{-1}\{x\}\subset \bigcup_{n\geq 1}G_n,\eqsp G_n = X\setminus F^nX.
  \end{equation*}
  Since $G_n$ is open and $G_n\subset G_{n+1}$, the compact set
  $F^{-1}\{x\}$ is contained in some $G_m$. Thus, $F^{-1}\{x\}\cap
  F^m X = \emptyset$, i.e. $x\not\in F^{m+1} X$, a fortiori, $x\notin\Omega$, 
a contradiction.
\end{proof}
In addition, we have 
\begin{lemma}
  \label{lem:5.2}
Every invariant set $Y$ such that $F|Y$ is surjective is contained in  $\Omega$.
\end{lemma}
\begin{proof}
  $Y=F^n Y\subset F^n X$ for all $n$, so $Y\subset\Omega$.
\end{proof}
Thus, {\em  $\Omega$ is the largest invariant subset of $X$ with surjective
restriction of $F$}. 

\begin{lemma}
  \label{lem:5.3}
  For any neighborhood $G\supset \Omega$ there exists $n$ such that
  $F^n X\subset G$.
\end{lemma}
\begin{proof}
We have
 \begin{equation*}
    \bigcap_{n\geq 1}\left(F^n X\setminus G\right) = \Omega\setminus G
    = \emptyset .
\end{equation*}
Since all $F^n X\setminus G$ are compact, there is $n$ such
  that $F^n X\setminus G=\emptyset$.
\end{proof}
In view of Lemma \ref{lem:5.3} let us call $\Omega$ the {\em total 
attractor} for the dynamical system $(X,F)$. Actually, this construction 
is classical, coming back to Birkhoff \cite{birk27}.  
\begin{cor}
  \label{cor:5.11}
  The $\omega$-limit sets of all orbits are contained in $\Omega$.  
\end{cor}

Now let $\mu$ be a regular Borel measure on $X$. We denote by $\supp\mu$ its
\defin{support}, i.e. the set of $x\in X$ such that $\mu(U)>0$ for
every neighborhood $U$ of $x$. It is easy to see that the set $\supp\mu$ is
closed, hence it is compact. For every function $\psi\in L_1(X,\mu)$ we have 
  \begin{equation*}
  \int_X\psi(x)\dmes\mu = \int_{\supp\mu}\psi(x)\dmes\mu.  
   \end{equation*}
\begin{lemma}
  \label{lem:5.4}
  For any invariant regular Borel measure $\mu$ the set $M=\supp\mu$ is invariant
  and the restriction $F|M$ is surjective.
\end{lemma}
\begin{proof}
  Let $x\in M$, $y=Fx$. For any neighborhood $V$ of $y$ the preimage
  $F^{-1}V$ is a neighborhood of $x$ and $\mu(V)=\mu(F^{-1}V)>0$, 
  hence $y\in M$. Thus, $M$ is invariant. Now assume $FM\neq M$. 
  By the Uryson Lemma there
  exists a nonnegative continuous function $\phi$ such that $\phi|{FM}=0$,
  $\phi\neq 0$. However, 
  \begin{equation*}
  \int_X\phi(x)\dmes\mu = \int_X \phi(Fx)\dmes\mu 
  =  \int_M \phi(Fx)\dmes\mu  = 0, 
   \end{equation*}
a contradiction.
\end{proof}

Denote by $\Omega_0$ the union of supports of all invariant regular Borel 
measures. Combining Lemmas~\ref{lem:5.2} and~\ref{lem:5.4} we
obtain
\begin{cor}
  \label{cor:5.5}
  $\Omega_0\subset\Omega$.
\end{cor}

In general, $\Omega_0\neq \Omega$. For example, if $X$ is the closed
interval $[0,1]\subset\R$ and $Fx=x^2$, then $\Omega=X$, but
$\Omega_0$ is the two-point set $\{0,1\}$.

The concept of total attractor plays an important role in our theory \cite{lyubich88} of 
dissipative systems. In particular, we have 
\begin{theorem}
  \label{thm:7.2}
  If a compact dynamical system $(X,F)$ is dissipative then its restriction to
  the total attractor $\Omega$ is conservative.
\end{theorem}

A compact dynamical system $(X,F)$ with the total attractor $\Omega$ 
is called \defin{topologically $\Omega$-transitive} if the restriction 
$(\Omega,F|\Omega)$ is topologically transitive.
\begin{cor} 
  \label{thm:7.5}
  For any compact topologically $\Omega$-transitive dissipative dynamical system
  $(X,F)$ the invariant regular Borel measure $\mu$ such that $\mu(X)=1$ is unique.
\end{cor}
\begin{proof}
 By Corollary~\ref{cor:5.5} \mbox{$\supp\mu\subset\Omega$}. Therefore, 
the measure $\mu|\Omega$ is invariant for $F|\Omega$.
This restriction is conservative (Theorem~\ref{thm:7.2}).
Hence, the dynamical system $(\Omega,F|\Omega)$ 
is homeomorphic to $(G,\tau_{g})$ where $G$ is a monothetic compact 
group and $g$ is its generator (Theorem~\ref{thm:7.4}).
By Proposition \ref{prop:meha} the image of the measure $\mu|\Omega$ is 
the Haar measure $\nu$ on $G$. 
\end{proof}

\section{The normal solvability in $C(X)$}
\label{sec:noso}

The following result was obtained in \cite{belitskii98}. 
\begin{theorem}
  \label{thm:6.1}
  Let $(X,F)$ be a compact dynamical system. The c.e.~\eqref{eq:1} is
  normally solvable in $C(X)$ if and only if $F$ is preperiodic.
\end{theorem}
Thus, as long as $F$ is not preperiodic, there is a continuous
function $\gamma$ such that the TNC \eqref{eq:14} is fulfilled 
for all invariant regular Borel measures $\mu$, but the c.e.~\eqref{eq:1} has no
continuous solutions. 

\begin{proof}
  The ``if'' part follows directly from Theorem~\ref{thm:preper}.
  Conversely, assume that the c.e.~\eqref{eq:1} is normally solvable in $C(X)$,
  i.e. $\im(T-I)$ is closed for $T=T_F$. Since $T$ is a contraction, the UET 
  (i.e., Theorem~\ref{thm:4.6}) yields
  \begin{equation*}
    \lim_{n\rightarrow\infty}\norm{\frac{1}{n} \sum_{k=0}^{n-1}T^k - P} =0
  \end{equation*}
  where $P$ is the projection onto $\ker(T-I)$ annihilating
  $\im(T-I)$. Hence, for every $\varepsilon>0$ there exists
  $n=n(\varepsilon)$ such that 

  \begin{equation}
    \label{eq:6.2}
    \abs{\frac{1}{n} \sum_{k=0}^{n-1}
      \gamma(F^k x)  - (P\gamma)(x)} < \varepsilon\norm{\gamma}, 
   \eqsp n\geq n(\varepsilon),\eqsp x\in X, 
  \end{equation}
for all $\gamma\in C(X)$, $\gamma\neq 0$.

  Let $\Omega$ be the total attractor of our system, 
  and let $\gamma|\Omega=0$. Then $\gamma(F^k x) = 0$ for $x\in\Omega$
  and for all $k$. Therefore,
  $\abs{(P\gamma)(x)}<\varepsilon\norm{\gamma}$, $x\in\Omega$, thus 
$P\gamma|\Omega=0$. Since $P=PT^m$ for every $m$, we have 
$(P\gamma)(x)=(P\gamma)(F^mx)$ for all $x\in X$. Hence, $(P\gamma)(x)= 0$
by Corollary \ref{cor:5.11}, and \eqref{eq:6.2} reduces to 

  \begin{equation*}
    \abs{\frac{1}{n} \sum_{k=0}^{n-1} \gamma(F^k x)} <
    \varepsilon\norm{\gamma}, \eqsp n\geq n(\varepsilon),\eqsp x\in X. 
  \end{equation*}
For $\varepsilon=1$ and $l=n(1)-1$ we get
  \begin{equation}
    \label{eq:6.3}
    \abs{\frac{1}{l+1} \sum_{k=0}^{l} \gamma(F^k x)} < \norm{\gamma}, \eqsp x\in X.
  \end{equation}

  The inequality \eqref{eq:6.3} implies $F^l X=\Omega$. Indeed, let $x\in X$ be such that
  $F^l x\not\in\Omega$, a fortiori, $F^k x\not\in\Omega$ for $0\leq
  k\leq l$. By the Uryson Lemma there exists $\gamma\in C(X)$ such that
  $\gamma|\Omega=0$, $\gamma(F^k x)=1$, $0\leq k\leq l$, and
  $\norm{\gamma}=1$. This contradicts~\eqref{eq:6.3}.

  It remains to prove that $F|\Omega$ is periodic.
  First, we show that the c.e.~\eqref{eq:1} remains normally solvable after 
  restriction to $C(\Omega)$. Let $\gamma\in C(\Omega)$ 
  be annihilated by all invariant regular Borel 
  measures on $\Omega$. Denote by $\tilde\gamma$ a continuous extension 
  of $\gamma$ to $X$, and let $\mu$ be an invariant regular Borel 
  measure on $X$. By Corollary~\ref{cor:5.5} the measure $\mu|\Omega$ is 
 invariant for $F|\Omega$ and 
  \begin{equation*}
    \int_X \tilde\gamma\dmes\mu = \int_\Omega \gamma\dmes(\mu|\Omega) = 0. 
  \end{equation*}
By assumption, the equation $\varphi(Fx)-\varphi(x) = \tilde\gamma(x)$ on $X$ has 
a continuous solution. Its restriction to $\Omega$ satisfies \eqref{eq:1}.

  Now one can assume $\Omega=X$. 
  Then $F$ is surjective by Lemma~\ref{lem:5.1}. Accordingly, the operator $T$ is
  isometric. Its spectrum $\sigma(T)$ contains the point $\lambda =1$  
  (as an eigenvalue with the eigenfunction $\id$). This 
  point is isolated in $\sigma(T)$ since $\im(I-T)$ is closed, see \cite{taylorlay80},
  p.330. Hence, $T$ is invertible, otherwise $\sigma(T)$ would be the unit disk 
  $\left\{ x\in\C\setsp :\setsp \abs{\lambda}\leq 1\right\}$, see 
\cite{lyubich92-funcan}, p.185. On the other
  hand, $T$ is an endomorphism of the Banach algebra $C(X)$. Therefore, $T$
  is an automorphism. By the Kamowitz-Scheinberg theorem \cite{kamsche} 
  either $T$ is periodic or $\sigma(T)$ is the unit circle 
  $\left\{\lambda\in\C\setsp :\setsp\abs{\lambda}=1\right\}$. 
  Thus, $T$ is periodic, so the underlying mapping $F$ is periodic as well. 
\end{proof}

Theorem~\ref{thm:6.1} can be extended to a wide class of noncompact
spaces $X$ by the Czech-Stone compactification $\overline{X}$. 
\begin{cor}
  \label{cor:6.2}
  Let $X$ be a completely regular topological space. Then
  c.e.~\eqref{eq:1} is normally solvable in $CB(X)$ if and only if $F$
  is preperiodic.
\end{cor}
\begin{proof}
The space $X$ can be considered as a dense subset in
$\overline{X}$ and then $F$ is the restriction to $X$ of a continuous
mapping $\overline{F}:\overline{X}\rightarrow\overline{X}$. Since the
Banach spaces $CB(X)$ and $C(\overline{X})$ are isometric, the normal
solvability of~\eqref{eq:1} in $CB(X)$ is equivalent to the same
property of the c.e. $\psi(\overline{F}z)-\psi(z) = \theta(z)$ in
$C(\overline{X})$. On the other hand, the preperiodicity of $F$ is
equivalent to that of $\overline{F}$.
\end{proof}
\begin{cor}
  \label{cor:6.3}
  For any set $X$ the c.e.~\eqref{eq:1} is normally solvable in $B(X)$
  if and only if $F$ is preperiodic.
\end{cor}
\begin{cor}
\label{cor:66}
Let $(X,F)$ be a minimal topological dynamical system with infinite 
completely regular $X$. Then the c.e. \eqref{eq:1} is not normally solvable in $CB(X)$. 
\end{cor}
\begin{proof}
The only minimal preperiodic system is the rotation of a cycle. 
\end{proof}
\begin{cor} 
\label{cor:cbb}
Under conditions of Corollary \ref{cor:66} the norms $\norm{.}_F$ and 
$\norm{.}$ on $\im(T_F-I)$ are not equivalent.
\end{cor}
\begin{proof}
Combine Corollary \ref{cor:66} with Proposition \ref{prop:cloq}. 
\end{proof}

The applications below relate to the shifts of topological semigroups and groups. 
For any semigroup $S$ we denote by $[[s]]$ the subsemigroup of $S$ generated by $s$, 
i.e. $[[s]]=\{s^n:n\geq 1\}$. We do not assume that $S$ has a unit. If $S$ is a group 
and $e$ is its unit then $[[s]]\cup\{e\}=[s]$ according to the notation from Section 
\ref{sec:togr}.  
\begin{cor}
  \label{cor:6.4}
  Let $S$ be a completely regular topological semigroup. With $s\in S$ the c.e. 
 \begin{equation}
 \label{eq:seeq} 
\varphi(sx)-\varphi(x)=\gamma(x),\eqsp x\in S, 
 \end{equation}
is normally solvable in $CB(S)$ if and only if $[[s]]$ is finite.
\end{cor}
\begin{proof}
If $[[s]]$ is finite then $s^{p+l}=s^p$ for some $p\geq 1$ and $l\geq 0$.  
Then the shift $x\mapsto sx$, $x\in S$, is preperiodic with the period $p$ 
and preperiod $l$. Conversely, if $s^{p+l}x=s^px$ for all $x\in S$ then 
$s^{p+l+1}=s^{p+1}$, so $[[s]]$ is finite.
\end{proof}
\begin{cor}
  \label{cor:6.5}
  Let $G$ be a topological group. With $g\in G$ the c.e. 
 \begin{equation}
 \label{eq:geeq}
\varphi(gx)-\varphi(x)=\gamma(x), \eqsp x\in G, 
 \end{equation}
is normally solvable in $CB(G)$ if and only if the element $g$ is of a finite order.
\end{cor}
\begin{proof}
Every topological group is completely regular, see e.g. \cite{hewittross63}, Theorem 8.4.  
\end{proof}
For example, {\em the c.e. corresponding to a rotation of the unit circle $\T$ 
is normally solvable in $C(\T)$ if and only if the rotation is rational}. 

For the almost periodic solutions we use the Bohr compactification. 
\begin{cor}
  \label{cor:6.7}
If $G$ is a commutative topological group such that its points are separated 
by characters (in particular, if $G$ is locally compact) then with $g\in G$ the c.e. 
 \begin{equation}
 \label{eq:geeqap} 
\varphi(gx)-\varphi(x)=\gamma(x), \eqsp x\in G, 
 \end{equation}
is normally solvable in $AP(G)$ if and only if the element $g$ is of a finite order.
\end{cor}
\begin{proof}
The normal solvability of \eqref{eq:geeqap} in $AP(G)$ is equivalent to the same 
property of the corresponding c.e. in $C(bG)$ where $bG$ is the Bohr 
compact of the group $G$.
Since the characters separate points of $G$, the canonical 
homomorphism $G\rightarrow bG$ is injective. Hence, $g$ is of a finite order 
if and only if such is its image in $bG$.
\end{proof}

\begin{cor}
\label{cor:6.77}
With any real $\alpha\neq 0$ the c.e.  
 \begin{equation}
 \label{eq:rap} 
\varphi(x+\alpha)-\varphi(x) = \gamma(x), \eqsp x\in\R, 
 \end{equation}
is not normally solvable in $AP(\R)$. 
\end{cor}
\begin{proof}
There are no nonzero elements of finite order in $\R$.
\end{proof}

We conclude this section with the question: {\em is the $L_p$-counterpart 
of Theorem \ref{thm:6.1} true?}

\section{Absence  of measurable solutions}
\label{sec:meso}

Let us start with a group situation.
\begin{theorem}
\label{thm:7.1}
Let $G$ be an infinite monothetic compact group, and let $g$ be 
its generator. Then there exists a function $\gamma\in C(G)$  
satisfying the TNC 
    \begin{equation}
    \label{eq:7.nu}
\int_G\gamma\dmes\nu =0
  \end{equation}
for the Haar measure $\nu$ but such that the c.e.
    \begin{equation}
    \label{eq:7.2}
    \varphi(gx)-\varphi(x) = \gamma(x), \eqsp x\in G, 
  \end{equation}
  has no measurable solutions.
\end{theorem}

Let us stress that by Proposition \ref{prop:meha} the only TNC is \eqref{eq:7.nu}.
The absence of continuous solutions with a $\gamma\in C(G)$ satisfying
\eqref{eq:7.nu} follows from Corollary \ref{cor:6.5}. 

In the proof of Theorem \ref{thm:7.1} 
we use the ergodicity of the dynamical system $(G,\tau_g)$. 
This property is stated by the following lemma.
\begin{lemma}
\label{lem:7.0}
Under conditions of Theorem \ref{thm:7.1} every measurable solution $\varphi$ to 
the homogeneous equation  
 \begin{equation}
    \label{eq:hom}
    \varphi(gx)-\varphi(x) = 0, \eqsp x\in G, 
  \end{equation} 
is a constant a.e.  
\end{lemma}
\begin{proof}
One can assume $\varphi\neq 0$. If the measure of the set 
$\{x\in G: \varphi(x)\neq 0\}$ is not zero then the measure of the subset 
$M_l= \{x\in G: 0<\abs{\varphi(x)}\leq l\}$ is not zero for an $l>0.$  
By \eqref{eq:hom} $M_l$ is completely invariant for the shift $\tau_g$.
We prove that $\varphi(x)$ is a constant a.e. on $M_l$ and $\mu(G\setminus M_l)=0$.

For any character $\chi$ of the group $G$ we have 
 \begin{equation*}
(\chi(g)- 1)\int_{M_l}\varphi(x)\overline{\chi(x)}\dmes\nu = 0, 
\end{equation*}
see \eqref{eq:cef}. If $\chi\neq\id$ then $\chi(g)\neq 1$ since $g$ is a generator 
of $G$. Hence, 
 \begin{equation}
\label{eq:pwo}
\int_G\omega_l(x)\varphi(x)\overline{\chi(x)}\dmes\nu = 0,  
\end{equation}
where $\omega_l$ is the indicator of $M_l$. The function 
$\theta_l(x)=\omega_l(x)\varphi(x)$ is measurable and bounded, a fortiori,  
$\theta_l\in L_2(G,\nu)$. The relation \eqref{eq:pwo} states that all its Fourier 
coefficients corresponding to the nonunity character vanish. Hence, 
$\theta_l(x)$ is a constant $\eta$ a.e. on $G$. However, $\theta_l(x)=\varphi(x)$ 
on $M_l$. Therefore, $\varphi(x)=\eta$ 
a.e. on $M_l$. On the other hand, $\theta_l(x)= 0$ on $G\setminus M_l$. 
If the measure of this set is not zero then $\eta=0$ and then $\varphi(x)=0$ a.e. on $M_l$,  
a contradiction. 
\end{proof}

Another lemma we need is the following.
\begin{lemma}
\label{lem:7.00}
Under conditions of Theorem \ref{thm:7.1} the set $\Gamma=\{\chi(g):\chi\in G^*\}$ 
is dense in $\T$. 
\end{lemma}
\begin{proof}
The set $\Gamma$ is the image of the mapping $\chi\mapsto\chi(g)$, $\chi\in G^*$. 
Since the latter is a homomorphism $G^*\rightarrow\T$, the set  
$\Gamma$ is a subgroup of $\T$. It is well known (and can be easily proven) 
that every nondense subgroup of $\T$ is finite. However, if $\Gamma$ is finite then all 
its elements are roots of 1 of a power $m$. Then $\chi(g^m)=\chi(g)^m=1$ for all $\chi\in G^*$
that results in  $g^m=e$ according to the duality theory. Since $g$ is a generator of $G$, 
the latter turns out to be finite that contradicts our assumption. 
\end{proof}

Now we proceed to the proof of Theorem~\ref{thm:7.1}, c.f. \cite{belitskii98}. 
\begin{proof}
  Suppose to the contrary that the equation ~\eqref{eq:7.2} 
  has a measurable solution for every $\gamma$ satisfying~\eqref{eq:7.nu}. 
  This TNC determines a closed subspace $C^0(G)\subset C(G)$. On the other hand, 
  we consider the space $M=M(G,\nu)$ of measurable functions on $G$. As 
usual, we identify the functions equal a.e. Then 
the formula  
  \begin{equation*}
    \rho(\varphi_1,\varphi_2) =
    \int_G\frac{\abs{\varphi_1-\varphi_2}}{1+
      \abs{\varphi_1-\varphi_2}} \dmes \nu
  \end{equation*}
determines a metric on $M$  corresponding to the convergence in measure:
  \begin{equation*}
    \lim_{n}\rho(\varphi_n,\varphi) =0 \Longleftrightarrow
    \forall\varepsilon>0 : \lim_{n} \nu
    \left\{x\,:\,\abs{\varphi_n(x)-\varphi(x)}\geq\varepsilon \right\} =0.
  \end{equation*}
This space is not Banach but it is 
an $F$-space. (See e.g. \cite{rudin} for    
information on this class of linear  metric spaces.) In what follows we deal with 
the quotient space $\widehat M = M/Q$ where $Q$ is the subspace of constants.
This $\widehat M$  is also an $F$-space.

  For a given $\gamma\in C^0(G)$ a measurable solution $\varphi$ is unique up
  to an additive constant by Lemma \ref{lem:7.0}. Therefore, its coset 
  $\widehat\varphi\in\widehat M$
  is uniquely determined, i.e. $\widehat\varphi = R\gamma$ where $R$ is a mapping 
  $C^0(G)\rightarrow \widehat M$. Obviously, $R$ is linear. Show that the graph of $R$ 
  is closed. 

Let a sequence $(\gamma_n)\subset C^0(G)$ converge to a $\gamma\in C^0(G)$
  uniformly, and let $(R\gamma_n) = (\widehat{\varphi_n})$ converge to a 
  $\hat\varphi\in\widehat M$. By adding of suitable constants
  the functions $\varphi_n$ can be chosen so that $\varphi_n$ 
  tends to $\varphi$ in measure. By restriction to a subsequence one can
  assume $\lim_n\varphi_n(x) = \varphi(x)$ a.e. on $G$. 
  The limit pair $(\gamma, \varphi)$ satisfies~\eqref{eq:7.2} a.e., thus  
  $R\gamma=\widehat\varphi$.

  By the Closed Graph Theorem the operator $R$ is continuous. To
  disprove this conclusion we apply Lemma \ref{lem:7.00} to get 
  a sequence $(\chi_n)\subset G^*\setminus\{\id\}$ such that 
 $\lim_{n}\chi_n(g)=1$. With 
 \begin{equation*}
\gamma_n(x) = (\chi_n(g)-1)\chi_n(x)
 \end{equation*}
  the functions $\chi_n$ satisfy~\eqref{eq:7.2} for $\gamma=\gamma_n$. 
By \eqref{eq:trin} we have 
\begin{equation*}
    \int_G \gamma_n\dmes\nu = (\chi_n(g)-1)\int_G \chi_n\dmes\nu =0   
  \end{equation*}
i.e.$\gamma_n\in C^0(G)$. Since $\lim_{n}\gamma_n=0$ in $C^0(G)$ and $R$ is continuous, 
the sequence $\widehat\chi_n=R\gamma_n$ tends to zero in $\widehat M$. This means 
that there are constants $c_n$ such that $(\chi_n - c_n)$ 
  tends to zero in measure. Therefore, the sequence $(c_n)$ is bounded and
  \begin{equation*}
     \lim_{n} \int_G (\chi_n -c_n)\dmes\nu =0, 
  \end{equation*}
whence $\lim_n c_n=0$, and then $\lim_n\chi_n=0$ in measure. This contradicts the equality
  \begin{equation*}
    \int_G\abs{\chi_n}\dmes\nu = \int_G\dmes\nu =1. 
  \end{equation*}
\end{proof}
\begin{cor}
\label{cor:cgt}
In $C^0(G)$ the set of $\gamma$ mentioned in Theorem \ref{thm:7.1} has 
the second Baire's category.
\qed
\end{cor}
\begin{proof}
This set is the complement to $\im(\tau_g-I:M\rightarrow C^0(G))\neq C^0(G)$.
This image is the set of first category by the classical Banach theorem. 
\end{proof}
Now we are in position to obtain the following general result.
\begin{theorem}
\label{thm:7.6}
Let a compact dynamical system $(X,F)$ be uniformly stable and topologically 
$\Omega$-transitive. If the total attractor $\Omega$ is infinite, then there 
is an invariant regular Borel measure $\mu$ on $X$ and a function $\gamma\in C(X)$ 
such that the TNC \eqref{eq:14} is fulfilled but the c.e.~\eqref{eq:1} has 
no measurable solutions.
\end{theorem}
\begin{proof}
One can assume that $(X,F)$ is dissipative. 
Then  $(\Omega,F|\Omega)$ is conservative by Theorem \ref{thm:7.2}, 
and it is topologically transitive by definition. 
By Theorem \ref{thm:7.4} the system $(\Omega,F|\Omega)$ is homeomorphic 
to $(G,\tau_g)$ where $G$ is a monothetic compact group and $g$ is a generator. 
The group is infinite since such is $\Omega$. Thus, Theorem \ref{thm:7.1} 
is applicable. Let $\gamma_0$ be the function appearing there, and let 
$\gamma_1$ be the corresponding function from $C(\Omega)$, and finally, 
let $\mu_1$ be the measure on $\Omega$ corresponding to the Haar measure $\nu$   
on $G$. Then 
\begin{equation}
\label{eq:7.000}
\int_{\Omega} \gamma_1\dmes\mu_1 =0
\end{equation}
but the equation 
\begin{equation}
    \label{eq:7.00}
    \phi(Fx)-\phi(x) = \gamma_1(x), \eqsp x\in\Omega, 
  \end{equation}
has no measurable solutions. To finish the proof it suffices to trivially extend 
$\mu_1$ from $\Omega$ to $X$ and $\gamma_1$ to 
a $\gamma\in C(X)$.   
\end{proof}
\begin{remark}
By Corollaries \ref{cor:5.5} and  \ref{thm:7.5} the measure $\mu$ is unique 
up to normalization. Therefore, in the context of Theorem \ref{thm:7.6} there 
are no TNC's essentially different from \eqref{eq:14}. 
\qed
\end{remark}
In \cite{lyubich80} the Closed Graph Theorem has been applied to 
prove the following Gordon theorem \cite{gordon75} 
the original proof of which is purely analytic. In the proof below we 
follow \cite{lyubich80}. Actually, this way is also a prototype for the 
proof of Theorem \ref{thm:7.1}. 
\begin{theorem}
\label{thm:gl}
Let $\theta(t)$ $(0<t\leq 1)$ be a positive nondecreasing function such that 
$t^{-1}\theta(t)\rightarrow\infty$ as $t\rightarrow 0$. Let $\alpha$ be an irrational number. 
Then there exists a continuous 1-periodic function $h(x)$, $x\in\R$, satisfying 
the TNC \eqref{eq:8} and such that  
\begin{equation}
\label{eq:lipt} 
\norm{h}_{\theta}\equiv
\sup_{0\leq y<x\leq 1}\frac{\abs{h(x)-h(y)}}{\theta(x-y)}<\infty,
\end{equation}
but the c.e. \eqref{eq:7} has no measurable solutions.    
\end{theorem}
\begin{proof}
Let us transfer the situation to the group $\T$ and consider the space 
$\widehat M$ introduced in the proof of Theorem \ref{thm:7.1}. 
On the other hand, the condition \eqref{eq:lipt} together with \eqref{eq:8} 
determines a Banach space $C_{\theta}^0(\T)$ of functions on $\T$. This is a 
linear nonclosed subspace of $C(\T)$ containing all characters 
\[
\chi_n(x) = e^{2\pi inx},\eqsp x\in\R, \eqsp n\in\Z\setminus\{0\}. 
\]
Indeed, 
\begin{equation}
\label{eq:ncn}
\norm{\chi_n}_{\theta} = \sup_{0<t\leq 1}\frac{2\abs{\sin\pi nt}}{\theta(t)}
\leq 2\pi\abs{n}\sup_{0<t\leq 1}\frac{t}{\theta(t)}<\infty.  
\end{equation}
Futher we use $n>0$ only. 

First of all, we show that  
\begin{equation}
\label{eq:ncnn}
\lim_{n\rightarrow\infty}\frac{\norm{\chi_n}_{\theta}}{n} = 0.
\end{equation}
To this end note that for every $\tau\in (0,1]$ we have 
\[
\frac{2\abs{\sin\pi nt}}{\theta(t)}\leq\frac{2}{\theta(\tau)},\eqsp t\geq\tau,
\]
since $\theta(t)$ is nondecreasing. On the other hand, 
\[
\frac{2\abs{\sin\pi nt}}{\theta(t)}\leq 2\pi n\sup_{0<s<\tau}\frac{s}{\theta(s)},
\eqsp 0<t<\tau. 
\]
By \eqref{eq:ncn}
\[ 
\frac{\norm{\chi_n}_{\theta}}{n}\leq 
\max\{\frac{2}{n\theta(\tau)},2\pi\sup_{0<s<\tau}\frac{s}{\theta(s)}\}
\rightarrow 0
\]
as $n\rightarrow\infty$ and $\tau = 1/n$.

The character $\chi_n$ is a solution to the the c.e. \eqref{eq:7} 
with 
\[
h(x)= h_n(x)= (e^{2\pi in\alpha} - 1)\chi_n(x).
\]
Obviously, 
\[
\norm{h_n}_{\theta}= 2\abs{\sin{\pi n\alpha}}\norm{\chi_n}_{\theta}
\leq2\pi d_n\norm{\chi_n}_{\theta} 
\]
where $d_n=\dist(n\alpha,\Z)$. For irrational $\alpha$ the classical Dirichlet 
theorem states that  $d_{n_k}\leq 1/n_k$ 
for a sequence $(n_k)\subset\N$. By \eqref{eq:ncnn} the sequence 
$(h_{n_k})$ converges to 0 in $C_{\theta}^0(\T)$.

Now let us assume that with an irrational $\alpha$ the c.e. \eqref{eq:7} 
has a measurable solution $f$ for every $h\in C_{\theta}^0(\T)$.
Then the Closed Graph Theorem is applicable as in the proof of Theorem 
\ref{thm:7.1}. As a result, the sequence $(\chi_{n_k})$ converges to 0 
in measure, that is impossible.
\end{proof}
\begin{remark}
\label{rem:redun}
In \cite{gordon75} the continuity modulus $\theta$ is supposed to be subadditive, 
i.e. $\theta(t_1+t_2)\leq\theta(t_1)+\theta(t_2)$. We have seen that this condition 
is redundant. 
\qed
\end{remark}
A nice particular case of Theorem \ref{thm:gl} is  
$\theta(t)=t^{\lambda}, 0\leq\lambda<1$, that means  
the H\"older condition for $h$. The case $\lambda = 0$ , i.e. $\theta =\id$, 
is just that of Theorem \ref{thm:7.1} for the irrational rotation of  $\T$. 
On the other hand, Theorem \ref{thm:gl} cannot be extended to the case 
$\lambda = 1$, i.e.  $\theta(t) =t$, that means the Lipshitz condition for $h$. 
\begin{theorem}
\label{thm:h1}
Let $\alpha$ be irrational such that in the corresponding continued fraction 
the set of elements is bounded. Then for every 1-periodic function $h$ satisfying 
the Lipshitz condition  and the TNC \eqref{eq:8} the c.e. \eqref{eq:7} has 
a 1-periodic solution $f\in L_2(0,1)$.
\end{theorem}
\begin{proof}
In the Fourier decomposition 
\begin{equation*}
h(x)=\sum_{n=-\infty}^{\infty} h_n e^{2\pi inx}
\end{equation*}
we have $h_0=0$ by \eqref{eq:8}. Consider the formal solution 
\begin{equation}
\label{eq:fs}
\sum_{n\neq 0}\frac{h_n}{e^{2\pi in\alpha}-1}e^{2\pi inx}.
\end{equation}
By our assumption on $\alpha$ there exists $c>0$ such that 
$\dist(n\alpha,\Z)>c/\abs{n}$, $n\neq 0$, see e.g. \cite{khinchin97}, Theorem 23. 
Therefore, 
\[
\abs{e^{2\pi in\alpha}-1} = 2\abs{\sin\pi n\alpha}>\frac{4c}{\abs{n}}
\]
since $\sin\pi t>2t$ for $0<t<1/2$ and $\dist(n\alpha,\Z)<1/2$. 
On the other hand, $h_n=-h_n'/n$
where $h_n'$ are the Fourier coefficients of the derivative $h'(x)$. 
The latter exists a.e. and bounded, a fortiori, $h'\in L_2(0,1)$. 
As a result, 
\begin{equation*}
\sum_{n\neq 0}\abs{\frac{h_n}{e^{2\pi in\alpha}-1}}^2< 
\frac{1}{16c^2}\sum_{n\neq 0}\abs{h_n'}^2<\infty, 
\end{equation*}
hence the series \eqref{eq:fs} converges in $L_2(0,1)$.
\end{proof}

In particular, all quadratic irrationalities satisfy the condition of 
Theorem \ref{thm:h1} since in this case the continued fractions are 
periodic. The simplest example is the Fibonacci irrationality 
$\alpha = \frac{1+\sqrt{5}}{2}$ where 
all elements of the corresponding continued fraction are equal to 1.

\section{Summation of  divergent  series}
\label{sec:sods} 

The material of this section is basically extracted from our paper \cite{lyubich08-axit}. 
This starts with the following general definition of summation of numerical series 
which is a modern form of the bringing together Hardy-Kolmogorov 
axioms \cite{hardy49}, \cite{kolmogorov25}.

Let $\bf s$ be the linear space of all scalar sequences $\boldsymbol{\xi}
= \left(\xi_n\right)$, and let $\tau$ be the shift operator in $\bf s$,
i.e. $\tau\boldsymbol{\xi}=(\xi_{n+1})$, and, finally, let
$L\subset \bf s$ be a $\tau$-invariant subspace. A linear functional
$\sigma$ on $L$ is called a \defin{summation of the series}
\begin{equation}
\label{eq:ser}
\boldsymbol{\xi} =  \xi_0 + \xi_1  + \cdots + \xi_n +\cdots, \eqsp \boldsymbol{\xi}\in L,
\end{equation}
if
\begin{equation}
  \label{eq:7.7}
  \sigma[\boldsymbol{\xi}] - \sigma[\tau\boldsymbol{\xi}] = \xi_0
\end{equation}
accordingly to the formal relation
\begin{equation*}
  \xi_0 + \xi_1 + \cdots +\xi_n + \cdots = \xi_0 + ( \xi_1 + \cdots + \xi_n + \cdots ).
\end{equation*}
In this situation we say that the subspace $L$ \defin{admits summation}.  
Also, we say that a series $\boldsymbol{\xi}$ is
{\em summable} if it belongs to a subspace admitting summation or,
equivalently, if the subspace $\Span\{\tau^k\boldsymbol{\xi}: k\geq 0\}$ admits summation.
If this summation is $\sigma$ we say that $\boldsymbol{\xi}$ is $\sigma$-{\em summable}.
Let us emphasize that the equation \eqref{eq:7.7} is a c.e. with given $\xi_0$ and unknown 
$\sigma$, both are linear functionals on $L$. Most of classical summations (for instance, the  
Ces\'aro summation) satisfy \eqref{eq:7.7}.  
\begin{lemma}
\label{lem:lis}
Let $(X,F)$ be a dynamical system. If there exists a summation $\sigma$ such that 
the resolving series $\boldsymbol{\gamma}_x$
of a function $\gamma\in\Phi(X)$ is $\sigma$-summable for all $x\in X$ then the c.e. 
\eqref{eq:1} is solvable. A solution is $\varphi(x)=-\sigma[\boldsymbol{\gamma}_x]$. 
\end{lemma}
\begin{proof}
The summation $\sigma$ is defined on a $T$-invariant subspace $L\in\bf s$ containing all 
$\boldsymbol{\gamma}_x$, $x\in X$. Since $\boldsymbol{\gamma}_{Fx} =
\tau\boldsymbol{\gamma}_x$, 
we have
\begin{equation*}
\varphi(Fx)-\varphi(x)= \sigma[\boldsymbol{\gamma}_x]-\sigma[\boldsymbol{\gamma}_Fx]=
\sigma[\boldsymbol{\gamma}_x ]-\sigma[\tau\boldsymbol{\gamma}_x]= 
(\boldsymbol{\gamma}_x)_0 = \gamma(x)
\end{equation*}
\end{proof}

The nonlinear functionals satisfying \eqref{eq:7.7} are also interesting 
because of their relation to the resolving functionals introduced 
in Section \ref{sec:sofu}. However, the resolving functional 
$\omega$ from Theorem \ref{thm:2.20} is defined on a set $\Lambda$ of sequences  
$(\eta_n)$ containing all $(-s_n(x))$, $x\in X$, while the summation $\sigma$ 
from \eqref{eq:7.7} is defined on a space $L$ of sequences whose coordinates $\xi_n$ 
are the members of the series, not the partial sums. Nevertheless, these situations are 
connected by means of the linear operator $V:\bf s\mapsto\bf s$ such that 
$V(\xi_n)= (\eta_n)$ where 
$\eta_n=-\sum_{k=0}^n\xi_k$, $n\geq 0$. Obviously, the operator $V$ is invertible: 
$V^{-1}(\eta_n)= (\xi_n)$ where $\xi_n=\eta_{n-1}-\eta_n$, $n\geq 1$, $\xi_0=-\eta_0$.
Let $L=V^{-1}\Lambda$ , and let $\sigma(\boldsymbol{\xi})=\omega(V\boldsymbol{\xi})$, 
$\xi\in L$. In general, the set $L$ and the functional $\sigma$ are nonlinear since 
such are $\Lambda$ and $\omega$. It is easy to prove the following
\begin{prop}
\label{prop:nonli}
Under conditions of Theorem \ref{thm:2.20} the set $L$ is $\tau$-invariant and the 
functional $\sigma$ satisfies the equation \eqref{eq:7.7}. Every such 
a {\em nonlinear summation} appears in this way.  
\end{prop}

In what follows all summations are linear. We start with a crucial example of a 
nonsummable series.
\begin{example}
  The series
  \begin{equation}
    \label{eq:7.8}
\boldsymbol{\varepsilon} =  1+1+\cdots +1+\cdots
  \end{equation}
  is not summable. Indeed, the substitution $\boldsymbol{\xi}=\boldsymbol{\varepsilon}$ 
  into~\eqref{eq:7.7} yields the classical contradiction $0=1$.
\end{example}

Thus, {\em if a subspace $L$ admits summation then the
series $\boldsymbol{\varepsilon}$  does not belong to $L$}. Remarkably, the converse
is also true.

\begin{theorem}
  \label{thm:7.8}
  If a $\tau$-invariant subspace $L$ does not contain the
  series  $\boldsymbol{\varepsilon}$ then $L$ admits summation.
\end{theorem}
\begin{proof}
The series $\boldsymbol{\varepsilon}$ is a fixed point of the operator $\tau$, i.e. 
it belongs to the subspace $\ker (I-\tau)$. The latter 
consists of the series of form $\xi +\xi \cdots +\xi +\cdots$, so 
$\ker (I-\tau)=\Span\{\boldsymbol{\varepsilon}\}$. Hence, under condition of Theorem 
\ref{thm:7.8} we have $L\cap\ker (I-\tau) = 0$, i.e. the operator $R=(I-\tau)|L$ is injective. 
Therefore, $R$ maps $L$ onto $\im R$ bijectively. We denote the corresponding 
inverse operator by $S$. Since the subspace $L$ is $\tau$-invariant, we have 
$\im R\subset L$. Let $Q:L\rightarrow L$ be a projection onto  $\im R$. 
Then the linear functional $\sigma(\boldsymbol{\xi}) = \xi_0(SQ\boldsymbol{\xi})$, 
$\boldsymbol{\xi}\in L$, satisfies \eqref{eq:7.7}, i.e. this is a summation on $L$. 
Indeed, 
 \begin{equation*}
\sigma(\boldsymbol{\xi}) - \sigma(\tau\boldsymbol{\xi}) = \sigma(R\boldsymbol{\xi}) = 
 \xi_0(SQR\boldsymbol{\xi}) = \xi_0(SR\boldsymbol{\xi})= \xi_0(\boldsymbol{\xi}). 
 \end{equation*}
\end{proof}

An important application of this criterion is the following. 
\begin{theorem}
  \label{thm:7.9}
  Let $(X,F)$ be a dynamical system with a finite ergodic invariant measure $\mu$,  
and let a function $\gamma\in L_1(X,\mu)$ 
  satisfy the TNC \eqref{eq:14}. Then there exists a summation $\sigma$ 
  such that the resolving series $\boldsymbol{\gamma}_x$ is $\sigma$-summable a.e.. 
\end{theorem}
\begin{proof}
  By Corollary \ref{cor:fem} of the IET we have 
  \begin{equation}
    \label{eq:7.10p}
    \lim_{n\rightarrow\infty}\frac{1}{n}\sum_{k=0}^{n-1}\gamma(F^k x) = 0 
   \end{equation}
  for $x\in X_0$, where $X_0$ is a subset of $X$ such that $\mu(X\setminus X_0) =0$.
 Since $X_0$ is $F$-invariant, the  space $L =\Span\{\boldsymbol{\gamma}_x:x\in X_0\}$ 
is $\tau$-invariant. We prove that there exists a required summation $\sigma$ on $L$. 

 Suppose to the contrary. Then by  
  Theorem~\ref{thm:7.8} the series $\boldsymbol{\varepsilon}$ belongs to $L$, i.e.  
  \begin{equation*}
    \sum_{i=1}^m \alpha_i\boldsymbol{\gamma}_{x_i}  = \boldsymbol{\varepsilon}
  \end{equation*}  
 with some $x_i\in X_0$ and some coefficients $\alpha_i$. In other words, 
   \begin{equation*}
    \sum_{i=1}^m \alpha_i \gamma(F^k x_i) = 1,\eqsp k\geq 0, 
  \end{equation*}
whence 
  \begin{equation*}
    \sum_{i=1}^m \alpha_i \left(\frac{1}{n}\sum_{k=0}^{n-1}\gamma(F^k x_i)\right) 
    = 1,\eqsp n\geq 1, 
  \end{equation*}
Passing to the limit as $n\rightarrow\infty$ we get the contradiction 0=1  
by \eqref{eq:7.10p}.
\end{proof}
By Lemma \ref{lem:lis} we obtain 
\begin{cor}
\label{cor:solu}
Under conditions of Theorem \ref{thm:7.9} the c.e. \eqref{eq:1} has a solution a.e.. 
The solution is $\varphi(x)=-\sigma(\boldsymbol{\gamma}_x)$.
\end{cor}

In particular, this yields a solution a.e. to the c.e. 
\begin{equation}
\label{eq:koz}
f(qx) - f(x) = h(x), \eqsp x\in\R, \eqsp q\in\N,\eqsp q\geq 2, 
\end{equation}   
in $2\pi$-periodic functions with $h\in L_1(0,2\pi)$. 
Indeed, this equation is equivalent to  
  \begin{equation}
  \label{eq:pw}
   \varphi(z^q)-\varphi(z)=\gamma(z),\eqsp z\in\T,  
  \end{equation}  
with $z=e^{ix}$. The mapping $F_qz=z^q$ is ergodic with respect 
to the standard measure on $\T$.  
By Corollary \ref{cor:solu}  the c.e. \eqref{eq:koz} with 
$h\in L_1(0,2\pi)$ such that 
  \begin{equation}
\label{eq:adco}
\int_0^{2\pi}h(x)\dmes x =0
  \end{equation}  
is solvable a.e.. In fact, the solvability follows from Example 
\ref{ex:pow} even without the assumption \eqref{eq:adco}. However, Corollary \ref{cor:solu} 
yields a solution a.e. as a result of a summation of the resolving series 
\begin{equation}
\label{eq:cozy}
  h(x) + h(qx) +\cdots + h(q^nx) + \cdots. 
\end{equation} 
The existence of such a summation for $q=3$ and $h(x)=\sin x$ was conjectured 
by Kolmogorov \cite{kolmogorov25}. 

Note that, in general, the summability of the resolving series of a c.e. is not necessary 
for the solvability. For instance, the resolving series 
of the Abel equation \eqref{eq:ae} is $\boldsymbol{\varepsilon}$ but 
this equation can be solvable, see Corollary \ref{cor:10}. 
\begin{theorem}
  \label{thm:nonme}
 Let h(x) be a trigonometric polynomial, 
\begin{equation*}
h(x)=\sum_{k=1}^m(a_k\cos\nu_kx+b_k\sin\nu_kx),
\end{equation*}
with $0<\nu_1<...<\nu_m$ such that all ratios $\nu_k/\nu_j$, $k\neq j$, 
are not powers of $q$. Then all $2\pi$-periodic solutions to the c.e. \eqref{eq:lyu}   
are nonmeasurable.
\end{theorem}
This is the part 3) of our Theorem 4.8 from \cite{lyubich08-axit}. 
For $q=3$ and $h(x)=\sin x$ this was announced by Kolmogorov \cite{kolmogorov25}. 
For $q=2$ and $h(x)=\cos x$ this was  proven by Zygmund 
(\cite{zygmund59}, Chapter 5, Problem 26). 
\begin{cor}
\label{cor:fin}
Let $\gamma(z)$ be a polynomial, 
\begin{equation*}
\gamma(z)=\sum_{k=1}^m c_kz^{\nu_ k},
\end{equation*}
with $\nu_k$ same as in Theorem \ref{thm:nonme}. Then all solutions to the c.e. \eqref{eq:pw}   
are nonmeasurable.
\end{cor}
\begin{remark}
\label{rem:final}
The self-mapping $F_qz=z^q$ of $\T$ is not minimal since $z=1$ is the fixed point.
Also, this is not uniformly stable. Indeed, by ergodicity of $F_q$  
the sequence $(F_q^nz)$ is dense in $\T$ for a.e. $z$, while $F_q^n1=1$ for all $n$.      
\qed
\end{remark}

\bibliography{cohomeq}

\def\cprime{$'$}
\providecommand{\bysame}{\leavevmode\hbox to3em{\hrulefill}\thinspace}
\providecommand{\MR}{\relax\ifhmode\unskip\space\fi MR }
\providecommand{\MRhref}[2]{%
  \href{http://www.ams.org/mathscinet-getitem?mr=#1}{#2}
}
\providecommand{\href}[2]{#2}
\begin{thebibliography}{10}

\bibitem{abel}
N.~H. Abel, \emph{De\'termination d'une fonction au moyen d'une \`equation qui
  ne contient qu'une seule variable}, Ouevres compl\`etes \textbf{2} (1881).

\bibitem{alonso00}
A.~Alonso, J.~Hong, and R.~Obaya, \emph{Absolutely continuous dynamics and real
  coboundary cocycles in {$L^p$}-spaces, {$0<p<\infty$}}, Studia Math.
  \textbf{138} (2000), no.~2, 121--134.

\bibitem{anosov73}
D.~V. Anosov, \emph{On an additive functional homological equation connected
  with an ergodic rotation on the circle}, Math. USSR - Izv. \textbf{7} (1973),
  1257--1271.

\bibitem{arnold}
V.~I. Arnol{\cprime}d, S.~M. Guse{\u\i}n-Zade, and A.~N. Varchenko,
  \emph{Singularities of differentiable maps. {V}ol. {I}}, Monographs in Math.,
  82, Birkh\"auser Verlag, 1985.

\bibitem{belitskii98a}
G.~Belitskii and Yu. Lyubich, \emph{The {A}bel equation and total solvability
  of linear functional equations}, Studia Math. \textbf{127} (1998), no.~1,
  81--97.

\bibitem{belitskii99}
\bysame, \emph{On the normal solvability of cohomological equations on locally
  compact topological spaces}, Nonlinear analysis and related problems, Trudy
  Inst. Mat. Akad. Nauk Belarusi, vol.~2, 1999, in Russian, pp.~44--51.

\bibitem{belitskii03}
G.~Belitskii and V.~Tkachenko, \emph{One-dimensional functional equations},
  Oper. Th.: Adv. and Appl., vol. 144, Birkh\"auser Verlag, 2003.

\bibitem{belitskii98}
G.~R. Belitskii and Yu.~I. Lyubich, \emph{On the normal solvability of
  cohomological equations on compact topological spaces}, Oper. Th.: Adv. and
  Appl. \textbf{103} (1998), 75--87.

\bibitem{birk27}
G.~D. Birkhoff, \emph{Dynamical systems}, Colloquium publications, vol.~9,
  Amer. Math. Soc., Providence, RI, 1927.

\bibitem{birk31}
\bysame, \emph{Proof of the ergodic theorem}, Proc. Nat. Acad. Sci. U.S.A.
  \textbf{17} (1931), 656--660.

\bibitem{bochner27}
S.~Bochner, \emph{Beitrage zur {T}heorie der {F}astperiodischen {F}unktionen},
  Math. Ann. \textbf{96} (1927), 119 -- 147.

\bibitem{bohr25}
H.~Bohr, \emph{Zur {T}heorie der {F}ast {P}eriodischen {F}unktionen}, Acta
  Math. \textbf{45} (1925), no.~1, 29--127.

\bibitem{bohr}
\bysame, \emph{Almost periodic functions}, Chelsea Publ. Co., 1947.

\bibitem{browder58}
F.E. Browder, \emph{On the iteration of transformations in noncompact minimal
  dynamical systems}, Proc. Amer. Math. Soc. \textbf{9} (1958), 773--780.

\bibitem{dunford43}
N.~Dunford, \emph{Spectral theory. {I}. {C}onvergence to projections}, Trans.
  Amer. Math. Soc. \textbf{54} (1943), 185--217.

\bibitem{furst60}
H.~Furstenberg, \emph{Stationary processes and prediction theory}, Ann. of
  Math. Studies, no. 44, Princeton Univ. Press, 1960.

\bibitem{furst61}
\bysame, \emph{Strict ergodicity and transformation of the torus},
  Amer.J.Math., \textbf{83} (1961), 573--601.

\bibitem{gelbaum64}
B.~Gelbaum and J.~Olmsted, \emph{Conterexamples in analysis}, Holden-Day, 1964.

\bibitem{gordon75}
A.Ya. Gordon, \emph{Sufficient condition for unsolvability of the additive
  functional homological equation connected with the ergodic rotation of a
  circle}, Funct. Anal. Appl. \textbf{9} (1975), no.~4, 334--336.

\bibitem{gottschalk46}
W.~H. Gottschalk, \emph{Almost periodic points with respect to transformation
  semi-groups}, Ann. Math. \textbf{47} (1946), 762--766.

\bibitem{gottschalk55}
W.~H. Gottschalk and G.~A. Hedlund, \emph{Topological dynamics}, AMS Coll.
  Publ., vol.~36, 1955.

\bibitem{haar33}
A.~Haar, \emph{Der {M}assbegriff in der {T}heorie der continuerlichen
  {G}ruppen}, Ann. Math. \textbf{34} (1933), 147--169.

\bibitem{halmos42}
P.~Halmos and Samelson H., \emph{On monothetic groups}, Proc. of Nat. Acad. of
  Sci. of the USA \textbf{28} (1942), 254--258.

\bibitem{halmos50}
P.~R. Halmos, \emph{Measure theory}, Van Nostrand Co., 1950.

\bibitem{halmos56}
\bysame, \emph{Lectures on ergodic theory}, Japan Math. Soc., 1956.

\bibitem{hardy49}
G.~H. Hardy, \emph{Divergent series}, Oxford Univ. Press, 1949.

\bibitem{helson85}
H.~Helson, \emph{Note on additive cocycles}, J. London Math. Soc. (2)
  \textbf{31} (1985), no.~3, 473--477.

\bibitem{hewittross63}
E.~Hewitt and K.~A. Ross, \emph{Abstract harmonic analysis. {V}ol. {I}}, Die
  Grundlehren der mathematischen Wissenschaften, Bd. 115, Academic Press Inc.,
  New York, 1963.

\bibitem{kamsche}
H.~Kamowitz and S.~Scheinberg, \emph{The spectrum of automorphisms of {B}anach
  algebras}, J. Func. Anal. \textbf{4} (1969), 268--276.

\bibitem{katok01}
A.~Katok, \emph{Cocycles, cohomology and combinatorial constructions in ergodic
  theory}, Smooth ergodic theory and its applications, Proc. Sympos. Pure
  Math., vol.~69, AMS, 2001, In collaboration with E. A. Robinson, Jr.,
  pp.~107--173.

\bibitem{katok95}
A.~Katok and B.~Hasselblatt, \emph{Introduction to the modern theory of
  dynamical systems}, Encycl. of Math. and Appl., vol.~54, Cambridge Univ.
  Press, 1995.

\bibitem{khinchin32}
A.~Ja. Khinchin, \emph{Zur {B}irckhoffs {L}\"osung des {E}rgodenproblem}, Math.
  Ann. (1932), 485 -- 488.

\bibitem{khinchin97}
\bysame, \emph{Continued fractions}, Dover Publ., 1997.

\bibitem{kirillov67}
A.~A. Kirillov, \emph{Dynamical systems, factors and group representations},
  Russian Math. Surveys \textbf{22} (1967), no.~5 (137), 63--75.

\bibitem{kolmogorov25}
A.~N. Kolmogorov, \emph{Sur la possibilit\'e de la d\'efinition g\'en\'erale de
  la d\'eriv\'ee, de l'int\'egrale et de la sommation de s\'eries divergentes},
  C. R. Acad. Sci. \textbf{180} (1925), 362--364.

\bibitem{kolmogorov53}
\bysame, \emph{On dynamical systems with an integral invariant on the torus},
  Doklady Akad. Nauk SSSR (N.S.) \textbf{93} (1953), 763--766.

\bibitem{kornfeld76}
I.~Kornfeld, \emph{On the additive homological equation}, Funct. Anal. Appl.
  \textbf{10} (1976), no.~2, 73--74.

\bibitem{sinaietalergodictheory}
I.~P. Kornfeld, S.~V. Fomin, and Ya.~G. Sina{\u\i}, \emph{Ergodic theory},
  Fundamental Principles of Mathematical Sciences, vol. 245, Springer-Verlag,
  1982.

\bibitem{krygin74}
A.~Krygin, \emph{An example of a continuous flow on the torus with a mixed
  spectrum}, Mat. Zametki \textbf{15} (1974), 235--240.

\bibitem{krybog37}
N.~Kryloff and N.~Bogoliouboff, \emph{La th\'eorie g\'en\'erale de la mesure
  dans son application \`a l'\'etude des syst\`emes dynamiques de la
  m\'ecanique non lin\'eaire}, Ann. of Math. (2) \textbf{38} (1937), no.~1,
  65--113.

\bibitem{krz89}
K.~Krzy{\.z}ewski, \emph{On regularity of measurable solutions of a cohomology
  equation}, Bull. Polish Acad. Sci. Math. \textbf{37} (1989), 279--287.

\bibitem{krz00}
\bysame, \emph{A note on a generalized cohomology equation}, Colloq. Math.
  \textbf{84/85} (2000), no.~2, 279--283.

\bibitem{kuczma90}
M.~Kuczma, B.~Choczewski, and R.~Ger, \emph{Iterative functional equations},
  Cambridge Univ. Press, 1963.

\bibitem{lin74}
M.~Lin, \emph{On the uniform ergodic theorem}, Proc. Amer. Math. Soc.
  \textbf{43} (1974), 337--340.

\bibitem{linsine83}
M.~Lin and R.~Sine, \emph{Ergodic theory and the functional equation
  {$(I-T)x=y$}}, J. Operator Theory \textbf{10} (1983), no.~1, 153--166.

\bibitem{livsh72}
A.~N. Liv{\v{s}}ic, \emph{Cohomology of dynamical systems}, Izv. Akad. Nauk
  SSSR Ser. Mat. \textbf{36} (1972), 1296--1320.

\bibitem{lorch}
E.~R. Lorch, \emph{Means of iterated transformations in reflexive vector
  space}, Bull. Amer. Math. Soc. \textbf{45} (1939), 945--957.

\bibitem{mylyubich}
M.~Yu. Lyubich and Yu.~I. Lyubich, \emph{The spectral theory of almost periodic
  representations of semigroup}, Ukrain. Math. J. \textbf{36} (1984), no.~5,
  474--478.

\bibitem{lyubich80}
Yu.~I. Lyubich, \emph{Method of closed graph for the additive homological
  equation on the circle}, Theory of Functions of Several Real Variables,
  Yaroslavl State Univ., 1980, in Russian, pp.~123--125.

\bibitem{lyubich88}
\bysame, \emph{Dissipative actions and almost periodic representations of
  abelian semigroups}, Ukrain. Mat. J. \textbf{40} (1988), no.~1, 70--74, 134.

\bibitem{lyubichbook88}
\bysame, \emph{Introduction to the theory of {B}anach representations of
  groups}, Oper. Th.: Adv. and Appl., vol.~30, Birkh\"auser Verlag, 1988.

\bibitem{lyubich92-funcan}
\bysame, \emph{Linear functional analysis}, Functional Analysis (N.~K.
  Nikolskij, ed.), Encycl. of Math. Sci., vol.~19, Springer-Verlag, 1992.

\bibitem{lyubich08-axit}
\bysame, \emph{Axiomatic theory of divergent series and cohomological
  equations}, Fund. Math. \textbf{198} (2008), no.~3, 263--282.

\bibitem{mccutcheon99}
R.~McCutcheon, \emph{The {G}ottschalk-{H}edlund theorem}, Amer. Math. Monthly
  \textbf{106} (1999), no.~7, 670--672.

\bibitem{moore80}
C.~C. Moore and K.~Schmidt, \emph{Coboundaries and homomorphisms for
  nonsingular actions and a problem of {H}. {H}elson}, Proc. London Math. Soc.
  (3) \textbf{40} (1980), no.~3, 443--475.

\bibitem{rozh08}
A.~V. Rozhdestvenski{\u\i}, \emph{On nontrivial additive cocycles on a torus},
  Sb. Math \textbf{199} (2008), no.~1--2, 229--251.

\bibitem{rudin}
W.~Rudin, \emph{Functional analysis}, McGraw-Hill Book Co., 1973.

\bibitem{sator03}
R.~Sato, \emph{On solvability of the cohomology equation in function spaces},
  Studia Math. \textbf{156} (2003), no.~3, 277--293.

\bibitem{sato03}
\bysame, \emph{A remark on real coboundary cocycles in {$L^\infty$}-space},
  Proc. Amer. Math. Soc. \textbf{131} (2003), no.~1, 231--233.

\bibitem{schmidt77}
K.~Schmidt, \emph{Cocycles on ergodic transformation groups}, Macmillan
  Lectures in Math., vol.~1, Macmillan, 1977.

\bibitem{schwartz96}
P.~Schwartz, \emph{A cocycle theorem with an application to {R}osenthal sets},
  Proc. Amer. Math. Soc. \textbf{124} (1996), no.~12, 3689--3698.

\bibitem{szekeres10}
G.~Szekeres, \emph{Abel's functional equation and its role in the problem of
  croissance r\'eguliere}, Cambridge Books Online, Cambridge Univ. Press, 2010.

\bibitem{taylorlay80}
A.~E. Taylor and D.~C. Lay, \emph{Introduction to functional analysis}, John
  Wiley \& Sons, 1980.

\bibitem{neumann32}
J.~von Neumann, \emph{Proof of the quasi-ergodic hypothesis}, Proc. Nat. Acad.
  Sci. USA \textbf{18} (1932), 70--82.

\bibitem{neumann}
\bysame, \emph{Zur {O}peratorenmethode in der {K}lassischen {M}echanik}, Ann.
  of Math. (2) \textbf{33} (1932), no.~3, 587--642.

\bibitem{weil40}
A.~Weil, \emph{L'int\'egration dans les groupes topologiques et ses
  applications}, Hermann, 1940.

\bibitem{wintner}
A.~Wintner, \emph{The linear difference equation of first order for angular
  variables}, Duke Math. J. \textbf{12} (1945), 445--449.

\bibitem{zygmund59}
A.~Zygmund, \emph{Trigonometric series}, vol.~1, Cambridge Univ. Press, 1959.

\end{thebibliography}

\vskip 1cm
  Address:

  {\it Department of Mathematics, 

    Technion, 32000, 

    Haifa, Israel}

  \smallskip
  email: {\it lyubich@tx.technion.ac.il}  

\end{document}